\documentclass[12pt]{amsart}
\usepackage{xypic}
\usepackage{amssymb}
\usepackage{amsthm}
\usepackage{bm}
\usepackage{latexsym}
\usepackage{amsmath}
\usepackage{eufrak}
\usepackage{mathrsfs}
\usepackage{amscd}
\usepackage[all]{xy}
\usepackage[usenames]{color}
\usepackage{graphicx}
\usepackage{epstopdf}
\usepackage{color}
\usepackage{amscd}
\theoremstyle{plain}

\newtheorem{theorem}{Theorem}[section]
\newtheorem{proposition}[theorem]{Proposition}
\newtheorem{lemma}[theorem]{Lemma}
\newtheorem{corollary}[theorem]{Corollary}

\newtheorem{question}[theorem]{Question}

\theoremstyle{definition}

\newcommand{\appsection}[1]{\let\oldthesection\thesection
\renewcommand{\thesection}{Appendix \oldthesection}
\section{#1}\let\thesection\oldthesection}

\newtheorem{definition}[theorem]{Definition}
\newtheorem{notation}[theorem]{Notation}

\theoremstyle{remark}

\newtheorem{remark}[theorem]{Remark}
\newtheorem{example}[theorem]{Example}

\def\D{{\mathbb{D}}}

\def\Z{{\mathbb{Z}}}
\def\Q{{\mathbb{Q}}}
\def\C{{\mathbb{C}}}
\def\P{{\mathbb{P}}}

\def\aaa{{\bold a}}
\def\O{{\mathcal{O}}}

\def\X{{\mathcal{X}}}
\def\Y{{\mathcal{Y}}}

\DeclareMathOperator{\QG}{QG}

\DeclareMathOperator{\Cl}{Cl}

\DeclareMathOperator{\Spec}{Spec}
\DeclareMathOperator{\Proj}{Proj}
\DeclareMathOperator{\Def}{Def}

\DeclareMathOperator{\Hom}{Hom}
\DeclareMathOperator{\cHom}{\mathcal{H}\mathnormal{om}}

\DeclareMathOperator{\wt}{wt}

\DeclareMathOperator{\SL}{SL}

\newcommand{\QED}{\ifhmode\unskip\nobreak\fi\quad {\rm Q.E.D.}} 

\newcommand{\bA}{\mathbb A}
\newcommand{\bB}{\mathbb B}
\newcommand{\bC}{\mathbb C}
\newcommand{\bE}{\mathbb E}

\newcommand{\bG}{\mathbb G}

\newcommand{\bN}{\mathbb N}
\newcommand{\bP}{\mathbb P}
\newcommand{\bQ}{\mathbb Q}
\newcommand{\bR}{\mathbb R}

\newcommand{\bU}{\mathbb U}
\newcommand{\bX}{\mathbb X}
\newcommand{\bY}{\mathbb Y}
\newcommand{\bZ}{\mathbb Z}

\newcommand{\cX}{\mathcal X}

\newcommand{\cI}{\mathcal I}
\newcommand{\cO}{\mathcal O}

\newcommand{\cM}{\mathcal M}
\newcommand{\cN}{\mathcal N}

\begin{document}
\bibliographystyle{amsplain}
\title[Construcci\'on]{Flipping surfaces}
\author{\textrm{Paul Hacking, Jenia Tevelev and Giancarlo Urz\'ua}}


\maketitle

\begin{abstract}
We study semistable extremal $3$-fold neighborhoods, which are fundamental building blocks of birational geometry,
following earlier work of Mori, Koll\'ar, and Prokhorov. We classify possible flips and extend Mori's algorithm for computing flips of extremal neighborhoods of type $k2A$ to more general $k1A$ neighborhoods.
The novelty of our approach is to show that $k1A$
belong to the same deformation family as $k2A$, in fact we explicitly construct the universal family of extremal neighborhoods. This construction follows very closely Mori's division algorithm, which can be interpreted as a sequence of mutations in the cluster algebra. We~identify, in the versal deformation space of a cyclic quotient singularity, the locus of deformations such that the total space admits a (terminal) antiflip. We show that these deformations come from at most two irreducible components of the versal deformation space. As an application, we give an algorithm for computing stable one-parameter degenerations of smooth projective surfaces (under some conditions) and describe several components of the Koll\'ar--Shepherd-Barron--Alexeev boundary of the moduli space of smooth canonically polarized surfaces of geometric genus zero.
\end{abstract}

\section{Introduction}

We continue the study of semistable extremal neighborhoods
pioneered by Mori, Koll\'ar, and Prokhorov \cite{M88, KM92, M02, MP10}.
Our motivation is to compute explicitly stable one-parameter degenerations of
smooth canonically polarized surfaces and to describe their moduli space introduced by Koll\'ar and Shepherd-Barron \cite{KSB88}.

We work throughout over $k=\bC$. Extremal neighborhoods are fundamental building blocks
of the Minimal Model Program approach to $3$-dimensional birational geometry \cite{KM98}.
The goal of this program is to construct a canonical representative in the birational class
of a smooth (or mildly singular) proper $3$-fold by gradually forcing the canonical divisor
to intersect curves non-negatively.
On each step of the program, one has to analyze the local picture of a $3$-fold
around a curve which has negative intersection with the canonical divisor.
Concretely, let $\X$ be the germ of a $3$-fold along a proper reduced irreducible\footnote{
In some papers $C$ is allowed to be reducible, however we can always reduce to the irreducible
case by factoring $f$ locally analytically over $Q \in \Y$, see \cite[8.4]{K88}.} curve $C$ such that $\X$ has terminal singularities.
We say $\X$ is an \emph{extremal neighborhood} if there is a germ $Q \in \Y$ of a $3$-fold and a proper birational morphism
$$f \colon (C \subset \X) \rightarrow (Q \in \Y)$$
such that
 $f_*\cO_\X = \cO_\Y$, $f^{-1}(Q)=C$ (as sets) and $K_\X\cdot C<0$. 
Then $C \simeq \bP^1$.
The exceptional locus (the union of positive-dimensional fibers)  of $f$
is either equal to $C$ or is a divisor with image a curve passing through~$Q$.
We say $f$ is a \emph{flipping contraction} in the first case and a \emph{divisorial contraction} in the second.

If $f$ is a divisorial contraction  
then $Q \in \Y$ is a terminal singularity \cite[Th 1.10]{MP10}.
If one is running the minimal model program, $\Y$~will (locally) give its next step.
If $f$ is a flipping contraction then one has to construct the flip.
The \emph{flip} of $f$ is a germ $\X^+$ of a $3$-fold along a proper reduced curve $C^+$ such that $\X^+$ has terminal singularities, and a proper birational morphism
$$f^+ \colon (C^+ \subset \X^+) \rightarrow (Q \in \Y)$$
such that $f^+_*\cO_{\X^+}=\cO_\Y$, $(f^+)^{-1}(Q)=C^+$ (as sets), and $K_{\X^+}\cdot C^+>0$. 

The existence of the flip was proved by Mori \cite[Th.~0.4.1]{M88}
and uniqueness follows from
$\X^+=\Proj_\Y\bigoplus\limits_{m \ge 0} \cO_\Y(mK_\Y)$
cf. \cite[Cor.~6.4]{KM98}.

To classify  extremal neighborhoods, following Miles Reid,
we choose a general member $E_X \in |-K_\X|$. Let  $E_Y=\pi(E_X) \in |-K_\Y|$. Then $E_X$ and $E_Y$ are normal surfaces with at worst Du Val singularities \cite[Th.~1.7]{KM92}\footnote{
This is a special case of the \emph{general elephant conjecture} of Miles Reid.
}.
We say the extremal neighborhood $f \colon \X \rightarrow \Y$ is \emph{semistable} if $Q \in E_Y$ is a Du Val singularity of type $A$ \cite[p.~541]{KM92}.
Semistable extremal neighborhoods are of two types $k1A$ and $k2A$. We~say $f$ is of type $k1A$ or $k2A$ if the number of singular points of $E_X$ equals one or two respectively \cite[\S1, C.4, p.542]{KM92}.
In the case of a flipping contraction, $f$ is semistable iff the general hyperplane section through $Q \in \Y$ is a cyclic quotient singularity $\frac{1}{n}(1,a)$ (\cite{KM92}, Corollary~3.4 and Appendix), i.e.,
 is (locally analytically) isomorphic to the quotient of $\bA^2$ by $\mu_n$ acting
by $(x,y)\to (\zeta x,\zeta^a y)$, where $\zeta\in\mu_n$ is a primitive root of unity.

We are primarily interested in semistable neighborhoods due to applications to compact moduli spaces of surfaces.
Let $s \in m_{\Y,Q} \subset \cO_{\Y,Q}$ be a general element. Consider the corresponding morphism $\Y\rightarrow\bA^1_s$
and the induced morphism $\X\rightarrow \bA^1_s$.
Write
$$X=X_0=(s=0) \subset \X\quad\hbox{\rm and}\quad Y=Y_0=(s=0) \subset \Y.$$
Thus we can view an extremal neighborhood as a total space of a flat family of surfaces.

Let $X_s$, $0< |s| \ll 1$, denote a nearby smooth fiber. In this paper we assume that the second Betti number
$$b_2(X_s)=1.$$
Our motivation for this assumption is as follows.

Let $\cM_{K^2,\chi}$ be Gieseker's moduli space of canonical
surfaces of general type with given numerical invariants $K^2$ and
$\chi$. Let $\overline{\cM}_{K^2,\chi}$ be its compactification,
the moduli space of stable surfaces of Koll\'ar, Shepherd-Barron,
and Alexeev. Given a map $\D^{\times}\to\cM_{K^2,\chi}$ from a punctured
smooth curve germ, i.e.~ a family $\X^{\times}\to\D^{\times}$ of canonical surfaces,
one is interested in computing its ``stable limit''
$\D\to\overline \cM_{K^2,\chi}$, perhaps after a finite base
change. In many important cases the
family $\X^{\times}\to\D^{\times}$ can be compactified by a flat family
$\X\to\D$, where the special fiber $X$ is irreducible, normal, and
has quotient singularities. Computing the stable limit then amounts to
running the relative minimal model program for $\X\to\D$.
It turns out that we can always reduce to the case of
flips and divisorial contractions of type $k1A$ and $k2A$
with $b_2(X_s)=1$ (in an analytic neighborhood of a contracted curve)
after a finite surjective base change by passing to a small partial resolution using simultaneous resolution of Du Val singularities
(see e.g. \cite[Th.~4.28]{KM98}  and \cite[\S2]{BC94}). See Section \ref{s4} for more details.

Building on \cite{MP10}, we show in Section \ref{s1} that
if $\X$
is an extremal neighborhood of type $k1A$ (resp.~$k2A$) with
$b_2(X_s)=1$ then $X$ is normal and has one (resp.~two) cyclic quotient singularities along $C$ of the following special form
$$(P \in X) \simeq 
\textstyle{\frac{1}{m^2}}(1,ma-1)$$
for some $m,a \in \bN$, $\gcd(a,m)=1$, a so-called \emph{Wahl singularity}.

In \cite{M02} Mori gave an explicit algorithm for computing flips of extremal neighborhoods of type $k2A$.
We extend his result to neighborhoods of type $k1A$.
The novelty of our approach is that
instead of considering extremal neighborhoods one-by-one,
we construct their universal family (see Section~\ref{s2}), and in particular show that $k1A$
neighborhoods belong to the same deformation family as $k2A$ (see Proposition \ref{k1Adegeneratestok2A}, whose proof in given at the end of \S \ref{k1Atok2A}).

\begin{theorem}
For any extremal neighborhood of type $k1A$ or $k2A$ with $b_2(X_s)=1$,
in Section~\ref{s2} we define a four-tuple of integers
$(a_1',m_1',a_2',m_2')$ and the following data
which depends only on this four-tuple:

A positive integer $\delta$, a toric surface $M$ and
a toric birational morphism $p:\, M \rightarrow \bA^2$ which depend only on $\delta$
($M$ is only locally of finite type for $\delta>1$),
flat irreducible  families of surfaces
$$\bU\to M,\quad \bY\to  \bA^2,\quad \bX^+\to\bA^2,$$
and morphisms
$$\pi:\,\bU\to\bY\times_{\bA^2} M,\quad\hbox{\rm and}\quad \pi^+:\,\bX^+\to\bY,$$
such that the following holds.

Let $f \colon (C \subset \X) \rightarrow (Q \in \Y)$ be an extremal neighborhood of type $k1A$ or $k2A$ with $b_2(X_s)=1$
and with associated four-tuple $(a_1',m_1',a_2',m_2')$.
Then there is a morphism
$g \colon (0 \in \bA^1_t) \rightarrow M$
such that $p(g(0))=0 \in \bA^2$ and $\X\to\Y$
is the pull-back of $\bU\to\bY \times_{\bA^2} M$ under $g$.
In other words, $\bU$ is the universal family of extremal neigborhoods.

If $f$ is a flipping contraction then the flip
$f^+:\,\X^+\to\Y$
is the pullback of $\bX^+\to\bY$ under $p \circ g$.
\label{thmintro}
\end{theorem}

We give an explicit combinatorial algorithm to compute all $k1A$ and $k2A$ corresponding to this family (see Propositions~\ref{k1Ainvariants} and \ref{k1Adegeneratestok2A}).

The construction of the universal family, which is the content of Section~\ref{s2}, follows very closely the division algorithm in \cite{M02},
which one can interpret as a sequence of mutations in the cluster algebra of rank $2$ with general coefficients (see Remark~\ref{cluster}).

If $\delta =1$ the morphism $p$ is the blowup of $0 \in \bA^2$.
If $\delta\ge2$, $M$ is the toric variety associated to the fan
$\Sigma$ with support
$$
\{ (x_1,x_2) \in \bR^2 \ | \ x_1,x_2 \ge 0 \mbox { and either } x_1 > \xi x_2 \mbox{ or } x_2 > \xi x_1 \} \cup \{0\},
$$
where
$\xi = (\delta + \sqrt{\delta^2-4})/2 \ge 1$.
It has rays spanned by $v_i \in \bZ^2$ for $i \in \bZ \setminus \{0\}$ given by
$v_1=(1,0)$, $v_2=(\delta,1)$,
$v_{i+1}+v_{i-1}=\delta v_i$ for  $i \ge 2$,
and
$v_{-1}=(0,1)$, $v_{-2}=(1,\delta)$,
$v_{-(i+1)}+v_{-(i-1)}=\delta v_{-i}$ for  $i \ge 2$.
Over torus orbits of dimension $1$ (resp. $0$) corresponding to rays of $\Sigma$ (resp.~$2$-dimensional cones),
fibers of $\pi$ are surfaces with one (resp.~two) Wahl singularities.

In the flipping case, the special fiber $Y\subset\Y$ is a germ of a cyclic quotient singularity
$\frac{1}{\Delta}(1,\Omega)$ and the special fiber $X^+\subset\X^+$ is its
\emph{extremal P-resolution}, i.e., a partial resolution $f^+_0 \colon (C^+ \subset X^+) \rightarrow (Q \in Y)$
such that $X^+$ has only Wahl singularities
$\frac{1}{m_i'^2}(1,m_i'a_i'-1)$ for $i=1,2$,
the exceptional curve $C^+=\P^1$, and $K_{X^+}$ is relatively ample.

\begin{example}
Let ${m}'_1=5, {a}'_1=2$, and ${m}'_2=3, {a}'_2=1$. The Wahl singularities $\frac{1}{5^2}(1,9)$ and $\frac{1}{3^2}(1,2)$ are minimally resolved by chains of smooth rational curves $E_1, E_2, E_3$ and $F_1, F_2$ respectively, such that $E_1^2=-3$, $E_2^2=-5$, $E_3^2=-2$, $F_1^2=-5$, and $F_2^2=-2$. Let $(C^+ \subset X^+)$ be a surface germ around a $C^+=\P^1$ such that $X^+$ has these two Wahl singularities, $C^+$ passes through both of them, and in the minimal resolution of $X^+$, the proper transform of $C^+$ is a $(-1)$-curve which only intersects $E_1$ and $F_1$, and transversally at one point. With this data we build a $\pi^+ \colon \bX^+ \to \bY$ as Theorem \ref{thmintro}. We have $\delta=4$. With this we construct the universal antiflip $\pi : \bU \to \bY$, where the central fiber of $\bY$ is a surface germ with one singularity $\frac{1}{94}(1,53)$, and the corresponding birational toric morphism $p \colon M \to \bA^2$ defined by the morphism between fans shown in Figure \ref{f0}.

In this figure, the rays $v_{i}$, $v_{-i}$ with $i\geq 2$ in the fan $\Sigma$ are decorated by the data $(m,a)$ of the associated $k1A$ antiflip, which has as central fiber a surface with one singularity $\frac{1}{m^2}(1,ma-1)$. Two consecutive $v_{i},v_{i+1}$ or $v_{-i},v_{-(i+1)}$ with $i \geq 2$ represents a $k2A$ antiflip, which has two singularities, one from each of the $k1A$ of the rays. These $k1A$'s are recovered from the $k2A$ by $\Q$-Gorenstein smoothing up one of the singularities while preserving the other. When $i$ tends to infinity, the rays $v_i$ and $v_{-i}$ approach to two lines with irrational slopes $\xi$ and $\frac{1}{\xi}$, where $$ \xi=2+\sqrt{3} = 4 - \frac{1}{4 - \frac{1}{\ddots }}.$$ The rays $v_{1}$ and $v_{-1}$ correspond to the Wahl singularities $\frac{1}{3^2}(1,2)$ and $\frac{1}{5^2}(1,9)$ of $X^+$ respectively.

\begin{figure}[htbp]
\includegraphics[width=12cm]{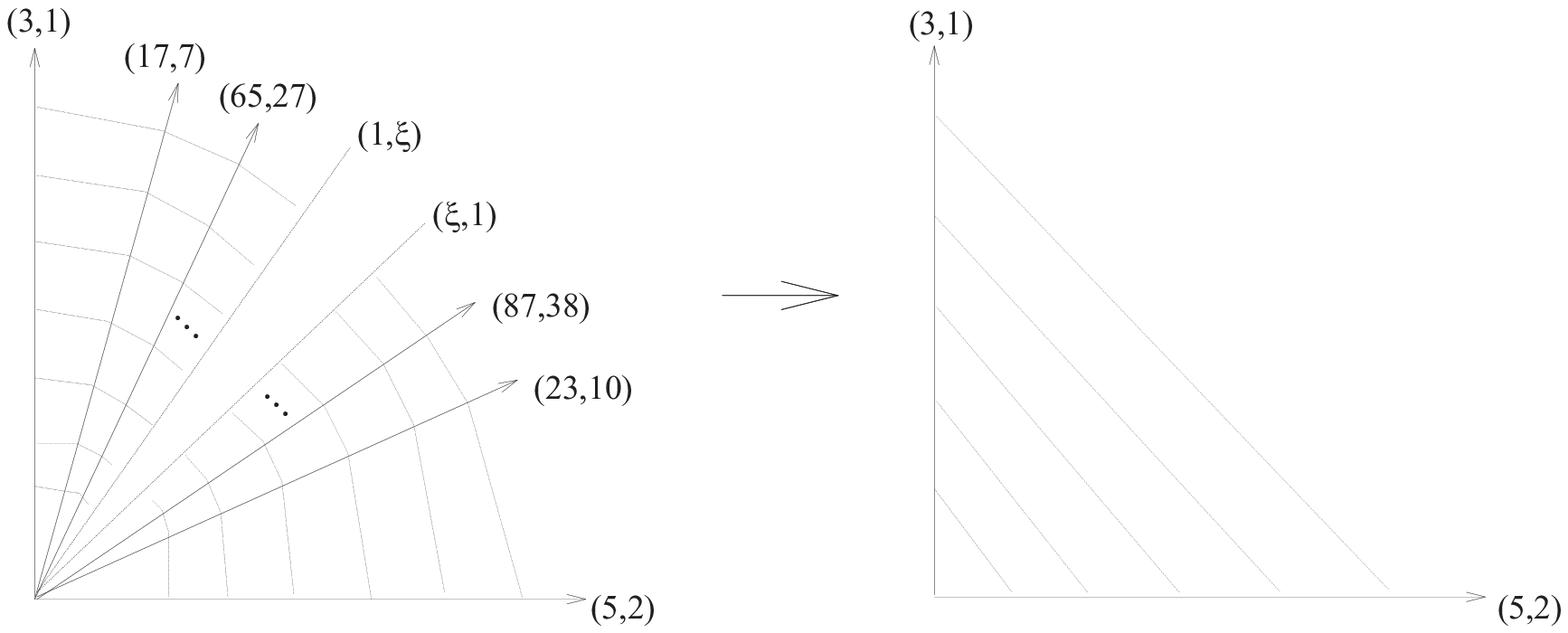}
\caption{The map between fans for $p \colon M \to \bA^2$.}
\label{f0}
\end{figure}
\end{example}

In Section~\ref{s2}, following \cite{M02},
we show that an extremal P-resolution of a cyclic quotient singularity gives a unique family $\bU\to M$ of (terminal) antiflips. In addition we identify, in the versal deformation space of this cyclic quotient singularity $X^+$,
the locus of deformations such that the total space admits a terminal antiflip, see Corollary~\ref{existenceofterminalantiflips}.
Specifically, we show that a smoothing $\X^+\to\bA^1_s$ admits a terminal antiflip
if and only if there is a divisor $D \in \mbox{$|-K_{\X^+}|$}$ such that $D|_{X^+}$ is the toric boundary of $X^+$
for some choice of toric structure and the axial multiplicities $\alpha_1,\alpha_2$
of the singularities of $\X^+$ satisfy
$\alpha_1^2 -\delta\alpha_1\alpha_2 + \alpha_2^2 > 0$.
This raises the following question, which we find very interesting:

\begin{question}
Classify smoothings of cyclic quotient singularities which admit an antiflip (not necessarily terminal).
\end{question}

See also Remark~\ref{asflnsfgajf} and Example~\ref{asdvasfgafg} of a canonical antiflip.

In \S \ref{2extremalnbhd} we classify extremal P-resolutions.
We show, by a combinatorial argument built on continued fractions representing zero \cite{C89,S89} (see \S \ref{contfracpresol}), that any cyclic quotient singularity admits at most two extremal P-resolutions (and thus at most two families of antiflips). Moreover, we prove in \S \ref{delta} that these two families must have the same parameter $\delta$, i.e., they live over the same toric surface $M$. At the end of \S \ref{contfracpresol}, we give a list with the first cyclic quotient singularities (up to order $45$) admitting one or two extremal P-resolutions.

Finally, in Section~\ref{s4}, we give an application to moduli of stable surfaces,
which was the original motivation for this paper.
We start by constructing a stable surface which admits a $\bQ$-Gorenstein smoothing
to a canonically polarized smooth surface with $K^2=4$ and $p_g=0$, following ideas of Lee and Park \cite{LP07}.
The surface $X$ has two Wahl singularities $P_1=\frac{1}{252^2}(1,252 \cdot 145 -1)$ and $P_2=\frac{1}{7^2}(1,7 \cdot 5 -1)$, and no local-to-global obstructions to deform (i.e. deformations of its singularities globalize to deformations of $X$).
The Koll\'ar--Shepherd-Barron--Alexeev moduli space of stable surfaces near $X$ is two-dimensional and has two boundary curves,
which correspond to $\bQ$-Gorenstein deformations $X\rightsquigarrow X_1$ and $X\rightsquigarrow X_2$
which smoothen $P_1$ (resp.~$P_2$)
and preserve the remaining singularity. We address the question of describing minimal models of $X_1$ and $X_2$.
Namely, we consider deformations $\tilde X\rightsquigarrow \tilde X_1$ and $\tilde X\rightsquigarrow \tilde X_2$
of surfaces obtained by simultaneously resolving $P_1$ (resp.~$P_2$) and run our explicit MMP
to find new limits $\tilde X^+$ of each family. This allows us to conclude that
$\tilde X_1$ is a rational surface and $\tilde X_2$ is a Dolgachev surface of type $(2,3)$, i.e.,
an elliptic fibration over $\bP^1$ with two multiple fibers of multiplicity $2$ and $3$, respectively.
More systematic study of the boundary of the moduli space of surfaces
of geometric genus $0$ will appear elsewhere.

\subsection*{Acknowledgements}

Paul Hacking was supported by NSF grants DMS-0968824 and DMS-1201439. Jenia Tevelev was supported by NSF grants DMS-1001344 and DMS-1303415. Giancarlo Urz\'ua was supported by the FONDECYT Inicio grant 11110047 funded by the Chilean Government. We would like to thank I.~Dolgachev, J.~Koll\'ar, S.~Mori, and M.~Reid for helpful discussions. In particular, M. Reid explained to us in detail his joint work with G. Brown which describes explicitly the graded ring of a $k2A$ flip \cite{BR12}.

\tableofcontents

\section{Extremal neighborhoods} \label{s1}

\subsection{Toric background} \label{toric}

We recall that the minimal resolution of a cyclic quotient singularity of type $\frac{1}{n}(1,a)$ has exceptional locus a nodal chain of smooth rational curves
with self-intersection numbers $-b_1,\ldots,-b_r$ where
$$n/a=[b_1,\ldots,b_r]= b_1 - \frac{1}{b_2-\frac{1}{\ddots -\frac{1}{b_r}}},\quad b_i \ge 2 \mbox{ for all $i$}$$
is the expansion of $n/a$ as a Hirzebruch--Jung continued fraction. See \cite{F93}, p.~46. Moreover, the minimal resolution can be described torically as follows.
Identify the singularity $(P \in X)$ with the germ
$$(0 \in \bA^2_{x_1,x_2}/\textstyle{\frac{1}{n}}(1,a))$$
of the affine toric surface corresponding to the cone $\sigma=\bR^2_{\ge 0}$ in the lattice
$$N=\bZ^2+\bZ \textstyle{\frac{1}{n}}(1,a).$$
Define $\alpha_{i},\beta_{i} \in \bN$ for $1 \le i \le r$ by $\alpha_{1}=\beta_{r}=1$ and
$$\alpha_{i}/\alpha_{i-1}=[b_{i-1},\ldots,b_1] \mbox{ for $2 \le i \le r$},$$
$$\beta_{i}/\beta_{i+1}=[b_{i+1},\ldots,b_{r}] \mbox{ for $1 \le i \le r-1$}.$$
Define $$v_i = \frac{1}{n}(\alpha_i,\beta_i) \mbox{ for $1 \le i \le r$}.$$
Then each $v_i$ is a primitive vector in the lattice $N$.
The minimal resolution $\pi \colon \tilde{X} \rightarrow X$ is the toric proper birational morphism corresponding to the subdivision $\Sigma$ of the cone $\sigma$ given by adding the rays
$$\rho_i = \bR_{\ge 0}v_i$$
generated by the $v_i$ for $i=1,\ldots,r$.
Let $E_i$ denote the exceptional curve corresponding to the ray $\rho_i$. Then $E_i$ is a smooth rational curve such that $E_i^2=-b_i$.
Define $$l_i =(x_i=0) \subset X \mbox{ for $i=1,2$.}$$
Let $l_i' \subset \tilde{X}$ denote the strict transform of $l_i$.
Then the toric boundary of $\tilde{X}$ is the nodal chain of curves $l_1',E_r,\ldots,E_1,l_2'$. Moreover, we have
$$\pi^*l_1=l_1'+\frac{1}{n}\sum \alpha_iE_i$$
and
$$\pi^*l_2=l_2'+\frac{1}{n}\sum \beta_i E_i.$$
(Indeed, since $E_i$ corresponds to the primitive vector $v_i=\frac{1}{n}(\alpha_i,\beta_i) \in N$, the orders of vanishing of $x_1$ and $x_2$ along $E_i$ equal $\alpha_i/n$ and $\beta_i/n$ respectively.)
Finally the discrepancies $c_i \in \bQ$ defined by
$$K_{\tilde{X}}=\pi^*K_X+\sum c_i E_i$$
are given by
$$c_i= -1+(\alpha_i+\beta_i)/n.$$
(Indeed, for any toric variety $X$ with boundary $B$ we have $K_X+B=0$. Thus in our case $K_X=-(l_1+l_2)$ and $K_{\tilde{X}}=-(l_1'+E_r+\cdots+E_1+l_2')$. Now the formula for the discrepancies follows from the formulas for the pullback of $l_1$ and $l_2$ above.)

The Wahl singularity is a cyclic quotient singularity of the form
$$(P \in X) \simeq \bA^2_{p,q}/\textstyle{\frac{1}{m^2}}(1,ma-1)$$
for some $m,a \in \bN$, $\gcd(a,m)=1$.
The number $m$ is the index of the singularity.
We have an identification
$$(P \in X) \simeq (\xi\eta=\zeta^{m}) \subset \bA^3_{\xi,\eta,\zeta}/\textstyle{\frac{1}{m}}(1,-1,a)$$
given by
$$\xi=p^m, \quad \eta=q^m, \quad \zeta=pq.$$
The deformation $(P \in \X) \rightarrow (0 \in \bA^1_t)$ of $(P \in X)$ is of the form
$$(0 \in (\xi\eta=\zeta^m+t^{\alpha}h(t)) \subset (\bA^3_{\xi,\eta,\zeta}/\textstyle{\frac{1}{m}}(1,-1,a)) \times (0 \in \bA^1_t))$$
for some $\alpha \in \bN$ and convergent power series $h(t)$, $h(0) \neq 0$.
It follows that we have an isomorphism of germs
$$(P \in \X) \simeq (0 \in (\xi\eta=\zeta^m+t^{\alpha}) \subset (\bA^3_{\xi,\eta,\zeta}/\textstyle{\frac{1}{m}}(1,-1,a))\times \bA^1_t).$$
The number $\alpha \in \bN$ is called the \emph{axial multiplicity} of $P \in \X$ \cite[Def.~1a.5, p.~140]{M88}.
In particular, the $3$-fold germ $\X$ is analytically $\bQ$-factorial by \cite[2.2.7]{K91}.

The following proposition is mostly a combination of results from \cite{M02} and \cite{MP10}.

\begin{proposition} \label{b2=1=>normal}
Let $f \colon (C \subset \X) \rightarrow (Q \in \Y)$
be an extremal neighborhood of type $k1A$ or $k2A$. Let $t \in m_{\Y,Q} \subset \cO_{\Y,Q}$ be a general element
and let $\X \rightarrow (0 \in \bA^1_t)$ be the corresponding morphism. Let $X=X_0$ denote the special fiber and
$X_s$, $0< s \ll 1$, a nearby fiber. Then $b_2(X_s)=1$ iff $X$ is normal and has the following description.

If $f$ is of type $k2A$ then $X$ has two singular points $P_i \in C \subset X$, $i=1,2$.
For each $i=1,2$ there are numbers $m_i,a_i \in \bN$, $m_i > 1$, $\gcd(a_i,m_i)=1$, and an isomorphism of germs
$$(P_i \in C \subset X) \simeq (0 \in (q_i=0) \subset \bA^2_{p_i,q_i}/\textstyle{\frac{1}{m_i^2}}(1,m_ia_i-1)).$$

If $f$ is of type $k1A$ then $X$ has one singular point $P_1 \in C \subset X$.
There are numbers $m_1,a_1 \in \bN$, $\gcd(a_1,m_1)=1$,  and an isomorphism
$$(P_1 \in C \subset X) \simeq (0 \in (p_1^{m_0}=q_1^{m_2}) \subset \bA^2_{p_1,q_1}/\textstyle{\frac{1}{m_1^2}}(1,m_1a_1-1))$$
for some $m_0,m_2 \in \bN$ such that
$$m_0 \equiv m_2(m_1a_1-1) \bmod m_1^2.$$
\end{proposition}

\begin{proof}
The proper birational morphism
$$f_0 \colon (C \subset X) \rightarrow (Q \in Y)$$
can be described as follows. If $f$ is of type $k2A$ then $X$ is normal and has two singular points $P_i \in C \subset X$, $i=1,2$, of index greater than $1$. For each $i=1,2$ there are numbers $\rho_i,m_i,a_i \in \bN$, $m_i > 1$, $\gcd(a_i,m_i)=1$, and an isomorphism of germs
$$(P_i \in C \subset X) \simeq (0 \in (q_i=0) \subset \bA^2_{p_i,q_i}/\textstyle{\frac{1}{\rho_im_i^2}}(1,\rho_im_ia_i-1)).$$
See \cite[Rk.~2.3]{M02}.
If $f$ is of type $k1A$ then $X$ is not necessarily normal.  If $X$ is normal then $X$ has one singular point $P_1 \in C \subset X$ of index greater than $1$ and possibly an additional singular point $P_2 \in C \subset X$ of index $1$. There are numbers $\rho_1,m_1,a_1 \in \bN$, $\gcd(a_1,m_1)=1$, $\rho_2 \in \bN$ and an isomorphism
$$(P_1 \in C \subset X) \simeq (0 \in (p_1^{m_0}=q_1^{m_2}) \subset \bA^2_{p_1,q_1}/\textstyle{\frac{1}{\rho_1m_1^2}}(1,\rho_1m_1a_1-1))$$
for some $m_0,m_2 \in \bN$ such that $$m_0 \equiv m_2(\rho_1m_1a_1-1) \bmod \rho_1m_1^2.$$
If $\rho_2=1$ then $P_1 \in X$ is the unique singular point.
If $\rho_2>1$ we have an additional singular point $P_2 \in X$ and an isomorphism
$$(P_2 \in C \subset X) \simeq (0 \in (x_2=0) \subset \bA^2_{x_2,y_2}/\textstyle{\frac{1}{\rho_2}}(1,-1))).$$
(So in particular $P_2 \in X$ is a Du Val singularity of type $A_{\rho_2-1}$.)
See \cite[Th.~1.8.1]{MP10}.

If $X$ is normal then we have $$b_2(X_s)=1+(\rho_1-1)+(\rho_2-1).$$ Indeed, $C$ is irreducible by assumption, the link of a cyclic quotient singularity is a lens space, and the Milnor fiber $M$ of a $\bQ$-Gorenstein smoothing of a singularity of type $\frac{1}{\rho m^2}(1,\rho ma-1)$ has Milnor number $\mu:=b_2(M)=\rho-1$. Now the result follows from a Mayer--Vietoris argument. So $b_2(X_s)=1$ iff $\rho_1=\rho_2=1$ as claimed.

It remains to show that if $X$ is not normal then $b_2(X_s)>1$. We use the classification of the non-normal case given in \cite{MP10}, Theorem~1.8.2.
Let $\nu \colon H' \rightarrow X$ denote the normalization of $X$ and $C'\subset H'$ the inverse image of $C$. Then $C'$ has two components $C_1', C_2' \simeq \bP^1$ meeting in a single point $N'$.  We have an isomorphism
$$(N' \in C' \subset H') \simeq (0 \in (x_1x_2=0) \subset \bA^2_{x_1,x_2}/\textstyle{\frac{1}{n}}(1,b))$$
for some $n,b$. There are also points $P_1' \in C_1' \setminus \{N'\}$, $P_2' \in C_2'\setminus \{N'\}$ and isomorphisms
$$(P_1' \in C_1' \subset H') \simeq (0 \in (x_1=0) \subset \bA^2_{x_1,x_2}/\textstyle{\frac{1}{m}}(1,a))$$
and
$$(P_2' \in C_2' \subset H') \simeq (0 \in (x_1=0) \subset \bA^2_{x_1,x_2}/\textstyle{\frac{1}{m}}(1,-a))$$
for some $m,a$.
The non-normal surface $X$ is formed from $H'$ by identifying the curves $C_1'$ and $C_2'$. The points $P'_1$ and $P'_2$ are identified to a point $P \in X$ and we have an isomorphism
$$(P \in X) \simeq (0 \in (xy=0) \subset \bA^2_{x,y,z}/\textstyle{\frac{1}{m}}(1,-1,a)).$$
Moreover the image $N \in X$ of $N'$ is a degenerate cusp singularity (cf.~\cite[Def.~4.20]{KSB88})
and $X$ has normal crossing singularities along $C \setminus \{N,P\}$.

Let $\mu \colon \tilde{H} \rightarrow H'$ be the minimal resolution. Since $Y=Y_0$ is a rational singularity (by \cite{KM92}, Corollary~3.4), $\tilde{H}$ is a resolution of $Y$, and $Y_s$ is a smoothing of $Y$, we have
$$K_{\tilde{H}}^2+b_2(\tilde{H}) = K_{Y_s}^2+b_2(Y_s).$$
See \cite{L86}, 4.1c. We remark that $K^2$ is well-defined in our situation (see e.g. \cite{L86}), and $b_2$ is the second Betti number as usual. The morphism $X_s \rightarrow Y_s$ is a proper birational morphism of smooth surfaces, hence a composition of blowups, so
$$K_{X_s}^2+b_2(X_s)=K_{Y_s}^2+b_2(Y_s).$$
Also $K_{X_s}^2=K_{X}^2$ because $\X \rightarrow (0 \in \bA^1_t)$ is $\bQ$-Gorenstein, so we obtain
$$K_{{\tilde{H}}}^2+b_2(\tilde{H})=K_X^2+b_2(X_s).$$
We use this equation to show that $b_2(X_s)>1$ if $X$ is non-normal.

Write $\tilde{C}_i \subset \tilde{H}$ for the strict transform of $C_i'$.
The exceptional locus of $\mu$ is a disjoint union of $3$ chains of smooth rational curves $E_1,\ldots,E_r$, $F_1,\ldots,F_s$, and $G_1,\ldots,G_l$ (the exceptional loci over the singularities $P_1'$, $P_2'$, and $N'$) such that the curves
$$E_r,E_{r-1},\ldots,E_1,\tilde{C}_1,G_1,\ldots,G_l,\tilde{C}_2,F_1,\ldots,F_s$$
form a nodal chain.
We have
$$K_{\tilde{H}}+\tilde{C}_1+\tilde{C_2}=\pi^*(K_{H'}+C')+ \sum \alpha_i E_i + \sum \beta_i F_i - \sum G_i$$
for some $\alpha_i,\beta_i \in \bQ$. Write $E_i^2=-a_i$ and $F_i^2=-b_i$. By Lemma~\ref{toricK^2} below, we have
$$(\sum \alpha_i E_i)^2 + (\sum \beta_i F_i)^2 = (2r-\sum a_i -a/m)+(2s-\sum b_i -(m-a)/m)$$
$$=2(r+s)-1 -\sum a_i - \sum b_i.$$
Now since the cyclic quotient singularities $P_1'$ and $P_2'$ are conjugate, that is, of the form $\frac{1}{m}(1,a)$ and $\frac{1}{m}(1,-a)$ for some $a,m$, we have the identity
$$\sum (a_i-1) = \sum (b_i-1) =r+s-1.$$
(Proof: By induction using the following characterization of conjugate singularities. We identify a cyclic quotient singularity $\frac{1}{m}(1,a)$ with the associated continued fraction $m/a=[a_1,\ldots,a_r]$, $a_i \ge 2$. Then $[2],[2]$ is a conjugate pair, and if $[a_1,\ldots,a_r]$, $[b_1,\ldots,b_s]$ is a conjugate pair, so is $[a_1+1,\ldots,a_r]$, $[2,b_1,\ldots,b_s]$.) So we obtain
$$(\sum \alpha_i E_i)^2 + (\sum \beta_i F_i)^2 = 1 - r-s.$$
Thus
$$(K_{\tilde{H}}+\tilde{C}_1+\tilde{C_2}+\sum G_i)^2= (K_{H'}+C')^2+ 1 - r-s.$$
We also compute
$$(K_{\tilde{H}}+\tilde{C}_1+\tilde{C_2}+\sum G_i)^2=K_{\tilde{H}}^2+2K_{\tilde{H}}(\tilde{C}_1+\tilde{C}_2+\sum G_i) + (\tilde{C}_1+\tilde{C_2}+\sum G_i)^2$$
$$=K_{\tilde{H}}^2-2((2+\tilde{C}_1^2)+(2+\tilde{C}_2^2)+\sum (2+G_i^2)) + \tilde{C}_1^2+\tilde{C}_2^2+(\sum G_i)^2+2(l+1)$$
$$=K_{\tilde{H}}^2-\tilde{C}_1^2-\tilde{C}_2^2-\sum G_i^2-2l-6.$$
where we have used the adjunction formula $K_{\tilde{H}}\Gamma+\Gamma^2=-2$ for $\Gamma \simeq \bP^1$.
Combining, we obtain
\begin{equation}\label{MPbound}
K_{\tilde{H}}^2=(K_{H'}+C')^2+\tilde{C}_1^2+\tilde{C}_2^2+\sum G_i^2+2l+7-r-s.
\end{equation}
By \cite{MP10}, Theorem 1.8.2, we have
$$\tilde{C}_1^2+\tilde{C}_2^2+\sum G_i^2+2l+5 \ge 0.$$
So
$$K_{\tilde{H}}^2 \ge (K_{H'}+C')^2 +2-r-s$$
Now $b_2(\tilde{H})=r+s+l+2$ and $K_{H'}+C'=\nu^*K_X$, so combining we obtain
$$K_{\tilde{H}}^2+b_2(\tilde{H}) \ge K_X^2+l+4.$$
Now (\ref{MPbound}) gives
$b_2(X_s) \ge l+4 \ge 4$.
\end{proof}

The first formula of the following lemma is well-known, see e.g. \cite{Ish00}.

\begin{lemma}\label{toricK^2}
Let $(P \in X)$ be a cyclic quotient singularity of type $\frac{1}{n}(1,a)$.
Let $\pi \colon \tilde{X} \rightarrow X$ be the minimal resolution of $X$.
Write $n/a=[b_1,\ldots,b_r]$, $b_i \ge 2$ for all $i$. Let $a'$ denote the inverse of $a$ modulo $n$.
Then $$K_{\tilde{X}}^2=2r+2-\sum_{i=1}^r b_i - (2+a+a')/n.$$

Let $P \in D \subset X$ be a curve such that $$(P \in X, D) \simeq (0 \in \bA^2_{x_1,x_2}/\textstyle{\frac{1}{n}}(1,a),(x_1=0)).$$
Let $D' \subset \tilde{X}$ be the strict transform of $D$.
Then
$$(K_{\tilde{X}}+D')^2=2r-\sum_{i=1}^r b_i -a/n.$$
\end{lemma}
\begin{proof}
Write $l_i=(x_i=0) \subset X$ and let $l_i' \subset \tilde{X}$ be the strict transform.
Let $E_1,\ldots,E_r$ be the exceptional curves of $\pi$, a nodal chain of smooth rational curves with self-intersection numbers $-b_1,\ldots,-b_r$, such that $l_2'$ meets $E_1$ and $l_1'$ meets $E_r$.
We have
$$K_{\tilde{X}}=-(l_1'+l_2'+\sum E_i) = \sum a_iE_i$$
for some $a_i \in \bQ$, $-1<a_i<0$. The exceptional curves $E_1$ and $E_r$ correspond to the rays in the fan $\Sigma$ of $\tilde{X}$ generated by $\frac{1}{n}(1,a)$ and $\frac{1}{n}(a',1)$, so $a_1=\frac{1}{n}(1+a)-1$ and $a_r=\frac{1}{n}(a'+1)-1$. Now we have
$$K_{\tilde{X}}^2= -K_{\tilde{X}}(l_1'+l_2'+\sum E_i) = -(\sum a_i E_i)(l_1'+l_2')-K_{\tilde{X}}(\sum E_i)$$
$$=-a_1-a_r+ \sum (E_i^2+2)= 2r+2 - \sum b_i - (2+a+a')/n$$
where we have used the adjunction formula $K_{\tilde{X}}E_i+E_i^2=-2$.

In the second case, we have
$$K_{\tilde{X}}+D'=K_{\tilde{X}}+l_1'=-(l_2'+\sum E_i)=\sum (a_i-\mu_i)E_i$$
where
$$\pi^*l_1=l_1'+\sum \mu_i E_i,$$
in particular, $\mu_1=\frac{1}{n}$.
So
$$(K_{\tilde{X}}+D')^2=-(K_{\tilde{X}}+D')(l_2'+\sum E_i) = -(\sum (a_i-\mu_i)E_i)l_2' -(K_{\tilde{X}}+l_1')(\sum E_i)$$
$$=-(a_1-\mu_1) + \sum (E_i^2+2)-1=2r-\sum b_i-a/n.$$
\end{proof}

\subsection{Explicit description of $k1A$ neighborhoods}

\begin{proposition}\label{k1Ainvariants}
Let $f \colon (C \subset \X) \rightarrow (Q \in \Y)$
be an extremal neighborhood of type $k1A$ with $b_2(X_s)=1$
as in Proposition~\ref{b2=1=>normal}, in particular
$$(P_1\in C \subset X) \simeq (0 \in (p_1^{m_0}=q_1^{m_2}) \subset \bA^2_{p_1,q_1}/\textstyle{\frac{1}{m_1^2}}(1,m_1a_1-1))$$
for some $m_0,m_2 \in \bN$ such that
$$m_0 \equiv m_2(m_1a_1-1) \bmod m_1^2.$$

Let $\pi \colon \tilde{X} \rightarrow X$ denote the minimal resolution of $X$. Thus the exceptional locus of $\pi$ is a nodal chain of smooth rational curves $E_1,\ldots,E_r$ with self-intersection numbers $-e_1$,..., $-e_r$ where
$$m_1^2/(m_1a_1-1)=[e_1,\ldots,e_r].$$
Let $\tilde C\subset \tilde{X}$ denote the strict transform of $C$.
Then $\tilde C$ is a $(-1)$-curve which intersects a single component $E_i$ of the exceptional locus of $\pi$
transversely in one point.
Moreover, identifying $(P_1 \in X)$ with the germ of the affine toric variety given by the cone $\sigma=\bR^2_{\ge 0}$ in the lattice
$$N=\bZ^2+\bZ\textstyle{\frac{1}{m_1^2}}(1,m_1a_1-1),$$
the exceptional divisor $E_i$ is extracted by the toric proper birational morphism 
given by subdividing $\sigma$ by the ray $$\rho=\bR_{\ge 0} \cdot \textstyle{\frac{1}{m_1^2}}(m_2,m_0).$$

Define $a_k \in \bN$, $1 \le a_k \le m_k$, $\gcd(m_k,a_k)=1$ for each $k=0,2$, by
$$[e_r,\ldots,e_{i+1}]=m_0/(m_0-a_0)$$
if $i<r$, $m_0=a_0=1$ if $i=r$, and
$$[e_1,\ldots,e_{i-1}]=m_2/(m_2-a_2)$$
if $i>1$, $m_2=a_2=1$ if $i=1$.

The  surface germ $(Q \in Y)$ is isomorphic to a cyclic quotient singularity of type $\frac{1}{\Delta}(1,\Omega)$, where
$$[e_1,\ldots,e_{i-1},e_i-1,e_{i+1},\ldots,e_r]= \Delta/\Omega,$$
equivalently, $$\Delta=m_1^2-m_0m_2\quad \hbox{\rm and}\quad \Omega=m_1a_1-m_0(m_2-a_2)-1.$$
We also have
$$K_X \cdot C = -\delta/m_1,\quad \hbox{\rm
where }\quad\delta :=(m_0+m_2)/m_1, \quad \delta \in \bN,$$
and
\begin{equation}\label{formulaa_0a_2}
a_0+(m_2-a_2)=\delta a_1.
\end{equation}
\end{proposition}

\begin{proof}
The description of the minimal resolution together with the strict transform $C'$ of $C$ is given by \cite{MP10}, Theorem~1.8.1.
Contracting $C'$, we obtain a resolution of $Q \in Y$ with exceptional locus a chain of smooth rational curves with self-intersection numbers
$$-e_1,.. ,-e_{i-1},-(e_i-1),-e_{i+1},\ldots,-e_r.$$
Hence $Q \in Y$ is a cyclic quotient singularity of type $\frac{1}{\Delta}(1,\Omega)$ where\footnote{
Although this resolution is not necessarily minimal, the continued fraction is well defined and computes the type of the singularity,
see \cite[Prop.~2.1]{BR12}.
}
$$\Delta/\Omega = [e_1,\ldots,e_{i-1},e_i-1,e_{i+1},\ldots,e_r].$$
Given $n,a \in \bN$ with $a<n$, $\gcd(a,n)=1$, write
$$n/a=[b_1,\ldots,b_r], \quad b_i \ge 2 \mbox{ for all $i$}$$
for the Hirzebruch--Jung continued fraction.
Define $a',q \in \bN$ by
$$aa'=1+qn, \quad a'<n.$$
The continued fraction corresponds to the factorization in $\SL(2,\bZ)$
$$\begin{pmatrix} -q & a \\ -a' & n \end{pmatrix} =
\begin{pmatrix} 0 & 1 \\ -1 & b_1 \end{pmatrix} \cdots \begin{pmatrix} 0 & 1 \\ -1 & b_r \end{pmatrix},$$
see \cite{BR12}, Proposition~2.1.
Now, assuming $i \neq 1,r$, the equalities
$$[e_r,\ldots,e_{i+1}]=m_0/(m_0-a_0)$$
$$[e_1,\ldots,e_{i-1}]=m_2/(m_2-a_2)$$
$$[e_1,\ldots,e_r]=m_1^2/(m_1a_1-1)$$
give
$$[e_1,\ldots,e_{i-1},e_i-1,e_{i+1},\ldots,e_r]=\Delta/\Omega,$$
where
$$\begin{pmatrix} * & (m_1a_1-1)-\Omega \\ * & m_1^2-\Delta \end{pmatrix}=\begin{pmatrix} * & m_2-a_2 \\ * & m_2 \end{pmatrix} \begin{pmatrix} 0 & 0 \\ 0 & 1 \end{pmatrix}\begin{pmatrix} * & m_0-a_0' \\ * & m_0 \end{pmatrix}$$
$$= \begin{pmatrix} * & m_0(m_2-a_2) \\ * & m_0m_2 \end{pmatrix}$$
So $\Omega=m_1a_1-m_0(m_2-a_2)-1$ and $\Delta=m_1^2-m_0m_2$ as claimed. The cases $i=1,r$ are proved similarly.


From the discussion in \S\ref{toric},
the discrepancy $c_i$ in the formula
$$K_{\tilde{X}}=\pi^*K_X+\sum_j c_j E_j$$
is given by
$$c_i=-1+(m_2+m_0)/m_1^2.$$
Now
$$K_X \cdot C = \pi^*K_X \cdot C' = (K_{\tilde{X}}-\sum c_j E_j) \cdot C' = -1-c_i = -(m_2+m_0)/m_1^2.$$
Note that $m_1K_X$ is Cartier so
$$\delta:=(m_2+m_0)/m_1 = -m_1 K_X \cdot C$$
is an integer. The matrix factorization
$$\begin{pmatrix} a_1^2+1-m_1a_1 & m_1a_1-1 \\ m_1a_1+1-m_1^2 & m_1^2 \end{pmatrix} =
\begin{pmatrix} 0 & 1 \\ -1 & e_1 \end{pmatrix} \cdots \begin{pmatrix} 0 & 1 \\ -1 & e_r \end{pmatrix}=$$
$$=\begin{pmatrix} * & m_2-a_2 \\ * & m_2 \end{pmatrix}
\begin{pmatrix} 0 & 1 \\ -1 & e_i \end{pmatrix}
\begin{pmatrix} * & * \\ -(m_0-a_0)\cdot & m_0 \end{pmatrix}$$
implies that
$$\begin{pmatrix} m_2&  -(m_2-a_2) \\ * & * \end{pmatrix}
\begin{pmatrix} a_1^2+1-m_1a_1 & m_1a_1-1 \\ m_1a_1+1-m_1^2 & m_1^2 \end{pmatrix} =
\begin{pmatrix} * & * \\ -(m_0-a_0) & m_0 \end{pmatrix},$$
which gives
\begin{equation}\label{SASasf}
m_0=m_2(m_1a_1-1)-(m_2-a_2)m_1^2
\end{equation}
and
\begin{equation}\label{SASSDFASDasf}
-(m_0-a_0)=m_2(a_1^2+1-m_1a_1)-(m_2-a_2)(m_1a_1+1-m_1^2).
\end{equation}
The equation \eqref{SASasf} implies that
\begin{equation}\label{SASasdgaasf}
\delta=a_1m_2+m_1a_2-m_1m_2.
\end{equation}
Adding \eqref{SASasf} and \eqref{SASSDFASDasf}
gives
\begin{equation}\label{SASasdgaaJJsf}
a_0=a_1^2m_2-(m_1a_1+1)(m_2-a_2).
\end{equation}
Finally, combining \eqref{SASasdgaasf} and \eqref{SASasdgaaJJsf} gives \eqref{formulaa_0a_2}.
\end{proof}


\subsection{$k1A$ degenerates to $k2A$}

\begin{proposition}\label{k1Adegeneratestok2A}
Let $f_0 \colon (C \subset X) \rightarrow (Q \in Y)$ be as in Proposition~\ref{k1Ainvariants}.
We retain the notations $\delta,\Delta,\Omega$ and $a_i,m_i$, $i=0,1,2$.

Define proper birational morphisms
$$f^i_0 \colon (C^i \subset X^i) \rightarrow (Q \in Y)$$
for $i=0,1$ as follows. We identify $(Q \in Y)$ with the germ of the affine toric variety given by the cone $\sigma=\bR^2_{\ge 0}$ in the lattice $$N=\bZ^2+\bZ\textstyle{\frac{1}{\Delta}}(1,\Omega).$$
Define
$$v^i=\frac{1}{\Delta}(m_{i+1}^2,m_i^2) \in N.$$
Let $f^i_0$ be the toric proper birational morphism corresponding to the subdivision $\Sigma^i$ of $\sigma$ given by the ray $\rho^i = \bR_{\ge 0} v^i$,
and $C^i$ the exceptional curve of $f^i_0$.

Let $P^i_1, P^i_2 \in X^i$ denote the torus fixed points corresponding to the maximal cones $\langle e_1,v^i \rangle_{\bR_{\ge 0}}$, $\langle e_2,v^i\rangle_{\bR_{\ge 0}}$ of $\Sigma^i$ respectively.
Then $X^i$ has cyclic quotient singularities of types
$$\frac{1}{m_i^2}(1,m_ia_i-1)\quad\hbox{\rm and}\quad \frac{1}{m_{i+1}^2}(1,m_{i+1}a_{i+1}-1).$$

Moreover,
$$K_{X^i} \cdot C^i =-\frac{\delta}{m_im_{i+1}}.$$

The one parameter deformation of $f^i_0 \colon (C^i \subset X^i) \rightarrow (Q \in Y)$ given by the trivial deformation of the $\frac{1}{m_1^2}(1,m_1a_1-1)$ singularity and the versal $\bQ$-Gorenstein smoothing of the remaining singularity of $X^i$ (with the target $(Q \in Y)$ being fixed) has general fiber isomorphic to the morphism $f_0 \colon (C \subset X) \rightarrow (Q \in Y)$.
\end{proposition}

We postpone the proof until the end of \S\ref{k1Atok2A}, since it is a consequence of the construction of the universal family.

\section{Construction of universal family of flipping and divisorial contractions} \label{s2}

We follow \cite{M02} closely.

Let $a_1,a_2, m_1,m_2 \in \bN$ be such that $1 \le a_i \le m_i$ and $\gcd(a_i,m_i)=1$ for each $i=1,2$, and

\begin{enumerate}
\item $\delta := m_1 a_2 + m_2 a_1 - m_1m_2 > 0$,
\item $\Delta := m_1^2+m_2^2 - \delta m_1m_2 > 0$.
\end{enumerate}

\subsection{Construction of versal deformation of $k2A$ surface}\label{k2Adeformation}

Write
$$W=(x_1y_1=z^{m_1}+u_1x_2^{\delta}, \quad x_2y_2 = z^{m_2} +u_2x_1^{\delta}) \subset \bA^5_{x_1,x_2,y_1,y_2,z} \times \bA^2_{u_1,u_2}.$$
$$\Gamma = \{ \gamma=(\gamma_1,\gamma_2)  \ | \ \gamma_1^{m_1} = \gamma_2^{m_2} \} \subset \bG_m^2.$$
Define an action of $\Gamma$ on $W$ by
\begin{equation} \label{action}
\Gamma \ni \gamma=(\gamma_1,\gamma_2) \colon (x_1,x_2,y_1,y_2,z,u_1,u_2) \mapsto \\
\end{equation}
$$(\gamma_1 x_1, \gamma_2 x_2, \gamma_2^{\delta}\gamma_1^{-1}y_1, \gamma_1^{\delta}\gamma_2^{-1}y_2,\gamma_1^{a_1}\gamma_2^{a_2-m_2}z,u_1,u_2)$$
Define
$$W^o=W \setminus (x_1=x_2=0)$$
and
$$\bX=W^o/\Gamma.$$
Write
$$U_1=(x_2 \neq 0) \subset \bX, \quad U_2=(x_1 \neq 0) \subset \bX.$$
Then $\bX=U_1 \cup U_2$ is a open covering of $\bX$ and we have identifications
$$U_i=(\xi_i\eta_i= \zeta_i^{m_i}+u_i) \subset \bA^3_{\xi_i,\eta_i,\zeta_i}/ \textstyle{\frac{1}{m_i}(1,-1,a_i)} \times \bA^2_{u_1,u_2}.$$
for each $i=1,2$.
Here the coordinates $\xi_1,\eta_1,\zeta_1$ are the restrictions of the coordinates $x_1,y_1,z$ on $W$ to the $\mu_{m_1}$ cover of $U_1$ given by setting $x_2=1$.
Similarly for $i=2$.
The glueing
$$U_1 \supset (\xi_1 \neq 0) = (\xi_2 \neq 0) \subset U_2$$
of $U_1$ and $U_2$ is given by
$$\xi_1^{m_1}=\xi_2^{-m_2}, \quad \xi_1^{-a_1}\zeta_1=\xi_2^{m_2-a_2}\zeta_2.$$
(The expression for $\eta_1$ in $U_2$ is determined by the equation of $U_1$.)

Write
$$C=(y_1=y_2=z=u_1=u_2=0) \subset \bX$$
and
$$P_i=(x_i=y_1=y_2=z=u_1=u_2=0) \in \bX$$
for each $i=1,2$.
Then $P_i$ is the point $0 \in U_i \subset \bX$ for each $i=1,2$, and $C$ is a proper smooth rational curve in $\bX$ with local equations
$$C \cap U_i = (\eta_i=u_1=u_2=0) \subset U_i$$
for each $i=1,2$.
Write
$$S_i=(x_i=0) \subset \bX$$
for each $i=1,2$.

The morphism
$$\bX \rightarrow \bA^2_{u_1,u_2}$$
is a flat family of surfaces (flatness is clear in the charts $U_1$ and $U_2$).
Write
$$X=X_0=(u_1=u_2=0) \subset \bX$$
for the fiber over $0 \in \bA^2$.
Then $X$ is a toric surface. Write
$$V_i := U_i \cap X$$
Then
$$V_i=(\xi_i\eta_i = \zeta_i^{m_i}) \subset \bA^3_{\xi_i,\eta_i,\zeta_i}/\textstyle{\frac{1}{m_i}(1,-1,a_i)}$$
and we have an identification
$$V_i=\bA^2_{p_i,q_i}/\textstyle{\frac{1}{m_i^2}(1,m_ia_i-1)}$$
given by
$$\xi_i=p_i^{m_i}, \quad \eta_i=q_i^{m_i}, \quad \zeta_i = p_iq_i.$$
The glueing
$$V_1 \supset (p_1 \neq 0) = (p_2 \neq 0) \subset V_2$$
is given by
$$p_1^{m_1^2}=p_2^{-m_2^2}, \quad p_1^{-(m_1a_1-1)}q_1=p_2^{m_2(m_2-a_2)+1}q_2.$$

The surface $X$ is the toric variety associated to the fan $\Sigma$ in the lattice $N=\bZ^2$ defined as follows. Let
$$v_1=(-m_1a_1+1,-m_1^2), \quad v = (1,0), \quad v_2 = (m_2(m_2-a_2)+1, m_2^2)$$
be vectors in $N$. Let $\Sigma$ be the fan in $N$ with cones $\{0\}$, $\rho_1=\bR_{\ge 0}v_1$, $\rho_2=\bR_{\ge 0}v_2$, $\rho=\bR_{\ge 0} v$, $\sigma_1 = \rho_1+\rho$ and $\sigma_2=\rho+\rho_2$. The ray $\rho_i$ corresponds to the divisor $$l_i:=(p_i=0) \subset V_i$$ for each $i=1,2$ and the ray $\rho$ corresponds to the curve $C \subset X$.

Let
$$\wedge \colon \wedge ^2 \bZ^2 \stackrel{\sim}{\longrightarrow} \bZ$$
$$((a_1,a_2), (b_1, b_2)) \mapsto a_1b_2-a_2b_1$$
denote the usual orientation of $\bZ^2$. Note that
$$v_1 \wedge v_2 = \Delta > 0$$
by assumption. So the support $|\Sigma|$ of the fan $\Sigma$ is the convex cone $\sigma:=\rho_1+\rho_2$. Let $Y$ denote the affine toric surface associated to the cone $\sigma$ in $N$ and let $0 \in Y$ denote the torus fixed point. We compute
$$N= \bZ v_1 \oplus \bZ v_2 + \bZ\frac{1}{\Delta}(1,\Omega)$$
where
$$\Omega \equiv (m_1a_1-1)(m_2a_2+1)-m_1^2(a_2(m_2-a_2)+1)$$
$$\equiv (m_2-\delta m_1)(m_2-a_2)+m_1a_1 -1 \bmod \Delta.$$
Thus
$$Y = \bA^2_{z_1,z_2}/\textstyle{\frac{1}{\Delta}(1,\Omega)}$$
where
$$z_1^{\Delta}=q_1^{m_1^2}, \quad z_2^{\Delta}=q_2^{m_2^2}.$$
We have a toric proper birational morphism
$$\pi_0 \colon X \rightarrow Y$$
with exceptional locus the curve $C \subset X$.

Let $M_{\Gamma}$ denote the cokernel of the homomorphism of abelian groups
$$\bZ \rightarrow \bZ^2, \quad 1 \mapsto (m_1,-m_2).$$
Then
$$\Gamma= \Hom(M_{\Gamma},\bG_m)$$
and
$$M_{\Gamma}=\Hom(\Gamma, \bG_m).$$

\begin{lemma} \label{classgrouplemma}
The equality $\bX=(W \setminus (x_1=x_2=0)) / \Gamma$ induces an identification
$$\Cl(\bX) \stackrel{\sim}{\longrightarrow} M_{\Gamma}$$
such that
$$[S_i] \mapsto \bar{e}_i \in M_{\Gamma}= \bZ^2 / \bZ (m_1,-m_2)$$
for each $i=1,2$.
\end{lemma}

\begin{proof}
Write
$$Z=(x_1=x_2=0) \subset W.$$
Then
$$Z=(x_1=x_2=z=0) \subset \bA^5_{x_1,x_2,y_1,y_2,z} \times \bA^2_{u_1,u_2}.$$
In particular $Z \subset W$ is an irreducible divisor.
We have
$$(x_1)=Z+\tilde{S}_1, \quad (x_2)=Z+\tilde{S}_2.$$
where $\tilde{S}_i \subset W$ is an irreducible divisor for each $i=1,2$.
Note that
$$W':=W \setminus (Z \cup \tilde{S}_1 \cup \tilde{S}_2) = W \setminus((x_1=0) \cup (x_2=0)) = (\bG_m)^2_{x_1,x_2} \times \bA^3_{z,u_1,u_2}.$$
In particular $\Cl(W')=0$ and $\Gamma(W',\cO_W^{\times})/k^{\times}=\langle x_1,x_2 \rangle$.
Now $W^o=W \setminus Z$ and $\tilde{S}_i \cap W^o=(x_i)$ for $i=1,2$, so
$$\Cl(W^o)=\Cl(W')=0$$
and
$$\Gamma(W^o,\cO_W^{\times})=k^{\times}.$$

Write
$$q \colon W^o \rightarrow \bX=W^o/\Gamma$$
for the quotient map.
We can define a group homomorphism
\begin{equation} \label{classgroupmap}
\Cl(\bX) \rightarrow M_{\Gamma}, \quad [D] \mapsto \chi
\end{equation}
as follows. Write
$$q^{-1}D=(f)$$
some $f \in \Gamma(W^o,\cO_W)$.
Then, for each $\gamma \in \Gamma$, we have
$$\gamma \cdot f = \chi(\gamma)f$$
for some $\chi \in M_{\Gamma}=\Hom(\Gamma,\bG_m)$.

The action of the group $\Gamma$ on $W^o$ has finite stabilizers and is free in codimension $1$.
It follows that the homomorphism (\ref{classgroupmap}) is an isomorphism with inverse
$$\chi \mapsto (\cO_{W^o} \otimes_k \chi^{-1})^{\Gamma}$$
where we regard $\chi \in M_{\Gamma}$ as a one dimensional representation of $\Gamma$ and we identify $[D] \in \Cl(\bX)$ with the isomorphism class of the rank one reflexive sheaf $\cO_\bX(D)$.
\end{proof}

\begin{remark}
The quotient map $q \colon W^o \rightarrow \bX$ is identified with the spectrum of the (relative) Cox ring
$$\Spec_\bX \text{Cox}(\bX) \rightarrow \bX,$$ where Cox$(\bX)=\bigoplus_{[D] \in \Cl(\bX)} \cO_\bX(D)$,
via
$$\Gamma=\Hom(\Cl(\bX),\bG_m).$$
Strictly speaking, in order to define the Cox ring we need to choose a representative $D_{\alpha}$ of each class $\alpha \in \Cl(\bX)$ and compatible isomorphisms
$$\cO_\bX(D_{\alpha_1}) \otimes \cO_\bX(D_{\alpha_2}) \stackrel{\sim}{\longrightarrow} \cO_\bX(D_{\alpha_1+\alpha_2}).$$
In our case we take representatives $c_1S_1+c_2S_2$ for $(c_1,c_2) \in \bZ^2$ a list of coset representatives for the quotient $M_{\Gamma}$ and isomorphisms derived from the isomorphism
$$\cO_\bX(m_1S_1) \simeq \cO_\bX(m_2S_2), \quad x_1^{-m_1} \mapsto x_2^{-m_2}.$$
\end{remark}

\subsection{Cluster variables}

Assume without loss of generality that $m_1 \ge m_2$.
Define a sequence
$$d \colon \bZ \rightarrow \bZ$$
as follows. Define
$$d(1)=m_1, \quad d(2)=m_2.$$
For $i \ge 2$, if $d(i)$ is defined and $d(i)>0$, define
$$d(i+1)=\delta d(i)-d(i-1).$$
\begin{lemma}\label{termination}
There exists $k \ge 3$ such that $d(k-1)>0$ and $d(k) \le 0$.
\end{lemma}
\qed

\begin{definition}
Let $k \in \bN$ be the number defined by Lemma~\ref{termination}. Define
$d(k+1)=-d(k-1)<0, d(k+2)=-d(k) \ge 0,$ and
$$m_1':=d(k-1)>0, \quad m_2':=-d(k) \ge 0.$$
\end{definition}

If $\delta=1$ then $k=3$ and we have
$$d(1)=m_1>0, \quad d(2)=m_2>0, \quad d(3)= m_2-m_1 \le 0,$$
$$d(4)=-m_2 <0, \quad d(5)=m_1-m_2  \ge 0.$$
In this case we define $d(i)$ for $i \in \bZ$ by requiring $d(i+5)=d(i)$ for each $i \in \bZ$.
Then
$$d(i-1)+d(i+1)=d(i) \mbox{ for } i \equiv 0,1,2 \bmod 5.$$

If $\delta>1$ then we define $d(i)$ for $i \le 0$ and $i \ge k+3$ by requiring that
$$d(i-1)+d(i+1)=\delta d(i) \mbox{ for } i \neq k+1,k+2.$$

\begin{lemma} If $\delta>1$ then $d(i) \ge 0$ for $i \neq k,k+1$.
\end{lemma}
\qed

Define a sequence
$$f \colon \bZ \rightarrow \bZ^2$$
as follows.
$$f(0)=(0,1), \quad f(1)=-(1,0), \quad f(2)=-(0,1), \quad f(3)=(1,0).$$
If $\delta=1$ define
$$f(4)=f(3)- f(2) = (1,1)$$
and define $f(i)$ for $i \in \bZ$ by requiring $f(i+5)=f(i)$ for each $i \in \bZ$.
If $\delta \ge 2$ define $f(i)$ for $i \in \bZ$ by requiring
$$f(i-1)+f(i+1)=\delta f(i) \mbox{ for } i \neq 1,2.$$

\begin{lemma}
If $\delta > 1$ then $f(i) \in \bZ^2_{> 0}$ for all $i \neq 0,1,2,3$.
\end{lemma}
\qed

We define a sequence $L_i \in \Cl(\bX)$ for $i \in \bZ$ as follows.
Define
$$L_1=[S_2] \in \Cl(\bX), \quad L_2=[S_1] \in \Cl(\bX).$$
Define $L_i$ for $3 \le i \le k$ by
$$L_{i-1}+L_{i+1}=\delta L_i \mbox{ for } 2 \le i \le k-1.$$
Define
$$L_{k+1}=-L_{k-1}, \quad L_{k+2}=-L_k.$$

If $\delta = 1$ then $k=3$ and
$$L_1=[S_2], \quad L_2=[S_1], \quad L_3= L_2 - L_1, \quad L_4=-L_2, \quad L_5= L_1-L_2.$$
We define $L_i$ for $i \in \bZ$ by requiring $L_{i+5}=L_i$ for each $i \in \bZ$.
Then
$$L_{i-1}+L_{i+1}=\delta L_i \mbox{ for } i \equiv -1,0,1 \bmod 5.$$

If $\delta>1$ define $L_i$ for $i\le 0$ and $i \ge k+3$ by requiring that
$$L_{i-1}+L_{i+1}=\delta L_i \mbox{ for } i \neq 1,2,k,k+1.$$

We define a sequence $F_i \in \Gamma(\bX,L_i)$ for $i \in \bZ$.
Define
$$F_0=y_1,\quad F_1=x_2, \quad F_2=x_1, \quad F_3=y_2.$$
We have $F_i \in \Gamma(\bX,L_i)$ for $i=1,2$ by the definitions, and for $i=0,3$ by the definition (\ref{action}) of the $\Gamma$-action and Lemma~\ref{classgrouplemma}.
Note that the equations of $W$ can be rewritten
$$F_0F_2=z^{d(1)}+F_1^{\delta}u^{-f(1)}=u^{-f(1)}(z^{d(1)}u^{f(1)}+F_1^{\delta})$$
$$F_1F_3=z^{d(2)}+F_2^{\delta}u^{-f(2)}=u^{-f(2)}(z^{d(2)}u^{f(2)}+F_2^{\delta})$$
Define $F_i$ for $4 \le i \le k$ by
$$F_{i-1}F_{i+1}=z^{d(i)}u^{f(i)}+F_i^{\delta} \mbox{ for } 3 \le i \le k-1.$$
Define $F_{k+1}$ by
$$F_{k-1}F_{k+1}=z^{-d(k)}F_k^{\delta}+u^{f(k)}=z^{-d(k)}(z^{d(k)}u^{f(k)}+F_k^{\delta})$$
and $F_{k+2}$ by
$$F_kF_{k+2}=z^{-d(k+1)}F_{k+1}^{\delta}+u^{f(k+1)}=z^{-d(k+1)}(z^{d(k+1)}u^{f(k+1)}+F_{k+1}^{\delta})$$

If $\delta = 1$ then $k=3$ and one checks that $F_5=F_0$. We define $F_i$ for all $i \in \bZ$ by requiring $F_{i+5}=F_i$ for all $i \in \bZ$.
If $\delta > 1$ then we define $F_i$ for $i \le -1$ and $i \ge k+3$ by requiring
$$F_{i-1}F_{i+1}=z^{d(i)}u^{f(i)}+F_i^{\delta} \mbox{ for } i \neq 1,2,k,k+1.$$

We can rewrite the equations defining the $F_i$ as follows:
$$F_{i-1}F_{i+1}=q_iF_i^{\delta}+r_i \mbox{ for } i \in \bZ$$
where $q_i,r_i \in k[z,u_1,u_2]$ are defined as follows.
If $\delta =1$ then
$$q_1=u_1, \quad q_2=u_2, \quad q_3=z^{m_1-m_2}, \quad q_4=z^{m_2}, \quad q_5=1$$
$$r_1=z^{m_1}, \quad r_2=z^{m_2}, \quad r_3=u_1, \quad r_4=u_1u_2, \quad r_5=z^{m_1-m_2}u_2$$
and $q_{i+5}=q_i$, $r_{i+5}=r_i$ for each $i \in \bZ$.
If $\delta >1$ then
$$q_i=\left\{\begin{array}{cc} u_1 & i=1 \\ u_2 & i=2 \\ z^{-d(k)}=z^{m_2'} & i=k \\ z^{-d(k+1)}=z^{m_1'} & i=k+1 \\ 1 & \mbox{otherwise.} \end{array} \right.$$
and
$$r_i=\left\{\begin{array}{cc} z^{d(1)}=z^{m_1} & i=1 \\ z^{d(2)}=z^{m_2} & i=2 \\ u^{f(k)} & i = k \\ u^{f(k+1)} & i=k+1 \\ z^{d(i)}u^{f(i)} & \mbox{otherwise.} \end{array} \right.$$
Note in particular that $q_i,r_i$ are coprime monomials in $k[z,u_1,u_2]$ for each $i$.
Moreover, we have
$$q_{i-1}q_{i+1}r_i^{\delta}=r_{i-1}r_{i+1} \mbox{ for all } i \in \bZ$$

\begin{remark}
In the terminology of \cite{FZ02} the algebra
$$k[u_1^{\pm 1}, u_2^{\pm 2}, z^{\pm 1}][\{F_i \ | \ i \in \bZ\}]$$
is a \emph{cluster algebra} of rank $2$, with clusters $$\{(F_i,F_{i+1})\}_{i \in \bZ}$$ and coefficient group $\bP \simeq \bZ^3$ the group of monomials in $z,u_1,u_2$.
In the notation of \cite{FZ02}, Example~2.5, p.~502,
the exchange relations are given by $b=c=\delta$ and coefficients $q_i,r_i$ as defined above.
In particular the case $\delta=1$ is a cluster algebra of finite type associated to the root system $A_2$.
\label{cluster}
\end{remark}

\begin{proposition}
For each $i \in \bZ$ we have $F_i \in \Gamma(\bX,L_i)$.
\end{proposition}
\begin{proof}
(Cf. \cite{M02}, Theorems~3.10, 3.12, and 3.13.)

From the definitions it is clear that $F_i$ is a rational section of $L_i$ for each $i \in \bZ$. We must show that it is a regular global section, or, equivalently, $F_i \in \Gamma(W^o,\cO_W)$.

We say $f,g \in \Gamma(W^o,\cO_W)$ are \emph{coprime on $\bX$} if $(f=g=0) \subset \bX$ has codimension $2$.

Define statements $A_i,B_i,C_{i,i+1}$ for $i \in \bZ$ as follows.
\begin{itemize}
\item[($A_i$)] $F_i \in \Gamma(\bX,L_i)$.
\item[($B_i$)] $F_i$ and $zu_1u_2$ are coprime on $\bX$.
\item[($C_{i,i+1}$)] $F_i$ and $F_{i+1}$ are coprime on $\bX$.
\end{itemize}

The following implications are immediate from the equations
$$F_{i-1}F_{i+1}=q_iF_i^{\delta}+r_i.$$
\begin{enumerate}
\item For each $i \in \bZ$,
$$(A_{i-1}, A_i, A_{i+1} \mbox{ and } B_i) \Rightarrow (C_{i-1,i} \mbox{ and } C_{i,i+1}).$$
\item For each $i \in \bZ \setminus \{1,2,k,k+1\}$, we have
\begin{enumerate}
\item
$$(A_{i-1}, A_i, A_{i+1} \mbox{ and } C_{i-1,i}) \Rightarrow B_{i-1}$$
\item
$$(A_{i-1}, A_i, A_{i+1} \mbox{ and } C_{i,i+1}) \Rightarrow B_{i+1}$$
\end{enumerate}
\end{enumerate}
Indeed, for (1) we have
$$(F_{i-1}=F_{i}=0) \subset (F_{i}=r_i=0) \subset (F_i=zu_1u_2=0).$$
Similarly for $B_{i,i+1}$.
For (2a), we have
$$(F_{i-1}=zu_1u_2=0) \subset (F_{i-1}=q_iF_i^{\delta}=0)=(F_{i-1}=F_i=0)$$
using $q_i=1$ for $i \neq 1,2,k,k+1$. Case (2b) is similar.

We show
$$(A_{i-2},A_{i-1},A_i,A_{i+1},B_{i-1},B_i \mbox{ and } C_{i-1,i}) \Rightarrow (A_{i-3} \mbox{ and } A_{i+2}).$$
for each $i \in \bZ$.
We have
\begin{equation}
F_iF_{i+2}=q_{i+1}F_{i+1}^{\delta}+r_{i+1} \in \Gamma(W^o,\cO_W).
\label{reg}
\end{equation}

Also, let $\equiv$ denote congruence modulo $F_i$ in $\Gamma(W^o,\cO_W)$, then
$$q_{i-1}F_{i-1}^{\delta}(F_iF_{i+2})=q_{i-1}F_{i-1}^{\delta}(q_{i+1}F_{i+1}^\delta +r_{i+1})$$
$$=q_{i-1}(q_{i+1}(F_{i-1}F_{i+1})^{\delta}+F_{i-1}^{\delta}r_{i+1})=q_{i-1}(q_{i+1}(q_iF_i^{\delta}+r_i)^{\delta}+F_{i-1}^{\delta}r_{i+1})$$
$$\equiv (q_{i-1}q_{i+1}r_i^{\delta})+q_{i-1}F_{i-1}^{\delta}r_{i+1} = r_{i-1}r_{i+1}+q_{i-1}F_{i-1}^{\delta}r_{i+1}$$
$$=r_{i+1}(q_{i-1}F_{i-1}^{\delta}+r_{i-1}) = r_{i+1}(F_{i-2}F_i) \equiv 0.$$
So $q_{i-1}F_{i-1}^{\delta}F_{i+2} \in \Gamma(W^o ,\cO_W)$. Now $F_i$ and $q_{i-1}F_{i-1}^{\delta}$ are coprime on $\bX$, so by \eqref{reg} we have $F_{i+2}$ is regular in codimension $1$ on $W^o$. Hence $F_{i+2}$ is regular on $W^o$ because $W$ is normal. Similarly $F_{i-3}$ is regular on $W^o$.

We verify $A_i,B_i$ for $i=0,1,2,3$ and $C_{i,i+1}$ for $i=0,1,2$ by direct computation. Now by induction using the implications above we prove
$A_i$ for $i \le k+2$, $B_i$ for $i \le k$ and $C_{i,i+1}$ for $i \le k-1$.

We claim that $B_{k+1}$ and $B_{k+2}$ hold. (Note that the implication (2b) was not established for $i=k,k+1$ so cannot be used here.)
Observe that the loci
$$(z=0),(u_1=0),(u_2=0) \subset \bX$$
are irreducible divisors. So it suffices to prove that $(F_k=0)$ and $(F_{k+1}=0)$ do not contain any of these divisors.
But by Lemma~\ref{restrict_to_H} below the intersection $(z=u_1=u_2=0)$ of these divisors is not contained in $(F_k=0)$ or $(F_{k+1}=0)$. This establishes the claim.

Finally we deduce that $A_i,B_i,C_{i,i+1}$ hold for all $i \in \bZ$ by induction.
\end{proof}

\begin{lemma}\label{restrict_to_H}
Define a sequence $g \colon \{1,2,\ldots,k+1\} \rightarrow \bZ$ by $g(1)=0$, $g(2)=1$, and $g(i+1)+g(i-1)=\delta g(i)$ for $2 \le i \le k$.
(Then $g(i) \ge 0$, and $g(i)>0$ unless $i=1$ or $\delta=1$ and $i=k+1=4$.)
Then
$$
\begin{array}{rcll}
(z=0)|_X       & = & l_1        +   l_2                   +  C                               \\
(F_0=0)|_X     & = &               m_1l_2                 +  m_1C                            \\
(F_1=0)|_X     & = &               m_2l_2                                                    \\
(F_i=0)|_X     & = & d(i-1)l_1                            +  m_2g(i-1)C                     \mbox{ for } 2 \le i \le k \\
(F_{k+1}=0)|_X & = &               m_2'l_2                +  (m_2g(k)+m_2')C                 \\
(F_{k+2}=0)|_X & = &               (m_1'+\delta m_2')l_2  +  (m_2g(k+1)+ m_1'+\delta m_2')C.
\end{array}
$$
In particular, set-theoretically we have
$$(F_k=F_{k+1}=0) \cap X = C$$
and
$$(F_{k-1}=F_{k+2}=0) \cap X \subseteq C,$$
with equality for $k \neq 3$.
\end{lemma}

\begin{proof}
(Cf. \cite{M02}, (3.2), p.~172.)
The description of $(z=0)|_X$ and $(F_i=0)|_X$ for $0 \le i \le 3$ are verified by direct computation.
We establish the description of $(F_i=0)|_X$ for $0 \le i \le k+2$ by induction using the equations
$$F_{i-1}F_{i+1} = q_iF_i^{\delta} + r_i \equiv q_iF_i^{\delta} \bmod (u_1,u_2) \mbox{ for $3 \le i \le k+1$.}$$
The final statements follow using $l_1 \cap l_2 = \emptyset$.
\end{proof}

\subsection{Construction of the contraction and flip}\label{contraction}

Define a sequence $c \colon \bZ \rightarrow \bZ$ as follows.
Define
$$c(1)=a_1, \quad c(2)=m_2-a_2.$$
Define $c(i)$ for $3 \le i \le k$ by
$$c(i-1)+c(i+1)=\delta c(i) \mbox{ for } 2 \le i \le k-1$$
We also define
$$c(k+1)=-c(k-1), \quad c(k+2)=-c(k).$$
If $\delta=1$ then $k=3$ so
$$c(1)=a_1, \quad c(2)=m_2-a_2, \quad c(3)=m_2-a_2-a_1,$$
$$c(4)=a_2-m_2, \quad c(5)=a_1+a_2-m_2.$$
We define $c(i)$ for $i \in \bZ$ by requiring $c(i+5)=c(i)$.
If $\delta>1$ we define $c(i)$ for $i \in \bZ$ by requiring
$$c(i-1)+c(i+1)=\delta c(i) \mbox{ for } i \neq k+1,k+2.$$
(Compare with the definition of the sequence $d(i)$, $i \in \bZ$.)

\begin{definition}
Define $a_1'$ by
$$a_1' \equiv c(k-1) \bmod m_1', \quad 1 \le a_1' \le m_1'$$
and if $m_2'=-d(k)>0$ define $a_2'$ by
$$a_2' \equiv c(k) \bmod m_2', \quad 1 \le a_2' \le m_2'.$$
\end{definition}

Define
$$W'=(x_1'y_1'=z^{m_1'}{x_2'}^{\delta}+u_1', \quad x_2'y_2'=z^{m_2'}{x_1'}^{\delta}+u_2') \subset \bA^5_{x_1',x_2',y_1',y_2',z} \times \bA^2_{u_1',u_2'}$$
and
$$\Gamma'=\{\gamma'=(\gamma_1',\gamma_2') \ | \ {\gamma_1'}^{m_1'}={\gamma_2'}^{m_2'} \} \subset \bG_m^2.$$
Define an action of $\Gamma'$ on $W'$ by
$$\Gamma' \ni \gamma' \colon (x_1',x_2',y_1',y_2',z,u_1',u_2') \mapsto$$ $$(\gamma_1'x_1',\gamma_2'x_2',\gamma_1'^{-1}y_1',\gamma_2'^{-1}y_2',\gamma_1^{c(k-1)}\gamma_2^{c(k)}z, u_1',u_2').$$
Define
$${W'}^o=W' \setminus (x_1'=x_2'=0)$$
and
$$\bX'=({W'}^o)/\Gamma'.$$

Write
$$U_1'=(x_2' \neq 0) \subset \bX', \quad U_2'=(x_1' \neq 0) \subset \bX'$$
Then $\bX'=U_1' \cup U_2'$,
$$U_i'=(\xi_i'\eta_i'={\zeta_i'}^{m_i'}+u_i') \subset \bA^3_{\xi_i',\eta_i',\zeta_i'}/\textstyle{\frac{1}{m_i'}(1,-1,a_i')} \times \bA^2_{u_1',u_2'}$$
for each $i=1,2$, and the glueing is given by
$$U_1' \supset (\xi_1' \neq 0) = (\xi_2' \neq 0) \subset U_2',$$
$${\xi_1'}^{m_1'}={\xi_2'}^{-m_2'}, \quad {\xi_1'}^{-c(k-1)}\zeta_1'={\xi_2'}^{-c(k)}\zeta_2'.$$

Write
$$X'=(u_1'=u_2'=0) \subset \bX',$$
$$C'=(y_1'=y_2'=z=u_1'=u_2'=0) \subset \bX',$$
and
$$P_i'=(x_i'=y_1'=y_2'=z=u_1'=u_2'=0) \subset \bX'$$
for each $i=1,2$.
Then $X'$ is a toric surface, $C' \subset X'$ is a proper smooth rational curve, and $P_1',P_2' \in C'$ are the torus fixed points of $X'$.
We also write $$l_i'=(x_i'=z=u_1'=u_2'=0) \subset X'$$ for each $i=1,2$.

Define
$$\bY=\Spec(\Gamma(\bX',\cO_{\bX'})) = \Spec(\Gamma({W'}^o,\cO_{W'})^{\Gamma'}) \rightarrow \bA^2_{u_1',u_2'}.$$

Note that $(x_1'=x_2'=0) \subset W'$ has codimension $2$, and $W'$ is normal. Thus the coordinate ring of the affine variety $\bY$ is given by
$$\Gamma(\bY,\cO_{\bY})=\Gamma({W'}^o,\cO_{W'})^{\Gamma'}=\Gamma(W',\cO_{W'})^{\Gamma'}=(R'/I')^{\Gamma'}$$
where
$$R':=k[x_1',x_2',y_1',y_2',z,u_1',u_2']$$
and
$$I':=(x_1'y_1'-(z^{m_1'}{x_2'}^{\delta}+u_1'), \, x_2'y_2'-(z^{m_2'}{x_1'}^{\delta}+u_2')) \subset R'.$$

We have a natural morphism
$$\pi' \colon \bX' \rightarrow \bY.$$

Define a birational toric morphism
$$\bA^2_{u_1,u_2} \rightarrow \bA^2_{u_1',u_2'}$$
via
$$u_1' \mapsto u^{f(k+1)}, \quad u_2' \mapsto u^{f(k)}.$$
Define a morphism
$$\pi \colon \bX \rightarrow \bY \times_{\bA^2_{u_1',u_2'}} \bA^2_{u_1,u_2}$$
over $\bA^2_{u_1,u_2}$ using the following diagram

$$ \xymatrix{ \Gamma(\bY,\O_{\bY})= \Gamma(W',\O_{W'})^{\Gamma'} \ar[d]^{\pi^*} \ar@{^{(}->}[r] & \Gamma(W',\O_{W'}) \ar[d] & \ar[l] \ar[dl]^{\phi} R' \\
   \Gamma(\bX,\O_{\bX})=\text{Cox}(\bX)^{\Gamma} \ar@{^{(}->}[r] & \text{Cox}(\bX) &  } $$

by $$\phi(x_1',x_2',y_1',y_2',z)=(F_k,F_{k+1},F_{k+2},F_{k-1},z)$$

Note that by construction the equations of $W'$ correspond to the relations $F_{k-1}F_{k+1}=\cdots$ and $F_kF_{k+2}=\cdots$ between the $F_{k-1},F_k,F_{k+1},F_{k+2}$.
Also, the group $\Gamma'=\Hom(M_{\Gamma'},\bG_m)$ is identified with $\Gamma=\Hom(M_{\Gamma},\bG_m)$ via
$$M_{\Gamma} \stackrel{\sim}{\longrightarrow} \Cl(\bX), \quad \bar{e}_1,\bar{e}_2 \mapsto [S_1], [S_2]$$
and
$$M_{\Gamma'} \stackrel{\sim}{\longrightarrow} \Cl(\bX), \quad \bar{e}_1,\bar{e}_2 \mapsto L_{k}, L_{k+1}.$$
Here we are using the case $i=k$ in the second formula of:

\begin{lemma}
For $1 \le i \le k+1$ we have
$$-K_{\bX}=c(i)L_{i+1}-c(i+1)L_i$$
and an identification
$$\bZ^2 / \bZ (-d(i+1),d(i))  \stackrel {\sim}{\longrightarrow} \Cl(\bX), \quad \bar{e}_1, \bar{e}_2 \mapsto L_{i}, L_{i+1}.$$
\end{lemma}
\begin{proof}
Cf. \cite{M02}, Propositions~3.14 and 3.15, p.~174.
\end{proof}

It follows that $\pi$ is well defined.

\begin{lemma}\label{contractingGmaction}
There exists a $\bG_m$-action on $\bX$, $\bA^2_{u_1,u_2}$, $\bX'$, $\bY$, and $\bA^2_{u_1',u_2'}$ such that
\begin{enumerate}
\item The morphisms
$\bX \rightarrow \bA^2_{u_1,u_2}$, $\bY \rightarrow \bA^2_{u_1',u_2'}$, $\bA^2_{u_1,u_2} \rightarrow \bA^2_{u_1',u_2'}$, $\pi$, and $\pi'$ are $\bG_m$-equivariant, and
\item $\bG_m$ fixes the points $0 \in \bA^2_{u_1,u_2}$, $0 \in \bA^2_{u_1',u_2'}$, and $0 \in \bY$ and in each case acts with positive weights on the corresponding maximal ideal in the affine coordinate ring.
\end{enumerate}
\end{lemma}
\begin{proof}
We define the $\bG_m$ action on $\bX \rightarrow \bA^2_{u_1,u_2}$ by
$$\lambda \ni \bG_m \colon (x_1,x_2,y_1,y_2,z,u_1,u_2) \mapsto (x_1,x_2,\lambda^{m_1}y_1,\lambda^{m_2}y_2,\lambda z, \lambda^{m_1}u_1,\lambda^{m_2}u_2).$$
Recall the equalities $F_0=y_1$, $F_1=x_2$, $F_2=x_1$, and $F_3=y_2$. Using the equations
$$F_{i-1}F_{i+1}=q_iF_i^{\delta}+r_i \mbox{ for } i \in \bZ$$
we find that $F_i$ is a $\bG_m$-eigenfunction for each $i \in \bZ$.
(The equations are seen to be homogeneous for the $\bG_m$-action by induction on $i$ using the equalities $q_{i-1}q_{i+1}r_i^{\delta}=r_{i-1}r_{i+1}$.)
Moreover, the weight of $F_i$ for $2 \le i \le k+2$ is given by
$$\wt(F_i)=g(i-1)m_2 \mbox{ for } 2 \le i \le k,$$
$$\wt(F_{k+1})=m_2g(k)+m_2',$$
and
$$\wt(F_{k+2})=m_2g(k+1)+m_1'+\delta m_2'.$$
Cf. Lemma~\ref{restrict_to_H}.
Recall that $x_1'=F_k$, $x_2'=F_{k+1}$, $y_1'=F_{k+2}$, $y_2'=F_{k-1}$, and $u_1'=u^{f(k+1)}$, $u_2'=u^{f(k)}$.
So we obtain a compatible $\bG_m$-action on $\bX' \rightarrow \bA^2_{u_1',u_2'}$, and an induced action on $\bY \rightarrow \bA^2_{u_1',u_2'}$.
We observe that $\bG_m$ acts with positive weights on $\bA^2_{u_1,u_2}$, $\bA^2_{u_1',u_2'}$, and $\bY$ as claimed.
(Note that if $k=3$ then $\bG_m$ acts with weight $0$ on $y_2'=F_2=x_1$, but no power of $y_2'$ is $\Gamma'$-invariant.)
\end{proof}

\begin{proposition}
Assume $d(k)=0$.
Then $m_1'=\delta=\gcd(m_1,m_2)$, $c(k)=-1$, and we have an identification
$$\bY = (\xi\eta=\zeta^{m_1'}+u_1') \subset \bA^3_{\xi,\eta,\zeta}/\textstyle{\frac{1}{m_1'}(1,-1,a_1')} \times \bA^2_{u_1',u_2'}.$$
\end{proposition}
\begin{proof}
(Cf. \cite{M02}, Theorem~4.5, p.~178.) Mori's proof works verbatim.
\end{proof}

\begin{remark}
In particular, in the case $d(k)=0$ the family $\bY \rightarrow \bA^2_{u_1',u_2'}$ is independent of $u_2'$, that is, it is the pullback of a family
$\bY' \rightarrow \bA^1_{u_1'}$ under the projection $\bA^2_{u_1',u_2'} \rightarrow \bA^2_{u_1'}$. For the moment we will keep the redundant variable $u_2'$ in order to treat both cases $d(k)=0$ and $d(k)<0$  simultaneously.
\end{remark}

\begin{proposition} \label{universalflip}
Assume $d(k)<0$.
The morphism $$\pi' \colon \bX' \rightarrow \bY$$ is a proper birational morphism with exceptional locus $C'={\pi'}^{-1}(0)$.
Moreover
$$K_{\bX'} \cdot C' = \frac{\delta}{m_1'm_2'} > 0,$$ or, equivalently, $\delta=c{m}'_1{m}'_2 - {m}'_1{a}'_2 -{m}'_2 {a}'_1$, where $-c$ is the self-intersection of the proper transform of $C'$ in the minimal resolution of $X'$.
\end{proposition}

\begin{proof}
(Cf. \cite{M02}, Theorem~4.7, p.~178.)

The morphism $\Pi'$ is clearly birational.

We have
$$(x_i')^a(y_j')^b, (x_i')^az^c \in \Gamma(\bY,\cO_{\bY})$$
for some $a,b,c \in \bN$, for each $i=1,2$ and $j=1,2$.
Recalling that
$$C':=(y_1'=y_2'=z=u_1'=u_2'=0) \subset \bX',$$
we deduce $(\pi')^{-1}(0)=C'$. It follows that $\pi'$ is proper over a neighbourhood of $0 \in \bY$,
and hence over $\bY$ using the existence of a contracting $\bG_m$-action on $\pi' \colon \bX' \rightarrow \bY$, see Lemma~\ref{contractingGmaction}.

We compute
$$K_{\bX'} \cdot C' = K_{X'} \cdot C' = \frac{\delta}{m_1'm_2'}$$
using the toric description of the section $X'$.

Observe that $z \in \Gamma(\bX',-K_{\bX'})$ and $C'$ is the only proper curve contained in $\bB':=(z=0) \subset \bX'$. (Indeed the fibers of
$$\bB' \rightarrow \bA^2_{u_1',u_2'}$$
are obtained from the special fiber
$$\bB'_0=(z=0)|_{X'} = l_1'+C'+l_2'$$
(the toric boundary of $X'$) by smoothing a subset of the two nodes.) Now since $K_{\bX'}$ is relatively ample over $\bY$ (using $K_{\bX'} \cdot C > 0$, $C=(\pi')^{-1}(0)$, and the contracting $\bG_m$-action) it follows that the exceptional locus of $\pi'$ equals $C'$.
\end{proof}

\begin{proposition}
\begin{enumerate}
\item Assume $d(k)<0$.
The morphism $\pi$ is a proper birational morphism with exceptional locus the codimension 2 set $$E:=(F_k=F_{k+1}=0).$$
We have
$$\pi(E)=\{0\} \times_{\bA^2_{u_1',u_2'}} \bA^2_{u_1,u_2} \subset \bY \times_{\bA^2_{u_1',u_2'}} \bA^2_{u_1,u_2}$$
or explicitly
$$\pi(E)=\left\{ \begin{array}{cc} (x_1'=x_2'=y_1'=y_2'=z=u_1u_2=0)  & \mbox{ if $k \ge 4$}\\ (x_1'=x_2'=y_1'=y_2'=z=u_2=0) & \mbox{ if $k=3$}.\\
\end{array}\right.$$
\item Assume $d(k)=0$.
The morphism $\pi$ is a proper birational morphism with exceptional locus the irreducible divisor
$$E:=(F_{k+1}=0)\subset X.$$
We have
$$\pi(E)=(\zeta={\xi}^{\delta}+u^{f(k)}=0) \subset \bY \times_{\bA^2_{u_1',u_2'}} \bA^2_{u_1,u_2}$$
\end{enumerate}
\end{proposition}
\begin{proof}
(Cf. \cite{M02}, Theorem~4.3, p.~177.)

It is easy to see that $\pi$ is birational.

Now consider the base change $$\bY \times_{\bA^2_{u_1',u_2'}} \bA^2_{u_1,u_2}$$ and the morphism $\pi \colon X \rightarrow \bY \times_{\bA^2_{u_1',u_2'}} \bA^2_{u_1,u_2}$.
Let $0 \in \bY \times_{\bA^2_{u_1',u_2'}} \bA^2_{u_1,u_2}$ denote the origin.

If $d(k)<0$ then the element $(x_i')^a(y_j')^b \in R'/I'$ is $\Gamma'$-invariant for some $a,b \in \bN$, for each $i=1,2$ and $j=1,2$.
If $d(k)=0$ then the elements
$$(x_1')^a, (y_1')^a, x_2'y_2' \in R'/I'$$
are $\Gamma'$ invariant for some $a \in \bN$.
It follows that
$$\pi^{-1}(0):=(x_1'=x_2'=y_1'=y_2'=u_1=u_2=0) \subset$$
$$(x_1'=x_2'=u_1=u_2=0) \cup (y_1'=y_2'=u_1=u_2=0).$$
Now
$$C=(x_1'=x_2'=u_1=u_2=0)$$
and
$$C \supseteq (y_1'=y_2'=u_1=u_2=0)$$
by Lemma~\ref{restrict_to_H} (using $F_{k-1}=y_2'$, $F_k=x_1'$, $F_{k+1}=x_2'$, $F_{k+2}=y_1'$).
Thus $\pi^{-1}(0) = C$.

It follows that there is an open neighbourhood $0 \in U \subset \bY \times_{\bA^2_{u_1',u_2'}} \bA^2_{u_1,u_2}$ such that $\pi$ is proper over $U$.
Moreover, by Lemma~\ref{contractingGmaction} we have a contracting $\bG_m$-action on $\pi \colon (C \subset \bX) \rightarrow (0 \in \bY \times_{\bA^2_{u_1',u_2'}} \bA^2_{u_1,u_2})$.  It follows that $\pi$ is proper.

Suppose $d(k)<0$.
Then $L_k \cdot C= \frac{d(k)}{m_1 m_2} < 0$ and $L_{k+1} \cdot C=\frac{-d(k-1)}{m_1 m_2} <0$ by \cite{M02}, Proposition~3.14(2), p.~174.
Thus $-L_{k}$ and $-L_{k+1}$ are relatively ample over some open neighbourhood $0 \in U \subset \bY \times_{\bA^2_{u_1',u_2'}} \bA^2_{u_1,u_2}$ (because $\pi^{-1}(0)=C$).
Using the contracting $\bG_m$-action again we deduce that $-L_k$ and $-L_{k+1}$ are relatively ample over $\bY \times_{\bA^2_{u_1',u_2'}} \bA^2_{u_1,u_2}$.
Now $E:=(F_k=F_{k+1}=0)$ where $F_k \in \Gamma(\bX,L_{k})$, $F_{k+1} \in \Gamma(\bX,L_{k+1})$. So the exceptional locus of $\pi$ is contained in $E$.
We have $E \cap X = C = \pi^{-1}(0)$ by Lemma~\ref{restrict_to_H}.
Also, $E$ is preserved by the contracting $\bG_m$-action (because $E=(F_k=F_{k+1}=0)$ and $F_k,F_{k+1}$ are eigenfunctions for the $\bG_m$-action).
Moreover, each fiber of $E$ over $\bA^2_{u_1,u_2}$ is either empty or irreducible of dimension $1$.
Thus every fiber of $E \rightarrow \bA^2_{u_1,u_2}$ is algebraically equivalent to a multiple of $C$. It follows that the exceptional locus of $\pi$ is equal to $E$.

Similarly, if $d(k)=0$, then $-L_{k+1}$ is relatively ample over $\bY \times_{\bA^2_{u_1',u_2'}} \bA^2_{u_1,u_2}$, and $E:=(F_{k+1}=0)$ satisfies $E \cap X = C$ by Lemma~\ref{restrict_to_H} (using $m_2'=-d(k)=0$). Also, the locus $E$ is preserved by the contracting $\bG_m$-action.
It follows that each fiber of $E \rightarrow \bA^2_{u_1,u_2}$ is a proper curve, such that each component is algebraically equivalent to a multiple of $C$.
So the exceptional locus of $\pi$ equals $E$. Finally, since the Milnor number of a $\bQ$-Gorenstein smoothing of a singularity of type $\frac{1}{m^2}(1,ma-1)$ equals $0$, we have $b_2(\bX_t) = b_2(\bX_0)=1$ for all $t \in \bA^2_{u_1,u_2}$. It follows that the fibers of $E \rightarrow \bA^2_{u_1,u_2}$ are irreducible, and hence that $E$ is irreducible.

See \cite{M02}, Theorem~4.5, p.~178 for the identification of $\pi(E)$.

\end{proof}

We identify the exceptional locus $E$ explicitly in the case $d(k)<0$ of flipping contractions.

\begin{proposition}\label{exceptional_locus_local_eq}
Assume $d(k)<0$. The locus
$$E:=(F_k=F_{k+1}=0) \subset \bX$$
is given in the charts $U_1,U_2$ as follows.
We have $E \subset (u_1=0)$ for $k=3$ and $E \subset (u_1u_2=0)$ for $k>3$.
Write $E_i=E \cap (u_i=0)$ for $i=1,2$. Then
$$E_1 \cap U_1 =(q_1^{m_2}+u_2p_1^{\delta m_1 - m_2}) \subset (u_1=0) \subset U_1$$
$$E_1 \cap U_2 =(\eta_2=u_1=0) \subset U_2$$
and similarly, for $k>3$,
$$E_2 \cap U_1 =(\eta_1=u_2=0) \subset U_1$$
$$E_2 \cap U_2=(q_2^{m_1}+u_1p_2^{\delta m_2 - m_1}) \subset (u_2=0) \subset U_2$$
In particular, the nonempty fibers of $E \rightarrow \bA^2_{u_1,u_2}$ are irreducible of dimension $1$.
\end{proposition}
\begin{proof}
First observe that $F_k=F_{k+1}=0$ implies $u^{f(k)}=u^{f(k+1)}=0$ using the equations $F_{k-1}F_{k+1}=\cdots$ and $F_kF_{k+2}=\cdots$.
Hence $E \subset (u_1=0)$ if $k=3$ and $E \subset (u_1u_2=0)$ if $k>3$.

Recall that $E \cap (u_1=u_2=0) = C$ by Lemma~\ref{restrict_to_H}.
In Proposition~\ref{glueingfamilies} we show that, for $k>3$, the restriction of the morphism $\pi \colon \bX \rightarrow \bY \times_{\bA^2_{u_1',u_2'}} \bA^2_{u_1,u_2}$ to the open subset
$$(u_1 \neq 0) \subset \bA^2_{u_1,u_2}$$
is identified with the restriction of the morphism $$\pi^2 \colon \bX^2 \rightarrow \bY \times_{\bA^2_{u_1',u_2'}} \bA^2_{u_1^2,u_2^2}$$ associated to the data
$$(m^2_1,m^2_2,a^2_1,a^2_2):=(m_2,\delta m_2 - m_1, m_2-a_2, \delta a_2 - (m_1-a_1))$$
to the open subset
$$(u_2^2 \neq 0) \subset \bA^2_{u_1^2,u_2^2}.$$
Thus it suffices to describe $E \cap (u_2 \neq 0)$.

Let $0 \neq s \in (u_1=0) \in \bA^2_{u_1,u_2}$. Consider the fibers
$$U_{1,s}=(\xi_1\eta_1=\zeta_1^{m_1}) \subset \bA^3_{\xi_1,\eta_1,\zeta_1}/\textstyle{\frac{1}{m_1}}(1,-1,a_1) = \bA^2_{p_1,q_1}/ \textstyle{\frac{1}{m_1^2}(1,m_1a_1-1)}$$
$$U_{2,s}=(\xi_2\eta_2=\zeta_2^{m_2}+u_2) \subset \bA^3_{\xi_2,\eta_2,\zeta_2}/\textstyle{\frac{1}{m_2}(1,-1,a_2)}.$$
Define closed irreducible curves $k_1,k_2,h,F \subset \bX_s$ as follows.
$$k_1=(\xi_1=\zeta_1=0) = (p_1=0) \subset U_{1,s}$$
$$k_2=(\xi_2=0) \subset U_{2,s}$$
$$h \cap U_{1,s} = (\eta_1=\zeta_1 = 0) = (q_1=0) \subset U_{1,s}$$
$$h \cap U_{2,s} = (\zeta_2 = 0) \subset U_{2,s}$$
$$F \cap U_{1,s} = (q_1^{m_2}+u_2p_1^{\delta m_1 - m_2}=0) \subset U_{1,s}$$
$$F \cap U_{2,s} = (\eta_2 = 0) \subset U_{2,s}$$
Now, similarly to Lemma~\ref{restrict_to_H}, one checks
$$
\begin{array}{rcll}
(z=0)|_{\bX_s}       & = & k_1        +   h                                                 \\
(F_0=0)|_{\bX_s}    & = &               m_1h                                             \\
(F_1=0)|_{\bX_s}     & = &               k_2                                                    \\
(F_i=0)|_{\bX_s}     & = & d(i-1)k_1                            +  g(i-1)F                     \mbox{ for } 2 \le i \le k \\
(F_{k+1}=0)|_{\bX_s} & = &               m_2'h                +  g(k)F
\end{array}
$$
Thus, noting that $k_1 \cap h \subset F$, we obtain
$$F=(F_k=F_{k+1}=0) \cap \bX_s = E \cap \bX_s.$$
\end{proof}

\subsection{The universal family of $K$ negative surfaces}

Let $j \in \bZ$ be such that $d(j),d(j+1)>0$. Explicitly, if $\delta>1$ then we require $j \neq k-1,k,k+1$ if $d(k)<0$ and $j \neq k-1,k,k+1,k+2$ if $d(k)=0$. If $\delta=1$ we require $j \equiv 0,1 \bmod 5$ if $d(k)<0$ and $j \equiv 1 \bmod 5$ if $d(k)=0$. We will call the allowed values of $j$ \emph{admissible}.
In what follows the superscripts $j$ are understood modulo $5$ in the case $\delta=1$.

Define
$$m^j_{1}=d(j), \quad m^j_{2}=d(j+1), \quad a^j_{1}=c(j), \quad a^j_{2}=m^j_{2}-c(j+1).$$
Thus $m^1_i=m_i$ and $a^1_i=a_i$ for each $i=1,2$.

\begin{lemma}
We have $1 \le a^j_i \le m^j_i$ and $\gcd(a^j_i,m^j_i)=1$ for each $i=1,2$ and all admissible $j$.
Moreover,
$$\delta^j:=m^j_1a^j_2+m^j_2a^j_1-m^j_1m^j_2 = \delta >0$$
and
$$\Delta^j:=(m^j_1)^2+(m^j_2)^2-\delta^j m^j_1 m^j_2 = \Delta >0.$$
\end{lemma}
\qed

We apply the above constructions to the data $(m^j_1,m^j_2,a^j_1,a^j_2)$, and denote the result by the same notations decorated by a superscript $j$.

\begin{proposition} \label{glueingfamilies}
For each admissible $j$ the  family ${\bX'}^j \rightarrow \bA^2_{{u'}^j_1,{u'}^j_2}$ is identified with $\bX' \rightarrow \bA^2_{u_1',u_2'}$ by
identifying variables with the same name (forgetting the superscript), that is,
$$({x'}^j_1,{x'}^j_2,{y'}^j_1,{y'}^j_2,{z'}^j,{u'}^j_1,{u'}^j_2)=(x_1',x_2',y_1',y_2',z',u_1',u_2').$$
In particular we have an induced identification $\bY^j=\bY$.

Let $j,j+1$ be admissible.
Then we have an identification
$$\bX^j \supset (u^j_1 \neq 0) = (u^{j+1}_2 \neq 0) \subset \bX^{j+1}$$
given by
$$z^j=z^{j+1},$$
$$(u^j_1,u^j_2)=((u^{j+1}_2)^{-1}, u^{j+1}_1(u^{j+1}_2)^{\delta}),$$
and
$$F^j_i =\left\{ \begin{array}{cc} F^{j+1}_{i-1} & i \neq 2 \\ (u^{j+1}_2)^{-1} F^{j+1}_{i-1} & i=2.\\
\end{array}\right.$$ where the subscript $i$ is understood modulo $5$ if $\delta=1$.

Using the above identifications we obtain the following:
\begin{enumerate}
\item A toric surface $M$ (only locally of finite type for $\delta > 1$) together with a toric birational morphism
$$p \colon M \rightarrow \bA^2_{u_1',u_2'}$$
such that $M$ is the union of the toric open affine sets $\bA^2_{u^j_1,u^j_2}$ for $j$ admissible and the restriction of $p$ to $\bA^2_{u^j_1,u^j_2}$ is the morphism
$$\bA^2_{u^j_1,u^j_2} \rightarrow \bA^2_{u_1',u_2'}.$$
\item A proper birational morphism
$$\Pi \colon \bU \rightarrow \bY \times_{\bA^2_{u_1',u_2'}} M$$
such that $\Pi^{-1}(\bA^2_{u^j_1,u^j_2})=\bX^j$ and the restriction of $\Pi$ to $\bX^j$ is the morphism
$$\pi^j \colon \bX^j \rightarrow \bY \times_{\bA^2_{u_1',u_2'}} \bA^2_{u^j_1,u^j_2}.$$
\end{enumerate}
\end{proposition}

\qed

The toric birational morphism $p \colon M \rightarrow \bA^2_{u_1',u_2'}$ can be described explicitly as follows.
We identify $\bA^2_{u_1',u_2'}$ as the toric variety associated to the cone $\sigma=\bR^2_{\ge 0}$ in the lattice $N=\bZ^2$.
We define a fan $\Sigma$ in $N$ with support $|\Sigma|$ contained in $\sigma$ such that the toric birational morphism corresponds to the map of fans
$\Sigma \rightarrow \sigma$.
If $\delta=1$, define $v_1 =(1,0)$,  $v_2=(1,1)$, and $v_{3}=(0,1)$.
Let $\Sigma$ be the fan with cones $\{0\}$, $\bR_{\ge 0}v_i$, $i=1,2,3$, and $\langle v_{i},v_{i+1} \rangle_{\bR_{\ge 0}}$, $i=1,2$.
(So in the case $\delta =1$ the morphism $p$ is the blowup of $0 \in \bA^2_{u_1',u_2'}$.)
If $\delta \ge 2$, define primitive vectors $v_i \in N$ for $i \in \bZ \setminus \{0\}$ by
$$v_1=(1,0), \quad v_2=(\delta,1),$$
$$v_{i+1}+v_{i-1}=\delta v_i \mbox{ for  $i \ge 2$},$$
and
$$v_{-1}=(0,1), \quad v_{-2}=(1,\delta),$$
$$v_{-(i+1)}+v_{-(i-1)}=\delta v_{-i} \mbox{ for  $i \ge 2$}.$$
(So $v_{-i}=T(v_i)$ where $T(x_1,x_2):= (x_2,x_1)$.)
Let $\Sigma$ be the fan  with cones $\{0\}$, $\bR_{\ge 0} v_i$ for $i \in \bZ \setminus \{0\}$, $\langle v_i,v_{i+1}\rangle_{\bR_{\ge 0}}$ for $i \ge 1$, and $\langle v_{-i},v_{-(i+1)} \rangle_{\bR_{\ge 0}}$ for $i \ge 1$.

For $\delta \ge 2$ define
$$\xi = (\delta + \sqrt{\delta^2-4})/2 \ge 1$$
That is, $\xi$ is the larger root of the quadratic equation
$$x^2-\delta x + 1 =0.$$
Equivalently, $\xi$ is given by the infinite Hirzebruch-Jung continued fraction
$$\xi=  \delta -
\frac{1}{\delta - \frac{1}{\ddots }}$$

Then the support $|\Sigma|$ of the fan $\Sigma$ is the region
$$|\Sigma| = \{ (x_1,x_2) \in \bR^2 \ | \ x_1,x_2 \ge 0 \mbox { and either } x_1 > \xi x_2 \mbox{ or } x_2 > \xi x_1 \} \cup \{0\}$$
or, equivalently,
$$|\Sigma|= \{ (x_1,x_2) \in \bR^2 \ | \ x_1,x_2 \ge 0 \mbox { and } x_1^2-\delta x_1x_2+x_2^2 > 0 \} \cup \{0\}.$$

Let $\bE \subset \bU$ denote the exceptional locus of $\Pi$ (thus $\bE \cap \bX^j = E^j$).
Write $\bB=(z=0) \subset \bU$ and $\bB'=(z=0) \subset \bX'$.

\begin{theorem}\label{universality}
Let $f \colon (C \subset \X) \rightarrow (Q \in \Y)$ be an extremal neighborhood of type $k1A$ or $k2A$.
Let $(Q \in \Y) \rightarrow (0 \in \bA^1_t)$ be the morphism determined by a general element $t \in m_{\Y,Q} \subset \cO_{\Y,Q}$.
Assume that the general fiber $X_s$ of $\X \rightarrow (0 \in \bA^1_t)$ satisfies $b_2(X_s)=1$.
Then there is a universal family
$$\Pi \colon \bU \rightarrow \bY \times_{\bA^2_{u_1',u_2'}} M, \quad p \colon M \rightarrow \bA^2_{u_1',u_2'}$$
as in Proposition~\ref{glueingfamilies} and a morphism
$$g \colon (0 \in \bA^1_t) \rightarrow M$$
such that $p(g(0))=0 \in \bA^2_{u_1',u_2'}$ and the diagram
$$(C \subset \X) \rightarrow (Q \in \Y) \rightarrow (0 \in \bA^1_t)$$
is isomorphic to the pullback of the diagram
$$(\bE_{g(0)} \subset \bU) \rightarrow ((0,g(0)) \in \bY \times_{\bA^2_{u_1',u_2'}} M) \rightarrow (g(0) \in M)$$
under $g$.

If $f$ is a flipping contraction and
$$f^+ \colon (C^+ \subset \X^+) \rightarrow (Q \in \Y)$$
is the flip of $f$ then the diagram
$$(C^+ \subset \X^+) \rightarrow (Q \in \Y) \rightarrow (0 \in \bA^1_t)$$
is identified with the pullback of the diagram
$$(C' \subset \bX') \rightarrow (0 \in \bY) \rightarrow (0 \in \bA^2_{u_1',u_2'})$$
under $p \circ g$.

Furthermore, if $D \in |-K_\X|$ is a general element, then we may assume that $D$ is the pullback of $\bB$.
(Then, if $f$ is a flipping contraction, the strict transform $D' \in |-K_{\X^+}|$ is the pullback of $\bB'$.)
\end{theorem}

\begin{remark}
We note that the cone $\langle v_1,v_2 \rangle_{\bR_{\ge 0}}$ corresponds to an extremal neighborhood of type $k2A$ with singularities $$\textstyle{\frac{1}{d(k-1)^2}}(1,d(k-1)(d(k-1)-c(k-1)) -1), \textstyle{\frac{1}{d(k-2)^2}}(1,d(k-2)c(k-2) -1).$$
If it is of flipping type, then $\delta d(k-1) -d(k-2) <0$. In this case one of the singularities of the surface $X'$, which is the fiber over $(0,0)$ of $\bX' \rightarrow \bA^2_{u_1',u_2'}$, is $$\textstyle{\frac{1}{d(k-1)^2}}(1,d(k-1) (d(k-1)-c(k-1)) -1)$$ (see \S \ref{contraction}). Similarly for the cone $\langle v_{-1},v_{-2} \rangle_{\bR_{\ge 0}}$. In this way, both singularities of $X'$ can be read from these ``initial" $k2A$. (It is possible that $d(k-1)=1$. In that case the corresponding point in $X'$ is smooth.)
\label{initialk2A}
\end{remark}

\begin{lemma}\label{locallyversal}
The flat family of surfaces $\bU \rightarrow M$ induces a versal $\bQ$-Gorenstein deformation of each fiber.
The flat families of pairs
$$(\bU,\bB) \rightarrow M$$
and
$$(\bX',\bB') \rightarrow \bA^2_{u_1',u_2'}$$
induce a versal $\bQ$-Gorenstein deformation of each fiber.
\end{lemma}
\begin{proof}
By openness of versality and the existence of the contracting $\bG_m$-actions (Lemma~\ref{contractingGmaction}) it suffices to consider the toric fibers.
We will establish the statement for pairs first. We prove it for $(\bU,\bB) \rightarrow M$, the other case being identical.
Let $\bA^2_{u_1,u_2} \subset M$ be a toric chart and $(X,B)$ the toric fiber over $0 \in \bA^2_{u_1,u_2}$.

The logarithmic tangent sheaf $T_{X}(-\log B)$ is the sheaf of infinitesimal automorphisms of the pair $(X,B)$. Since $(X,B)$ is toric we have an identification
$$T_{X}(-\log B) = \cO_{X} \otimes_{\bZ} N \simeq \cO_{X}^{\oplus 2}$$
where $N$ is the lattice of $1$-parameter subgroups of the torus acting on $X$.
Thus $H^i(T_{X}(-\log B))=0$ for $i>0$. It follows that the local-to-global map from the versal $\bQ$-Gorenstein deformation space of the pair $(X,B)$ to the product of the $\bQ$-Gorenstein versal deformation spaces of the singularities $(P_i \in X,B)$, $i=1,2$ of the pair is an isomorphism. Recall also that a singularity $(P \in Z)$ of type $\frac{1}{m^2}(1,ma-1)$ has versal $\bQ$-Gorenstein deformation space a smooth curve germ $(0 \in \bA^1_t)$, with versal family
$$(0 \in (\xi\eta=\zeta^m+t) \subset (\bA^3_{\xi,\eta,\zeta}/\textstyle{\frac{1}{m}}(1,-1,a)) \times \bA^1_t).$$
Moreover, writing $P \in B \subset Z$ for the toric boundary, the forgetful map
\begin{equation}\label{forgetdivisor}
\Def^{\QG}(P \in Z,B) \rightarrow \Def^{\QG}(P \in Z)
\end{equation}
from the versal $\bQ$-Gorenstein deformation space of the pair $(P \in Z, B)$ to the versal $\bQ$-Gorenstein deformation space of $(P \in Z)$ is an isomorphism.
(Indeed $B$ is Cartier on the index one cover $$(P' \in Z') \rightarrow (P \in Z)$$ given by
$$(P' \in Z') = \textstyle{\frac{1}{n}}(1,-1)=(\xi\eta=\zeta^n) \subset \bA^3_{\xi,\eta,\zeta},$$
and the $\bQ$-Gorenstein deformation is induced by an equivariant deformation of the index one cover. So there are no obstructions to lifting $B$ to the deformation. Moreover, since $B=(\zeta=0) \subset Z$, it's easy to see that every infinitesimal embedded deformation of $B$ is induced by an infinitesimal automorphism of $Z$. Combining it follows that the forgetful map (\ref{forgetdivisor}) is an isomorphism.)
Now it is clear from the explicit charts $U_1$, $U_2$ for the restriction of the family $\bU \rightarrow M$ to $\bA^2_{u_1,u_2} \subset M$ that $(\bU,\bB) \rightarrow M$ induces a versal $\bQ$-Gorenstein deformation of the pair $(X,B)$.

Finally, we prove versality for $\bU \rightarrow M$. By the above, it suffices to show that $H^1(T_X)=H^2(T_X)=0$ for each toric fiber $X$, so that the local-to-global map from the versal $\bQ$-Gorenstein deformation space of $X$ to the product of the versal $\bQ$-Gorenstein deformation spaces of the singularities of $X$ is an isomorphism.
Let $\mathfrak{H}$ be the orbifold stack associated to $X$, that is, $\mathfrak{H}$ is the Deligne--Mumford stack with coarse moduli space $X$ determined by the local smooth covers $$(0 \in \bA^2) \rightarrow (0 \in \bA^2/G)$$ at each quotient singularity of $X$.
Write
$$B=B_0+B_1+B_2=C+l_1+l_2$$
and $\mathfrak{B} \subset \mathfrak{H}$, $\mathfrak{B}_i \subset \mathfrak{H}$ for the associated stacks.
Then we have an exact sequence of sheaves on $\mathfrak{H}$
$$0 \rightarrow T_{\mathfrak{H}}(-\log \mathfrak{B}) \rightarrow T_{\mathfrak{H}} \rightarrow \bigoplus_i \cN_{\mathfrak{B}_i/\mathfrak{H}} \rightarrow 0$$
(where $\cN_{Z/W}=\cHom(\cI_{Z/W}/\cI_{Z/W}^2,\cO_Z)$ denotes the normal bundle of a closed embedding $Z \subset W$).
Pushing forward via the natural morphism $q \colon \mathfrak{H} \rightarrow X$, we obtain an exact sequence
$$0 \rightarrow T_{X}(-\log B) \rightarrow T_X \rightarrow \bigoplus \cN'_{B_i/X} \rightarrow 0.$$
where $\cN'_{B_i/X}:=q_*\cN_{\mathfrak{B}_i,\mathfrak{H}}$.
So, since $H^i(T_X(-\log B))=0$ for $i>0$ (see above) and $l_1,l_2$ are affine, we have $H^1(T_X)=H^1(\cN'_{C/X})$ and $H^2(T_X)=0$.
Finally, we compute that the degree of the line bundle $\cN'_{C/X}$ on $C \simeq \bP^1$ equals $-1$, so $H^1(\cN'_{C/X})=0$ as required.
In general, if $L$ is a line bundle on a proper orbifold curve $\mathfrak{C}$, with coarse moduli space $q \colon \mathfrak{C} \rightarrow C$, and orbifold points $$Q_i \in \mathfrak{C} \simeq (0 \in [\bA^1_x/\mu_{r_i}]),$$ where
$$\mu_{r_i} \ni \zeta \colon x \mapsto \zeta x$$
and the action of $\mu_{r_i}$ on the fiber $L|_{Q_i}$ is via the character $\zeta \mapsto \zeta^{-w_i}$, $0 \le w_i < r_i$, then
$$\deg q_*L = \deg L - \sum \frac{w_i}{r_i}.$$
Recall that we have isomorphisms
$$(P_i \in C \subset X) \simeq (0 \in (q_i=0) \subset \bA^2_{p_i,q_i}/\textstyle{\frac{1}{m_i^2}}(1,m_ia_i-1)).$$
Hence the degree of the line bundle $\cN'_{C/X}$ on $C \simeq \bP^1$ is given by
$$\deg \cN'_{C/X} = C^2-\frac{m_1a_1-1}{m_1^2}-\frac{m_2a_2-1}{m_2^2}$$
Using the formula
$$C^2=-\Delta/m_1m_2$$
(see \cite{M02}, Proof of Proposition~2.6), we find $\deg \cN'_{C/X} = -1$. Thus $H^1(\cN'_{C/X})=0$ as required.
\end{proof}

\begin{proof}[Proof of Proposition~\ref{k1Adegeneratestok2A}] \label{k1Atok2A}
The computation of the singularities of $X^i$ and the intersection number $K_{X^i} \cdot C^i$ are straightforward toric calculations.
(Note that it suffices by symmetry to treat the case $i=1$. Now the formulas for $\Delta$ and $\Omega$ agree with those in \S\ref{k2Adeformation} and in particular we see that the vector $v^1$ lies in $N$ and is primitive.)

Now by the description of the versal deformation
$$\pi \colon \bX \rightarrow \bY \times_{\bA^2_{u_1',u_2'}} \bA^2_{u_1,u_2}$$
of $f^i_0 \colon X^i \rightarrow Y$ over $\bA^2_{u_1,u_2}$ given in \S\ref{contraction} and the identification of the fibers over $(u_1u_2=0)$ (in particular the description of the exceptional curve in Proposition~\ref{exceptional_locus_local_eq}) we find that the general fiber over the appropriate component of the boundary $(u_1u_2=0)$ is isomorphic to the given contraction $f_0$. Indeed, it suffices to check that the strict transform of the exceptional curve for the deformation of $f^i_0$ intersects the same component of the exceptional locus of the minimal resolution of the $\frac{1}{m_1^2}(1,m_1a_1-1)$ singularity as in the original $k1A$ surface contraction $f_0 \colon X \rightarrow Y$.
By the description of the local equation of the exceptional curve in Proposition~\ref{exceptional_locus_local_eq} this component corresponds to the ray
$$\rho=\bR_{\ge 0} \textstyle{\frac{1}{m_1^2}(m_2,m_0)}$$
in the toric description of the minimal resolution of the singularity.
(To see this, note first that $m_0+m_2=\delta m_1$ by definition (see Proposition~\ref{k1Ainvariants}), and that the number $\delta$ coincides with the invariant of the same name for the $k2A$ surface $X^i$ defined in \S2, by the computation of $K_{X^i} \cdot C^i$.
Now observe that the local equation of $C'$ is homogeneous with respect to the grading determined by the primitive generator of $\rho$. So the strict transform meets the interior $\bG_m \subset \bP^1$ of the associated exceptional divisor.)
This agrees with the description of the original $k1A$ surface $X$ given in Proposition~\ref{k1Ainvariants}.
\end{proof}

\begin{proof}[Proof of Theorem~\ref{universality}]
By construction (and the classification of $k2A$ extremal neighborhoods), the hyperplane section $X_0=X$ of a $k2A$ extremal neighbourhood occurs as a fiber of a universal family $\bU \rightarrow M$ over a torus fixed point of $M$. See \cite{M02}, Remark~2.3, p.~159.
By Proposition~\ref{k1Adegeneratestok2A}, the hyperplane section of a $k1A$ extremal neighborhood occurs as a fiber of a universal family over an interior point of a $1$-dimensional toric stratum of $M$. Moreover, in each case, the deformation of the fiber given by the family $\bU \rightarrow M$ is a versal $\bQ$-Gorenstein deformation. This establishes the existence of $g$.

In the case of a flipping contraction it is immediate that the pullback $(C' \subset \X') \rightarrow (Q \in \Y)$ of $\bX' \rightarrow \bY$ is a proper birational morphism with exceptional locus $C'$ such that $K_{\X'}$ is $\bQ$-Cartier and $K_{\X'} \cdot C' > 0$.
(Note here that the proper birational morphism $\bX' \rightarrow \bY$ has exceptional locus $C' \subset {\bX'}_0$, see Proposition~\ref{universalflip}, and $p(g(0))=0$.) So it is the flip of $f$.

For the final statement, note first that $H^0(-K_{\X})\rightarrow H^0(-K_X)$ is surjective by Kawamata--Viehweg vanishing and cohomology and base change (using $-K_{\X}$ relatively ample), so the restriction of a general element $D$ of $|-K_{\X}|$ is a general element of $|-K_X|$.
We observe that every pair $(X,B)$ of a $k1A$ or $k2A$ surface $X$ together with a general element $B \in |-K_X|$ occurs as a fiber of $(\bU,\bB) \rightarrow M$.
(See Lemma~\ref{restrict_to_H} and the proof of Proposition~\ref{exceptional_locus_local_eq} for the the description of the fibers of $\bB=(z=0) \subset \bU$ in the $k2A$ and $k1A$ cases respectively.)
The family of pairs $(\bU,\bB)\rightarrow M$ gives a versal $\bQ$-Gorenstein deformation of each fiber by Lemma~\ref{locallyversal}. So we may assume $D$ is the pullback of $\bB$.
\end{proof}

\begin{corollary}\label{existenceofterminalantiflips}
Let $f^+_0 \colon (C^+ \subset X^+) \rightarrow (Q \in Y)$ be a partial resolution of a cyclic quotient singularity such that $X^+$ has T-singularities and $K_{X^+}$ is relatively ample, and $(C^+ \subset \X^+)/(0 \in \bA^1_t)$ a one parameter $\bQ$-Gorenstein smoothing of the germ $(C^+ \subset X^+)$. Then the morphism $f^+_0$ extends to a proper birational morphism
$$f^+ \colon (C^+ \subset \X^+) \rightarrow (Q \in \Y)$$
over $(0 \in \bA^1_t)$.

Assume that the general fiber $X^+_s$, $0<|s|\ll 1$ satisfies $b_2(X^+_s)=1$.
Then $C^+$ is irreducible and $X^+$ has only cyclic quotient singularities of type $\frac{1}{m^2}(1,ma-1)$ for some $m,a \in \bN$ with $\gcd(a,m)=1$. The $3$-fold $\X^+$ has $\bQ$-factorial terminal singularities.

The morphism $f^+_0$ is toric. Choose a toric structure and let $B^+ \subset X^+$ denote the toric boundary.
Let $P_i \in X^+$, $i=1,2$ be the torus fixed points, $m_i$ the index of $P_i$, and define $\delta \in \bN$ by $K_{X^+} \cdot C^+ = \delta/m_1m_2$.

Assume that the exceptional locus of $f^+$ equals $C^+$.

Then there exists an antiflip
$$f \colon (C \subset \X) \rightarrow (Q \in \Y)$$
of $f^+$ such that $\X$ has terminal singularities if and only if the following conditions are satisfied.
\begin{enumerate}
\item For some choice of toric structure, there is a divisor $D \in \mbox{$|-K_{\X^+}|$}$ such that $D|_{X^+}=B^+$, the toric boundary of $X^+$.
\item If $\delta \ge 2$, define $\xi=(\delta+\sqrt{\delta^2-4})/2 \ge 1$. Then the axial multiplicities $\alpha_1,\alpha_2$ of the singularities of $\X^+$ satisfy $\alpha_1 > \xi \alpha_2$ or $\alpha_2 > \xi\alpha_1$. Equivalently, we have
$$\alpha_1^2 -\delta\alpha_1\alpha_2 + \alpha_2^2 > 0.$$
\end{enumerate}
\end{corollary}

\begin{remark}
In the case that $\X^+$ has fewer than two singularities, we define the axial multiplicities $\alpha_1,\alpha_2$ as follows.
Let $D$ be a general element of $|-K_{\X^+}|$. By condition (1) the restriction $D|_{X^+}=B^+$ is a choice of toric boundary for $X^+$.
Let $P_1,P_2 \in B^+$ be the singular points of $B^+$. If $P_i$ is a smooth point of $\X^+$, then the local deformation $(P_i \in D \subset \X^+) \rightarrow (0 \in \bA^1_t)$ of $(P_i \in B^+ \subset X^+)$ is of the form
$$(0 \in (\xi\eta=t^{\alpha}h(t)) \subset \bA^2_{\xi,\eta} \times \bA^1_t)$$
for some $\alpha \in \bN$ and convergent power series $h(t)$, $h(0) \neq 0$.
Equivalently, $P_i \in D$ is a Du Val singularity of type $A_{\alpha-1}$.
We define $\alpha_i=\alpha$. If $P_i$ is a singular point of $\X^+$ then the definition of the axial multiplicity is the usual one.
\end{remark}

\begin{proof}
The morphism $f^+_0 \colon X^+ \rightarrow Y$ extends to a morphism $f^+ \colon \X^+ \rightarrow \Y$ because $R^1{f^+_0}_*\cO_{X^+}=0$ (see \cite{KM92}, Proposition 11.4).

The morphism $f^+_0 \colon X^+ \rightarrow Y$ is a $P$-resolution of $Y$ \cite{KSB88}, Definition~3.8.
It is toric by \cite{KSB88}, Lemma~3.14. The singularities of $X^+$ are cyclic quotient singularities of type $\frac{1}{\rho m^2}(1,\rho ma -1)$ for some $\rho, m,a \in \bN$ with $\gcd(m,a)=1$ (by the existence of the $\bQ$-Gorenstein smoothing, see \cite{KSB88}, Proposition~3.10). The Milnor fiber $M$ of a $\bQ$-Gorenstein smoothing of such a singularity has Milnor number  $\mu:=b_2(M)=\rho-1$. Thus, if $r$ is the number of irreducible components of $C^+$ and the singularities $P_i \in X^+$ have invariants $\rho_i$, a Mayer--Vietoris argument gives
$$b_2(X^+_s)=r+\sum (\rho_i-1)$$
Hence $b_2(X^+_s)=1$ iff $C^+$ is irreducible and $\rho_i=1$ for each $i$.

The $3$-fold $\X^+$ has singularities of the form
$$(0 \in (\xi\eta=\zeta^m+t^{\alpha}) \subset (\bA^3_{\xi,\eta,\zeta}/\textstyle{\frac{1}{m}}(1,-1,a)) \times \bA^1_t).$$
The number $m \in \bN$ is the index of the singularity and the number $\alpha \in \bN$ is the \emph{axial multiplicity} of the singularity \cite{M88}, Definition--Corollary~1a.5, p.~140.
As already noted, these singularities are terminal, and analytically $\bQ$-factorial by \cite{K91}, 2.2.7. Thus $\X^+$ is terminal and $\bQ$-factorial.
(Note: In general, a one parameter $\bQ$-Gorenstein smoothing of a surface quotient singularity is terminal \cite{KSB88}, Corollary~3.6. The modern proof of this fact uses ``inversion of adjunction", see e.g. \cite{KM92}, Theorem~5.50(1), p.~174.)


Since $Q \in Y=(t=0) \subset \Y$ is a cyclic quotient singularity, it follows from the classification of non-semistable flips \cite{KM92}, Appendix that if $f \colon \X \rightarrow \Y$ is a terminal antiflip of $f^+$ then $f$ is semistable, that is, of type $k1A$ or $k2A$. Moreover $X=(t=0) \subset \X$ is normal and $\rho_1=\rho_2=1$ by Proposition~\ref{b2=1=>normal} because $b_2(X_s)=b_2(Y_s)=b_2(X^+_s)=1$ for $0<|s|\ll 1$.
Let $D \in |-K_{\X}|$ be a general element. Then, by Theorem~\ref{universality}, $f \colon (\X,D) \rightarrow \Y$ is the pullback of a universal family
$$\Pi \colon (\bU,\bB) \rightarrow \bY \times_{\bA^2_{u'_1,u'_2}} M, \quad p \colon M \rightarrow \bA^2_{u'_1,u'_2}$$
via a morphism $g \colon (0 \in \bA^1_t) \rightarrow M$, such that $p(g(0))=0$.

Let $D' \in |-K_{\X^+}|$ denote the strict transform of $D$.
Then $$(\X^+,D') \rightarrow \Y \rightarrow (0 \in \bA^1_t)$$ is the pull back of
$$(\bX',\bB') \rightarrow \bY \rightarrow \bA^2_{u_1',u_2'}$$ via the morphism $p \circ g$, and $D' \in |-K_{\X^+}|=|-K_{\X}|$ is a general element such that $D'|_{X^+}=B^+$, so (1) holds.

The morphism
$$p \circ g \colon \bA^1_t \rightarrow \bA^2_{u_1',u_2'}$$
is given by
$$u_1'=\lambda_1t^{\alpha_1}+\cdots, \quad u_2'=\lambda_2t^{\alpha_2}+\cdots$$
for some $\lambda_1,\lambda_2 \in k\setminus\{0\}$, $\alpha_1,\alpha_2 \in \bN$ (where $\cdots$ denotes higher order terms in $t$).
By the description of the family $\bX' \rightarrow \bA^2_{u_1',u_2'}$, we see that $\alpha_1,\alpha_2$ are the axial multiplicities of the singularities of $\X^+$.

Identify the fan $\Sigma$ of $M$ with its image in the cone $\sigma=\bR_{\ge 0}^2$ of $\bA^2_{u_1',u_2'}$  under the map of fans corresponding to $p$.
The point $g(0)$ lies in the toric boundary stratum of $M$ corresponding to the smallest cone $\tau$ of the fan of $M$ containing the point $(\alpha_1,\alpha_2)$. In particular $(\alpha_1,\alpha_2)$ lies in the support of the fan of $M$. Now the explicit description of the fan of $M$ above gives $(2)$.

Conversely, suppose the conditions (1) and (2) are satisfied.
Let $D \in |-K_{\X^+}|$ be general, so $D|_{X^+}=B^+$ is a choice of toric boundary for $X^+$ by (1).
Let
$$\Pi' \colon (\bX',\bB') \rightarrow \bY \rightarrow \bA^2_{u_1',u_2'}$$
be the deformation of
$$f^+ \colon (X^+,B^+) \rightarrow Y$$
constructed in \S\ref{contraction}.
By Lemma~\ref{locallyversal}, $$f^+ \colon (\X^+,D) \rightarrow \Y \rightarrow (0 \in \bA^1_t)$$ is the pullback of
$$\Pi' \colon (\bX',\bB') \rightarrow \bY \rightarrow \bA^2_{u_1',u_2'}$$ via a map $h \colon (0 \in \bA^1_t) \rightarrow \bA^2_{u_1',u_2'}$ such that $h(0)=0$.

The numerical data to build $\bU$ from $\bX'$ is the following. Let $$\textstyle{\frac{1}{{m'_1}^2}}(1,m'_1 {a'}_1 -1), \textstyle{\frac{1}{{m'_2}^2}}(1,{m'_2} {a'_2} -1)$$ be the singularities of $X^+ \subset \bX'$ such that $\frac{{m'_1}^2}{{m'_1}
{a'_1}-1}=[e_1,\ldots,e_{r_1}]$, $\frac{{m'_2}^2}{{m'_2} {a'_2}
-1}=[f_1,\ldots,f_{r_2}]$, and $$\frac{\Delta}{\Omega}=[f_{r_2},\ldots,f_{1},c,e_1,\ldots,e_{r_1}],$$ where $-c$
is the self-intersection of the proper transform of $C^+$ in the
minimal resolution of $X^+$. If a singularity (or both) is (are) actually smooth, then we set ${m'_i}={a'_i}=1$. Define $$\delta=c{m'_1} {m'_2} - {m'_1} {a'_2} - {m'_2} {a'_1}.$$ Define $m_1={m'_2}$, $a_1={m'_2}-{a'_2}$ if ${m'_2} \neq {a'_2}$, or $a_1=1$ else, and $m_2=\delta {m'_2} +{m'_1}$, $a_2=\frac{\delta + m_1 m_2 - a_1 m_2}{m_1}$. One can check that $0<a_2 <m_2$ and gcd$(m_2,a_2)=1$. If $\frac{{m}_2^2}{{m}_2 {a}_2-1}=[g_1,\ldots,g_{r_3}]$, then one can check $$[g_{r_3},\ldots,g_1,1, f_{r_2},\ldots,f_{1}]= \frac{\Delta}{\Omega}.$$ One can verify that this gives an ``initial" $k2A$ for a Mori sequence whose ``flipping" surface is $X^+$ (see Remark \ref{initialk2A}). For the other initial $k2A$, we define $m_2={m'_1}$, $a_2={m'_1}-{a'_1}$ if ${m'_1} \neq {a'_1}$, or $a_2=1$ else, and $m_1=\delta {m'_1} +{m'_2}$, $a_1=\frac{\delta + m_1 m_2 - a_2 m_1}{m_2}$. Notice that these are the only two possibilities for initial $k2A$ corresponding to $\bX'$.

By assumption (2) and the description of the toric morphism $p \colon M \rightarrow \bA^2_{u_1',u_2'}$, the morphism $h$ admits a lift $g \colon (0 \in \bA^1_t) \rightarrow M$ such that $h=p \circ g$. Indeed, as above, (2) is equivalent to requiring that the vector $(\alpha_1,\alpha_2)$ given by the vanishing orders of the components of $h$ lies in the image of the fan of $M$ in the cone $\sigma=\bR^2_{\ge 0}$ corresponding to $\bA^2_{u_1,u_2}$. This in turn is equivalent to the existence of the lift $g$. Now the pull back of
$\bU \rightarrow \bY \times_{\bA^2_{u_1',u_2'}} M \rightarrow M$ by $g$ is the desired terminal antiflip $f \colon \X \rightarrow \Y \rightarrow (0 \in \bA^1_t)$ of $f^+$.
\end{proof}

\begin{remark}\label{asflnsfgajf}
If we only impose condition (1) in Corollary~\ref{existenceofterminalantiflips} then there is an antiflip $f \colon \X \rightarrow \Y$ such that $\X$ has \emph{canonical} singularities.
This is an application of \cite{KM92}, Theorem~3.1, p.~561. Indeed, by (1), for a general divisor $D \in |-K_{\X^+}|$ the restriction $D|_{X^+}$ is a choice $B^+$ of toric boundary for $\X^+_0=X^+$.
Then, since $D|_{X^+}$ is a nodal curve, $D \in |-K_{\X^+}|$ is a normal surface with at worst Du Val singularities of type $A$ (at the nodes of $D_0$).
Moreover the contraction $\bar{D}=f_*D \in |-K_{\Y}|$ has at worst a Du Val singularity of type $A$. (More precisely, for $D$ general, we have Du Val singularities of type $A_{m_i\alpha_i-1}$ at $P_i \in D$ for $i=1,2$, where $m_i$ is the index and $\alpha_i$ the axial multiplicity of $P_i \in \X^+$, and a Du Val singularity of type $A_{m_1\alpha_1+m_2\alpha_2-1}$ at $Q \in \bar{D}$.) Now by \cite{KM92}, Theorem~3.1, the $\cO_{\Y}$-algebra
$$R(\Y,-K_{\Y}):=\bigoplus_{m \ge 0} \cO_{\Y}(m(-K_{\Y}))$$
is finitely generated, and
$$\X:= \Proj_{\Y} R(\Y,-K_{\Y})$$
has canonical singularities. The morphism $f \colon \X \rightarrow \Y$ is the antiflip of $f^+ \colon \X^+ \rightarrow \Y$.
\end{remark}

\begin{example}\label{asdvasfgafg}
This is an explicit example of a $\Q$-Gorenstein smoothing of a $X^+$ as in Corollary \ref{existenceofterminalantiflips} which has a nonterminal but canonical antiflip. Let $X$ be a surface with one singularity given by the $\Z/4\Z$-quotient of a simple
elliptic singularity (cf. \cite{K88}, Theorem 9.6(3)). The exceptional locus of its minimal
resolution $\widetilde{X}$ is a union of $4$ smooth rational
curves $E_1$, $E_2$, $E_3$, and $F$ so that the $E_i$ are
disjoint, each meets the curve $F$ transversally at a single
point, and $E_1^2=-2$, $E_2^2=-4$, $E_3^2=-4$, and $F^2=-3$.
Assume $X$ has a smooth rational curve $C$ such that its
proper transform in $\widetilde{X}$ is a $(-1)$-curve
intersecting only $E_1$ and transversally at one point. Notice that
$K_{X} \cdot C = -\frac{1}{2}$ and $C \cdot C
=-\frac{2}{5}$.

The surface $X$ has a $\Q$-Gorenstein smoothing $X \subset
\X \to (0 \in \bA_t^1)$ (cf. \cite{W11}). The corresponding singularity in $\X$ is canonical
by \cite{KSB88}, Theorem 5.1. Let $(Q \in Y \subset \Y)$ be the contraction of
$C \subset X \subset \X$ (the blowing down deformation; cf. \cite[11.4]
{KM92}). Then $(Q \in Y)$ is the cyclic quotient singularity $\frac{1}{24}(1,7)$. By
\cite{W11}, Proposition 3.1(4), the extremal nbhd $C \subset \X \to Q \in \Y$ is $\Q$-factorial. Therefore, if
it is of divisorial type, then $\Y$ would be
$\Q$-Gorenstein but $(Q \in \Y)$ is not a T-singularity. So, it is
flipping, and the flip is given by the
P-resolution $X^+ \to Y$ whose exceptional curve is a $\P^1$ passing through two singularities of type $\frac{1}{4}(1,1)$.
\label{canexample}
\end{example}

\section{Classification of extremal P-resolutions}\label{s3}

Fix integers $0<\Omega<\Delta$ with gcd$(\Delta,\Omega)=1$. Let $(Q \in Y)$ be the cyclic quotient singularity
$\frac{1}{\Delta}(1,\Omega)$. An \emph{extremal P-resolution} of $(Q \in Y)$ is a partial resolution $f^+_0 \colon (C^+ \subset X^+) \rightarrow (Q \in Y)$ such that $X^+$ has only \emph{Wahl singularities} ($=$ cyclic of the type $\frac{1}{m^2}(1,ma-1)$ with gcd$(m,a)=1$), the exceptional curve $C^+=\P^1$, and $K_{X^+}$ is relatively ample. They appear when we perform flips with the extremal neighborhoods of type $k1A$ or $k2A$ of the previous sections (see Corollary \ref{existenceofterminalantiflips}). The surface $X^+$ has at most two Wahl singularities $\frac{1}{m_i^2}(1,m_ia_i-1)$ ($i=1,2$) such that $\frac{m_1^2}{m_1
a_1-1}=[e_1,\ldots,e_{r_1}]$, $\frac{m_2^2}{m_2 a_2
-1}=[f_1,\ldots,f_{r_2}]$, and $$\frac{\Delta}{\Omega}=[f_{r_2},\ldots,f_{1},c,e_1,\ldots,e_{r_1}],$$ where $-c$
is the self-intersection of the proper transform of $C^+$ in the
minimal resolution of $X^+$. If a singularity (or both) is (are) actually smooth, then we set $m_i=a_i=1$. We define $$\delta= cm_1 m_2 - m_1 a_2 - m_2 a_1.$$
Then $\Delta= m_1^2 + m_2^2 + \delta m_1 m_2$,
$$K_{X^+} \cdot C^+=\frac{\delta}{m_1 m_2 } >0, \ \ \ \text{and} \ \ \
C^+ \cdot C^+= \frac{- \Delta}{m_1^2 m_2^2} <0.$$

\subsection{The continued fraction of an extremal P-resolution} \label{contfracpresol}

By \cite{KSB88}, Theorem 3.9, there is a natural bijection between P-resolutions of $(Q \in Y)$ and irreducible components of the formal deformation space $\Def(Q \in Y)$. In \cite{C89,S89}, Christophersen and Stevens prove that P-resolutions of $(Q \in Y)$ are in bijection with certain continued fractions representing zero. For extremal P-resolutions we have the following. Let $\frac{\Delta}{\Delta - \Omega} = [c_1,\ldots, c_s]$. Then the extremal P-resolutions of $(Q \in Y)$ are in bijection with pairs $1\leq \alpha<\beta\leq s$ such that $$ 0= [c_1,\ldots,c_{\alpha-1},c_{\alpha}-1,c_{\alpha+1},\ldots,c_{\beta-1},c_{\beta}-1,c_{\beta+1},\ldots,c_s]. $$ We label these P-resolutions by $[c_1,\ldots, \bar{c}_{\alpha},\ldots, \bar{c}_{\beta},\ldots, c_s]$. We now recall how to describe torically the P-resolution from the zero continued fraction. We follow \cite{A98}.

Let us define in $\Z^2$ the vectors $w^0=(0,1)$, $w^1=(1,1)$, and $$w^{i-1} + w^{i+1} = c_i w^i $$ for $i \in \{1,\ldots,s \}$. If $w^i=(x_i,y_i)$, then $\frac{y_i}{x_i}=[1,c_1,\ldots,c_{i-1}]$ for $2\leq i \leq s+1$. Note that $w^{s+1}=(\Delta,\Omega)$.

The fan $\Sigma$ of the toric surface $X^+$, where the toric birational morphism $f^+_0 \colon (C^+ \subset X^+) \rightarrow (Q \in Y)$ is the extremal P-resolution, has two cones $\tau^1= \langle u^1, u^2 \rangle$ and $\tau^2=\langle u^2,u^3 \rangle$ where $u^1=(1,0)$, $u^2=(-(m_2(m_2-a_2)-1),m_2^2)$, and $u^3=(-\Omega,\Delta)$; see \cite{A98} pp.7--8. The $w^{\alpha}, w^{\beta}$ which correspond to $\tau^1, \tau^2$ respectively in \cite{A98} (giving the ``roofs" of the cones) are $$w^{\alpha}=(m_2,m_2-a_2), \ \ \ w^{\beta}=((\Delta-m_2^2) / m_1,(\Omega-(m_2(m_2-a_2)-1))/m_1).$$ One can verify that $\delta = w^{\beta} \wedge w^{\alpha}$.

\begin{proposition}\label{contfracPres}
The extremal P-resolution $[c_1,\ldots, \bar{c}_{\alpha},\ldots, \bar{c}_{\beta},\ldots, c_s]$ has
$$\frac{m_2}{a_2}=[c_1,\ldots,c_{\alpha-1}], \ \ \ \frac{m_1}{a_1}=[c_s,\ldots,c_{\beta+1}],$$ (if $\alpha=1$ and/or $\beta=s$, the corresponding points are smooth) and either $\frac{\delta}{\epsilon}= [c_{\alpha+1},\ldots,c_{\beta-1}]$, if $\alpha +1 \neq \beta$, for some $0<\epsilon<\delta$, or $\delta= 1 \, , \text{if} \ \alpha +1 = \beta$.
\end{proposition}

\begin{proof}
Since $w^{\alpha}=(m_2,m_2-a_2)$, we have $\frac{m_2-a_2}{m_2}=[1,c_1,\ldots,c_{\alpha-1}]$. Hence $\frac{m_2}{a_2}=[c_1,\ldots,c_{\alpha-1}]$. By symmetry, we have $\frac{m_1}{a_1}=[c_s,\ldots,c_{\beta+1}]$.

We know that $\delta = w^{\beta} \wedge w^{\alpha}= x_{\beta}y_{\alpha}-x_{\alpha}y_{\beta}$. Then use Lemma \ref{partfrac} below with $[1,c_1,\ldots,c_r]$ and $\frac{y_i}{x_i}=[1,c_1,\ldots,c_{i-1}]=:[e_1,\ldots,e_i]=\frac{p_{i+1}}{q_{i+1}}$, setting $i=\alpha+2$ and $l=\beta-\alpha-1$ for the second equality in that lemma.
\end{proof}

\begin{lemma} \label{partfrac}
Let $[e_1,\ldots,e_r]$ be a continued fraction with $e_i\geq 1$ and $\frac{p_i}{q_i}=[e_1,\ldots,e_{i-1}]>0$ for all $i \in \{2,\ldots, r+1 \}$. Define $p_1=1$, $p_0=q_1=0$, and $q_0=-1$. Then
$$p_{i-1}q_{i}-p_{i}q_{i-1}=1 \ \ \ \text{and} \ \ \ \frac{p_{i-1}q_{i+l}-p_{i+l}q_{i-1}}{p_{i-1}q_{i+l-1}-p_{i+l-1}q_{i-1}}=[e_{i+l-1},\ldots,e_i] $$ for all $1 \leq i \leq r$, $1 \leq l \leq r+1-i$.
\end{lemma}

\begin{proof}
We have $p_{i-1}+p_{i+1}=e_i p_i$ and $q_{i-1}+q_{i+1}=e_i q_i$ for all $1 \leq i \leq r$. We now do induction on $i$ for the first equality, and induction on $l$ for the second.
\end{proof}

The following are all the extremal P-resolutions for cyclic quotient singularities $\frac{1}{\Delta}(1,\Omega)$ with $2 \leq \Delta \leq 45$, apart from the type $\frac{\Delta}{\Delta-1}=[\bar 2, 2,\ldots,2,\bar 2]$ (which corresponds to $\frac{1}{\Delta}(1,1)$). The underlined notation represents another extremal P-resolution for the same singularity.

{\small
$$\frac{7}{5}=[2,\overline 2,\overline 3] \ \ \frac{11}{8}=[2,\overline 2,3, \overline 2] \ \  \frac{13}{10}=[2,2,\overline 2,\overline 4] \ \ \frac{15}{11}=[\overline 2,\underline 2,3,\overline 2, \underline 2]$$ $$\frac{16}{7}=[\overline 3,\underline 2,\overline 2,\underline 3] \ \ \frac{19}{15}=[\overline 2,2,2,3,\overline 2,2] \ \ \frac{19}{12}=[2,\overline 3, \overline 2, 3] $$ $$\frac{21}{17}=[2,2,2,\overline 2,\overline 5] \ \ \frac{22}{17}=[2,2,\overline 2,4,\overline 2] \ \ \frac{23}{19}=[\overline 2,2,2,2,3,\overline 2,2] $$ $$\frac{24}{17}=[2,\overline 2,4,\overline 2,2] \ \ \frac{25}{16}=[\overline 2,3,2,\overline 2,3] \ \ \frac{27}{23}=[\overline 2,2,2,2,2,3, \overline 2,2]$$ $$ \frac{29}{13}=[\overline 3,2,2,\overline 2,4] \ \ \frac{31}{27}=[\overline 2,2,2,2,2,2,3,\overline 2,2] \ \ \frac{31}{26}=[2,2,2,2,\overline 2,\overline 6]$$ $$\frac{31}{24}=[2,2,\overline 2,4,2,\overline 2] \ \ \frac{34}{25}=[\overline 2,2,3,2,\overline 2,3] \ \ \frac{35}{31}=[\overline 2,2,2,2,2,2,2,3,\overline 2,2] $$ $$\frac{36}{23}=[2,\overline 3,\underline 2,\overline 2,\underline 4] \ \ \frac{37}{30}=[2,2,2,\overline 2,5,\overline 2] \ \ \frac{37}{27}=[2,2,\overline 3,\overline 2,4] $$ $$ \frac{39}{35}=[\overline 2,2,2,2,2,2,2,2,3,\overline 2,2] \ \ \frac{39}{29}=[2,\overline 2,3,2,2,2,2,2,2,2,\overline 2] $$ $$\frac{40}{31}=[\overline 2,2,\underline 2,4,\overline 2,2,\underline 2] \ \ \frac{40}{29}=[2,\overline 2,3,3,\overline 2,2] \ \ \frac{41}{18}=[3,2,\overline 2,3,\overline 3]$$ $$\frac{43}{39}=[\overline 2,2,2,2,2,2,2,2,2,3,\overline 2,2] \ \ \frac{43}{37}=[2,2,2,2,2,\overline 2,\overline 7] \ \ \frac{43}{34}=[\overline 2,2,2,3,2,\overline 2,3]$$ $$\frac{43}{33}=[2,2,\overline 2,5,\overline 2,2] \ \ \frac{45}{29}=[\overline 2,3,2,2,\overline 2,4] \ \ \frac{45}{19}=[3, \overline 2,3,\overline 2,3] $$
}

The purpose of the next two subsections is to prove

\begin{theorem} \label{atmost2}
A cyclic quotient singularity $\frac{1}{\Delta}(1,\Omega)$ can admit at most two distinct extremal P-resolutions.
\end{theorem}

\begin{theorem} \label{samedelta}
If a cyclic quotient singularity $\frac{1}{\Delta}(1,\Omega)$ admits two extremal P-resolutions, then the $\delta$'s are equal.
\end{theorem}

Theorem \ref{atmost2} says that there are at most two universal antiflips over a given$\frac{1}{\Delta}(1,\Omega)$, and Theorem \ref{samedelta} says that the $\delta$ of a universal antiflip is an invariant of $\frac{1}{\Delta}(1,\Omega)$. The proofs for both theorems are going to be completely combinatorial, using the continued fraction of an extremal P-resolution and the formula for $\delta$ in Proposition \ref{contfracPres}.

\subsection{At most two extremal P-resolutions over $\frac{1}{\Delta}(1,\Omega)$} \label{2extremalnbhd}

According to \cite{C89,S89}, $[v_1,\ldots,v_r]$ is a zero-continued fraction if and only if
there exists a triangulation of the $(r+1)$-gon with vertices
indexed by $\{0,\ldots,r\}$ such that the number of triangles
meeting at the $i$-th vertex is equal to $v_i$. We define $v_0$
as the number of triangles meeting at the $0$-th vertex. We
obviously have
\begin{equation}\label{zfbdfbafbaf}
v_0=3r-3-\sum_{i=1}^rv_i.
\end{equation}

\begin{definition}
We say that a sequence of integers
$$\aaa=\{a_1,\ldots,a_r\},\quad a_i>1\quad\hbox{\rm for any $i$},$$
is {\em of WW type} if there exists $1\le \alpha<\beta\le r$ such
that
\begin{equation}\label{ADHJSCGVDHMSVG}
[a_1,\ldots,a_\alpha-1,\ldots,a_\beta-1,\ldots,a_r]
\end{equation}
is a zero-continued fraction.
\end{definition}

To prove Theorem \ref{atmost2}, we proceed by induction on
$r$, and along the way we more or less classify sequences of $WW$
type.

A sequence $\aaa$ of $WW$ type gives rise to a triangulation that
corresponds to \eqref{ADHJSCGVDHMSVG}. We say that this
triangulation is {\em compatible} with~$\aaa$. The claim is that
there are at most two triangulations compatible with~$\aaa$. For
any triangulation, we call $\{v_0,\ldots,v_r\}$ the vertex
sequence. By~\eqref{zfbdfbafbaf}, all triangulations compatible
with $\aaa$ have the same $v_0$. We define $a_0=v_0$. Notice that
$a_0$ can be equal to one.

It is easy to see (e.g.~by induction) that the vertex sequence
always contains at least two $1$'s. So we have the following
possibilities:
\begin{enumerate}
\item[(A)] $a_0=v_0>1$, all vertex sequences compatible with
$\aaa$ have two~$1$'s, $\aaa$~is obtained by substituting these
$1$'s with $2$'s. \item[(B)] $a_0=v_0=1$.
\end{enumerate}

We can handle case (B) by induction. Indeed, in this case all
triangulations compatible with $\aaa$ contain a triangle $T$
formed by vertices $r$, $0$, and $1$. Therefore they define
triangulations of the $r$-gon obtained by chopping off the
triangle $T$. These triangulations are compatible with the
sequence $$a_0'=a_1-1,\  a_1'=a_2,\ \ldots,\  a_{r-2}'=a_{r-1},\
a_{r-1}'=a_r-1.$$ If $\aaa'$ contains two $1$'s then $a_1=a_r=2$.
This implies that $r=3$ and we are triangulating a square. This
case is obvious. If $\aaa'$ contains at most one $1$ then we are
done by induction.

Now we deal with case (A). First we inductively classify all
triangulations with exactly two $1$'s in the vertex sequence. If
$v_i=1$ then we can chop off a triangle with vertices $i-1,i,i+1$.
Unless $r=3$, a new triangulation will have exactly two $1$'s, one
of them at $i-1$ or $i+1$. We chop off the next vertex with $1$,
and proceed inductively. In other words, we can get any sequence
$\aaa$ of this type by the following procedure. Start with a
sequence $2,2,2$. Then do a recursion analogous to enumeration of
all $T$-singularities: add a new $2$ on one end and increase
another and by $1$. So, for example, we can do
$$[2,\bar 2,2]\to [3,\bar 2,2,2]\to [2,3,\bar 2,2,3]\to [2,2,3,\bar 2,2,4]\to \ldots$$
After doing several steps, fold a sequence around a circle and add
a new~$\bar 2$ between the ends. Now choose $v_0$ anywhere except
at overlined~$2$'s (which should be decreased to ones in the
vertex sequence). We are going to ignore $v_0$ from now on. In
particular, we will sometimes rotate indices if this will be
convenient.

Start with our sequence $\aaa=\{a_0,\ldots,a_r\}$ of type (A) and
let $s$ be the number of $i$'s such that $a_i>2$. Now take any
triangulation compatible with $\aaa$. Let $0< p<q\le r$ be the
indices such that $v_p=v_q=1$. By the discussion above, unless we
are triangulating a square, there exist unique indices $p',q'$
such that the $p$-th vertex is adjacent to the vertex $v_{p'}>2$,
and similarly for $q,q'$. We rotate, and if necessary reflect the
indices (in the dihedral group) so that, for the first
triangulation, $p=1$ and $p'=0$. We create two new sequences:
\begin{itemize}
\item $v_0=\alpha_0,\ldots,\alpha_{s-1}$ is a subsequence of
$\aaa$ of terms greater than $2$. \item $x_i$ for $i=0,\ldots,s-1$
is the number of $2$'s in $\aaa$ between the terms that correspond
to $\alpha_i$ and $\alpha_{i+1}$ increased by $1$.
\end{itemize} The indexing set here is $\bZ_s$.

Then the inductive procedure above allows us to compute
$\alpha_i$'s in terms of $x_i$'s. There are two cases. If $s=2l$
is even, we have
$$\alpha_i=\begin{cases} x_i& i=0,l\cr x_{-i}+2 & i\ne 0, l\end{cases}$$
If $s=2l+1$ is odd, we have
$$\alpha_i=\begin{cases} x_0& i=0\cr x_l & i=l+1 \cr x_{-i}+2 & i\ne 0, l+1\end{cases}$$
Now consider the second triangulation. We change notation, and let
$p,p',q,q'$ denote the data of the second triangulation.

\begin{example}
The sequence
$$\{5,\overline 2,2,2,2,3,\underline 2,2,7,2, \overline 2,3,2,2,2, \underline 2,5\}$$
has $r=16$, $s=5$, and $l=2$. The overlined and the underlined
$2$'s should be decreased by $1$ to get the two vertex sequences.
The first triangulation has $p=1$, $q=10$; the second has $p=6$,
$q=15$. In this way, $\alpha_0=5$, $\alpha_1=3$, $\alpha_2=7$,
$\alpha_3=3$, and $\alpha_4=5$, and so $x_0=5$, $x_1=3$, $x_2=3$,
$x_3=5$, and $x_4=1$. The sequence
$$\{3,\overline 2,2,3,2, \underline 2,3,5,3,\overline 2,2,3,2,\underline 2,3,5\}$$
has $r=15$, $s=8$, and $l=4$. We have $\alpha_0=3$, $\alpha_1=3$,
$\alpha_2=3$, $\alpha_3=5$, $\alpha_4=3$, $\alpha_5=3$,
$\alpha_6=3$, and $\alpha_7=5$, and so $x_0=3$, $x_1=3$, $x_2=1$,
$x_3=1$, $x_4=3$, $x_5=3$, $x_6=1$, and $x_7=1$.
\end{example}

{\em We first consider the case of odd $s=2l+1$.} After
interchanging $p$ and $q$ if necessary, we can assume that the
vertex $p$ is located clock-wise from the vertex $p'$. We define
$k \in \Z_s$ through $v_{p'}=\alpha_k$. Notice that $k>0$. Then
the sequence $\alpha_0,\ldots,\alpha_{s-1}$ satisfies another
system of equations:
$$\alpha_i=\begin{cases} x_k& i=k\cr x_{k+l} & i=k+l+1 \cr x_{2k-i}+2 & i\ne k, k+l+1\end{cases}$$
(as a reminder, the indexing set is $\bZ_s$). We claim that $k$ is
uniquely determined by the sequence $\{x_i\}$. This will show that
the second triangulation compatible with $\aaa$ is unique (if
exists). We have the following system of equations when $k\neq l$
or $k\neq l+1$

$$\begin{cases} x_i=x_{i+2k} & i\ne 0,-k,l, -k+l\cr
x_0=x_{2k}+2\cr x_{-k}=x_k-2\cr x_l=x_{l+2k}+2\cr
x_{-k+l}=x_{k+l}-2\cr
\end{cases}$$
Let $$\Gamma=\langle k\rangle=\langle 2k\rangle\subset\bZ_s.$$
Then the system of equations above shows that $x_i$ is constant on
$\Gamma$-cosets in $\bZ_s$ unless a coset contains $0$, $-k$, $l$,
or $-k+l$. Next we compute a Gauss sum
$$G=\sum_{i=0}^{s-1}x_i\mu^i,$$
where $\mu$ is a primitive $s$-th root of unity. Clearly only one
or two $\Gamma$ cosets contribute non-trivially to the Gauss sum.
Namely, if
$$l\not\in\Gamma$$
then there will be two non-trivial cosets, and we have
$$G=-2(\mu^{2k}+\mu^{4k}+\ldots+\mu^{2k\alpha})-2(\mu^{2k+l}+\mu^{4k+l}+\ldots+\mu^{2k\alpha+l}),$$
where $\alpha$ is the minimal positive integer such that
$2k\alpha\equiv -k$. Geometric series and some manipulations give
\begin{equation}\label{asdaagaqg}
G=2{\mu^k\over \mu^{k}+1}(1+\mu^l).
\end{equation}
We can clearly recover $\mu^k$, and hence $k$, by applying a
fractional-linear transformation.

If $l\in\Gamma$ then in fact $\Gamma=\bZ_s$, so there is only one
coset. We have $2kl\equiv -k$, and we choose a minimal $t$ such
that $2kt \equiv l$. There are two further subcases, either $t<l$
and we have
$$G=-2(\mu^{2k}+\ldots+\mu^{2kt})-4(\mu^{2k(t+1)}+\ldots+\mu^{2kl})-2(\mu^{2k(l+1)}+\ldots+\mu^{2k(t+l)}),$$
and so $G = 2{\mu^k\over \mu^{k}+1}(1+\mu^l)$, or $t>l$, in which
case a similar calculation gives $G=2{\mu^k\over
\mu^{k}+1}(1+\mu^l)$. For the cases $k=l$ and $k=l+1$ we obtain
same $G$ as well. Therefore $k$ is uniquely determined by the
$\{x_i\}$.

{\em Now let's consider the case of even $s=2l$.} After
interchanging $p$ and $q$ if necessary, we can assume that $p'$
satisfies $v_{p'}=\alpha_{k}$ with $0\le k<l$. There are two
further subcases: in the {\em clockwise} (resp.~{\em
counter-clockwise}) case the vertex $p$ is located clockwise
(resp.~{\em counter-clockwise}) from the vertex $p'$. Notice that
in the clockwise case $k>0$. The sequence
$\alpha_0,\ldots,\alpha_{s-1}$ satisfies a new system of
equations:
$$\alpha_i=\begin{cases} x_i& i=k, l+k\cr x_{2k-i}+2 & i\ne k, k+l\end{cases}$$
in the clockwise case and
$$\alpha_i=\begin{cases} x_{i-1}& i=k, l+k\cr x_{2k-i-1}+2 & i\ne k, k+l\end{cases}$$
in the counter-clockwise case. This gives a system of equations on
$x_i$:
$$\begin{cases} x_i=x_{i+2k}& i\ne0, -k, l, l-k\cr x_i=x_{i+2k}+2 & i=0, l\cr x_i=x_{i+2k}-2 & i=-k, l-k\end{cases}$$
in the clockwise case and
$$\begin{cases} x_i=x_{i+2k-1}& i\ne0, -k, l, l-k\cr x_i=x_{i+2k-1}+2 & i=0, l\cr x_i=x_{i+2k-1}-2 & i=-k, l-k\end{cases}$$
in the counter-clockwise case. We have to recover $k$ (and whether
we have a clockwise or a counter-clockwise case from the sequence
$x_i$). To simplify this problem, consider a new sequence $y_i$
defined as
$$y_i={x_i+x_{i+l}\over 4}, \ i=0,\ldots,l-1.$$
we think about $y_i$ as indexed by the cyclic group $\bZ_l$. It
satisfies the following system of equations:
$$\begin{cases} y_i=y_{i+2k}& i\ne0, -k\cr y_0=y_{2k}+1 \cr y_{-k}=y_{k}-1\end{cases}$$
in the clockwise case and
$$\begin{cases} y_i=y_{i+2k-1}& i\ne0, -k\cr y_0=y_{2k-1}+1 \cr y_{-k}=y_{k-1}-1\end{cases}$$
in the counter-clockwise case. We define a subgroup
$$\Gamma=\langle 2k\rangle\subset\bZ_l\quad(\hbox{\rm resp.\ }\Gamma=\langle 2k-1\rangle\subset\bZ_l)$$
in the clockwise (resp.~counter-clockwise) case. We also define a
Gauss sum
$$G=\sum_{i=0}^{l-1}y_i\mu^i,$$
where $\mu$ is a primitive $l$-th root of unity.

The system of equations above shows that $y_i$ is constant on
$\Gamma$-cosets in $\bZ_l$ unless a coset contains $0$ or $-k$.
Moreover, in the latter case the coset must contain both $0$ and
$-k$, as otherwise we will have a non-constant periodic monotonous
function. Only this coset will contribute non-trivially to~$G$. In
the clockwise case we have
$$-k\in \langle 2k\rangle\quad\Rightarrow\quad \langle k\rangle=\bZ_{2m+1},$$
$-k\equiv 2mk$, and
$$y_{2k}+1=y_{4k}+1=\ldots=y_{2mk}+1=y_{2(m+1)k}=\ldots=y_0.$$
Therefore
$$G=-\mu^{2k}-\mu^{4k}-\ldots-\mu^{2mk}={\mu^k\over 1+\mu^k}.$$
Obviously $\mu^k={G\over 1-G}$ and $G$ determines $k$. In the
counter-clockwise case we have
\begin{equation}\label{afvbdfbsf}
-k\in \langle 2k-1\rangle \quad\Rightarrow\quad \langle k
\rangle=\langle 2k-1\rangle=\bZ_l.
\end{equation}
Hence
$$y_{2k-1}+1=y_{2(2k-1)}+1=\ldots=y_{a(2k-1)}+1=y_{k-1}=\ldots=y_0,$$
where $-k\equiv a(2k-1)$. So we obtain
\begin{equation}\label{afvbAAGS}
G=-\mu^{2k-1}-\mu^{2(2k-1)}-\ldots-\mu^{a(2k-1)}={\mu^{k-1}-\mu^{2k-1}\over
1-\mu^{2k-1}}.
\end{equation}
It follows that
$${G\over 1-G}=\mu^{k-1}{1-\mu^k\over 1-\mu^{k-1}},$$
and so
\begin{equation}\label{akSDJFHL}
\left|{G\over 1-G}\right|\ne1:
\end{equation}
indeed, as we connect a fixed vertex $1$ of an $l$-gon to other
vertices by diagonals, two consecutive diagonals never have equal
length, unless $l=2k-1$, but this would contradict
\eqref{afvbdfbsf}. In particular, we cannot have both a clockwise
(in which case $\left|{G\over 1-G}\right|=1$) and a
counter-clockwise case. It follows from \eqref{afvbAAGS} that we
can find $k$ by solving the quadratic equation
$$(\mu^k)^2+{1\over G-1}(\mu^k)+{G\over 1-G}\mu=0.$$
By \eqref{akSDJFHL}, this equation cannot have two solutions.

\subsection{Same $\delta$ over $\frac{1}{\Delta}(1,\Omega)$} \label{delta}

We now prove Theorem \ref{samedelta}. As for Theorem \ref{atmost2},
we only need to consider case (A). Let $\aaa=\{a_0,\ldots,a_r\}$
be a sequence of type (A). We adopt the same conventions as
before, and define $a_0=\alpha_0, \alpha_1, \ldots, \alpha_{s-1}$,
and the $\{x_0,\ldots,x_{s-1}\}$, where $p_1=1<q_1$ and $p_2<q_2$
are the indices defining the two zero continued fractions
corresponding to $\aaa$. We define for $j=1,2$
$$\frac{\delta_j}{\varepsilon_j}:=[a_{p_j +1},\ldots,a_{q_j -1}],$$
for some $\varepsilon_j$ (see Proposition \ref{contfracPres}). One can verify that
$$\frac{\delta_j^2}{\delta_j \lambda_j +1}= [a_{q_j + 1},a_{q_j
+2},\ldots,a_r,a_0,\ldots,a_{q_j -1}]$$ for some
$0<\lambda_j<\delta_j$ with gcd$(\delta_j,\lambda_j)=1$; this is a property of the dual of continued fractions from Wahl singularities (see \cite{S89}, 3.5). The following is \cite{S89}, Lemma 3.4.

\begin{lemma} Given $[b_1,\ldots,b_r]=\frac{m^2}{ma-1}$ with $0<a<m$
and gcd$(m,a)=1$, we define as in \S \ref{toric} the sequences $\alpha_{i}/\alpha_{i-1}=[b_{i-1},\ldots,b_1]$ for $2 \le i \le r$, and
$\beta_{i}/\beta_{i+1}=[b_{i+1},\ldots,b_{r}]$ for $1 \le i \le r-1$. Define $\alpha_0=0$ and $\beta_0=m^2$. Let $m \delta_i:= \alpha_i + \beta_i$ for $0 \leq i \leq r+1$. Let $\tilde{\alpha}_i$, $\tilde{\beta}_i$,
$(a+m) \tilde{\delta}_i := \tilde{\alpha}_i + \tilde{\beta}_i$ be the
numbers for $\frac{(a+m)^2}{(a+m)m-1} = [e_1+1,\ldots,e_r,2]$.

Then $\tilde{\delta}_i = \delta_i$ for $1 \leq i \leq r+1$,
and $\tilde{\delta}_0=\tilde{\delta}_{r+2}=m+a= \delta_1 +
\delta_{r+1}$.
\end{lemma}

This lemma gives a way to compute $m,a$ for continued fractions of Wahl
singularities from the inductive procedure (cf. \cite{KSB88}, Proposition 3.11) $[e_1,\ldots,e_r] \to
[e_1+1,\ldots,e_r,2]$ or $[2,e_1,\ldots,e_r+1]$, starting with
$[4]$. For $[4]$ one has $m=2, a=1$; here $\delta_1=1$,
$\delta_0=\delta_2=2$. Thus for $[5,2]$ we have $m=1+2=3,
a=1+2-2=1$, and so on. Therefore, the $\delta_i$'s can be used to
compute $m, a$ from the continued fraction of $\frac{m^2}{ma+1}$.
Start with $[2,2,2]$ and assign $1$ for the first $2$
and $1$ for the last $2$. For the new continued fraction we assign
to the new $2$ the addition of the extremal previous assignations.
To obtain $m$ from $[e_1,\ldots,e_r]$, we just add the numbers
assigned to $e_1$ and to $e_r$. For $[2,2,2]$ we have
$m=1+1=2$. For $[3,2,2,2]$ we have $m=1+2=3$, for
$[2,3,2,2,3]$ $m=3+2=5$, for $[3,3,2,2,3,2]$
$m=3+5=8$, etc.

The continued fraction $$\frac{x}{y}= b_1 + \frac{1}{b_2 +
\frac{1}{ \ddots + \frac{1}{b_s}}},$$ where $b_i>0$ for all $i$,
will be denoted by $\langle b_1,\ldots,b_s \rangle$. We have

$$\begin{pmatrix} x& *\cr y& * \cr\end{pmatrix}=\begin{pmatrix} b_1& 1\cr 1& 0\cr\end{pmatrix}
\begin{pmatrix} b_2& 1\cr 1& 0\cr\end{pmatrix} \ldots
\begin{pmatrix} b_s& 1\cr 1& 0\cr\end{pmatrix}.$$

\begin{lemma}
Consider a sequence $\aaa= \{a_0,\ldots, a_r \}$ of type (A) as above. We recall that
for $p_2, q_2$ we have indices $p'_2, q'_2$ and an integer $0 \leq
k < s-1$. Then $$\frac{\delta_1}{*}=\langle x_l -1, \ldots, x_1,
x_{-1}, x_0-1 \rangle \ \ \ \  \text{and}
$$ $$\frac{\delta_2}{*}=\begin{cases} \langle x_{k+l} -1, \ldots,
x_{k+1}, x_{k-1}, x_k-1 \rangle & \text{if} \ p'_2=p_2-1
 \cr \langle x_{k-1+l} -1, \ldots, x_{k-2},
x_{k}, x_{k-1}-1 \rangle & \text{if} \ p'_2=p_2+1 \end{cases}.$$
\label{deltalemma}
\end{lemma}

\begin{proof}
The proof uses the inductive way to compute $\delta_1$ for the
continued fraction of $\frac{\delta_i^2}{\delta_i \lambda_i +1}$,
and the numbers $\{x_0,\ldots,x_{s-1} \}$ defined before. The role
of $x_0$ is taken by $x_k$ (if $p'_2=p_2-1$) or $x_{k-1}$ (if
$p'_2=p_2+1$) to compute $\delta_2$. We start with $[2,\overline
2,2]$ and use $x_0$ to obtain $[\alpha_0,\overline 2,2,\ldots,2]$.
Then use $x_{-1}$ to obtain $[2,\ldots,2,\alpha_0,\overline
2,2,\ldots,2,\alpha_1]$, and continue. During this process we keep
track of the assigned multiplicities explained above. This defines
the recursion $n_{-1}=1$, $n_0=x_0-1$,
$$n_i=\begin{cases} x_{\frac{-i-1}{2}}n_{i-1} + n_{i-2} & i= \text{odd} \cr x_{\frac{i}{2}} n_{i-1} + n_{i-2} & i= \text{even}\end{cases}$$
for $1\leq i \leq s-2$, and $\delta_1=n_{s-1}= (x_l -1)n_{s-2} +
n_{s-3}$. In this way, $$\frac{n_{s-1}}{n_{s-2}} = \langle x_l -1,
\ldots, x_1, x_{-1}, x_0-1 \rangle.$$
\end{proof}

We are going to compare the $\delta_i$'s, but before another
lemma.

\begin{lemma}\label{magic}
Let $X=\begin{pmatrix} a& b\cr c& d\cr\end{pmatrix}$ be
a matrix such that $\det X=1$. Define matrices
$$R_1=\begin{pmatrix} 1& -1\cr 1& 0\cr\end{pmatrix}\quad\hbox{\rm and}\quad
R_{-1}=\begin{pmatrix} 1& 1\cr -1& 0\cr\end{pmatrix}.$$
Consider an arbitrary sequence $\underline
i=\{i_1,\ldots,i_{q}\}$, where $i_\alpha=\pm1$, and the matrix
$$R_{\underline i}:=XR_{i_1}X^tR_{i_2}XR_{i_3}X^t\ldots $$
(which ends with $X$ or $X^t$ depending on parity of $q$). Then
\begin{enumerate}
\item $R_{\underline i}=\begin{pmatrix} p_{\underline i}(a)&
*\cr *& *\cr\end{pmatrix}$, where $p_{\underline i}(a)$ is a
polynomial which depends only on~$a$. \item If $q$ is even, the
left-upper corner of $R_{\underline i}$ is equal to the left-upper
corner of $R_{\underline j}$, where $\underline
j=\{j_1,\ldots,j_k\}$ is a sequence given by $j_\alpha=-i_\alpha$
for any~$\alpha$.
\end{enumerate}
\end{lemma}

\begin{proof}
The first statement follows by induction on $q$, where we can
show, more precisely, that
$$R_{\underline i}=\begin{pmatrix} af_{\underline i}(a)& cf_{\underline i}(a)+g_{\underline i}(a)\cr *& *\cr\end{pmatrix}$$
if $q$ is odd and
$$R_{\underline i}=\begin{pmatrix} af_{\underline i}(a)& bf_{\underline i}(a)+g_{\underline i}(a)\cr *& *\cr\end{pmatrix}$$
if $q$ is even, where $f_{\underline i}(a)$ and $g_{\underline
i}(a)$ are polynomials in $a$ only.

The second statement of the lemma follows from the first. Indeed,
in order to compute $p_{\underline i}(a)$ we can assume that
$X=X^t$ is a symmetric matrix of determinant $1$. Then
$R_{\underline i}=R_{\underline k}^t$ (and in particular their
upper-left entries are the same), where $\underline k$ is a
sequence obtained from $\underline j$ by reversing the order. On
the other hand, we can also assume that $X=\begin{pmatrix} a&
-1\cr 1& 0\cr\end{pmatrix}$. Then $p_{\underline i}(a)$ is
equal (up to sign) to the numerator of the Hirzebruch--Jung
continued fraction with entries $(a,i_1,-a,i_2,\ldots,a)$. This
numerator does not change if we reverse the order.
\end{proof}

We will show $\delta_1=\delta_2$ case by case. Some notation: a
sequence $\{i_0,\ldots\}$ is a sequence of $\pm1$,
$\underline{z}=\{z_1,\ldots,z_w\}$ is a sequence of integers, and
$\underline{\bar z}$ is the same sequence of integers but in
reversed order.

\textbf{Case I.} Assume $s=2l$ and $p'_2=p_2+1$, i.e. the
counter-clockwise case. Here we have $\Gamma=\langle 2k-1
\rangle=\Z_{s}$. One can check that for all $d$, $x_{-d} =
x_{k-1+d}$. By the formulas of Lemma \ref{deltalemma}, we have
$\delta_1=\delta_2$.

\textbf{Case II.} Assume $s=2l$ and $p'_2=p_2-1$, i.e. the
clockwise case. We know that either $-k$ or $l-k$ belongs to
$\Gamma$. Assume that $-k \in \Gamma$. We also redefine
$x_0:=x_0-1$ and $x_l=x_l-1$. Then, the transposed formulas of
Lemma \ref{deltalemma} are $$\frac{\delta_1}{*}=\langle x_0,
\underline{z}, x_l+i_0,x_l-i_0,\underline{\bar
z},x_0+i_1,x_0-i_1,\underline{z}, x_l+i_2,x_l-i_2,\underline{\bar
z},\ldots,x_l \rangle
$$ and $$\frac{\delta_2}{*}= \langle x_0, \underline{z}, x_l-i_0,x_l+i_0,
\underline{\bar z},x_0-i_1,x_0+i_1,\underline{z},
x_l-i_2,x_l+i_2,\underline{\bar z},\ldots,x_l \rangle $$ for some
$\{i_0,\ldots \}$ and $\underline z$. But we know that
$\delta_i$'s are computed via the upper-left corner of the
multiplication of the corresponding $2 \times 2$ matrices, and if
$$X=\begin{pmatrix} x_0& 1\cr 1& 0\cr\end{pmatrix} \begin{pmatrix} z_1& 1\cr 1& 0\cr\end{pmatrix} \ldots
\begin{pmatrix} z_w& 1\cr 1&
0\cr\end{pmatrix} \begin{pmatrix} x_l& 1\cr 1&
0\cr\end{pmatrix}$$ we can use Lemma \ref{magic} to prove
$\delta_1=\delta_2$. We notice that
$$\begin{pmatrix} x_l+i_0& 1\cr 1& 0\cr\end{pmatrix} \begin{pmatrix} x_l-i_0& 1\cr 1& 0\cr\end{pmatrix}=
\begin{pmatrix} x_l& 1\cr 1& 0\cr\end{pmatrix}
\begin{pmatrix} 1& -i_0\cr i_0& 0\cr\end{pmatrix}
\begin{pmatrix} x_l& 1\cr 1& 0\cr\end{pmatrix}.$$

\textbf{Case III.} Assume $s=2l$, $p'_2=p_2-1$, and that $l-k \in
\Gamma$. We also redefine $x_0:=x_0-1$ and $x_l=x_l-1$. Then, the
transposed formulas of Lemma \ref{deltalemma} are
$$\frac{\delta_1}{*}=\langle x_0, \underline{z},
x_l+i_0,x_l-i_0,\underline{\bar z},x_0+i_1,x_0-i_1,\underline{z},
x_l+i_2,x_l-i_2,\underline{\bar z},\ldots,x_l \rangle
$$ and $$\frac{\delta_2}{*}= \langle x_l, \underline{\bar z}, x_0-i_0,x_0+i_0,
\underline{z},x_l-i_1,x_l+i_1,\underline{\bar z},
x_0-i_2,x_0+i_2,\underline{z},\ldots,x_0 \rangle $$ for some
$\{i_0,\ldots \}$ and $\underline z$. So we do the same as in Case
II, noticing that by Lemma \ref{magic} the use of $X$ or $X^t$
does not change the $\delta_i$.

\textbf{Case IV.} Assume $s=2l+1$, $k \neq l$ or $l+1$, and that
$l \in \Gamma$. Then $\Gamma=\Z_s$. In this case $p'_2=p_2-1$. One
can verify that $x_d=x_{k+l-d}$ for all $d$, and so by Lemma
\ref{deltalemma} if $\frac{\delta_1}{*}=\langle \underline z
\rangle$ then $\frac{\delta_2}{*}= \langle \underline{ \bar z}
\rangle$. This gives $\delta_1=\delta_2$.

\textbf{Case V.} Assume $s=2l+1$, $k \neq l$ or $l+1$, and that
$l$ is not in $\Gamma$. Here $p'_2=p_2-1$. This case can be
treated as case II.

\textbf{Case VI.} Assume $s=2l+1$, and that $k=l$. Then via Lemma
\ref{deltalemma} we have $\frac{\delta_1}{*}=\langle
x_l-1,x_0,\ldots,x_0,x_0-1 \rangle$ and
$\frac{\delta_2}{*}=\langle x_0-1,x_0,\ldots,x_0,x_l-1 \rangle$.
In the case $k=l+1$, we would have $\frac{\delta_1}{*}=\langle
x_0-1,x_0,x_0-2,\ldots,x_0-2,x_0-1 \rangle$ and
$\frac{\delta_2}{*}=\langle x_0-1,x_0-2,\ldots,x_0-2,x_0,x_0-1
\rangle$.

Therefore, in any case, $\delta_1=\delta_2$.

\section{Applications to moduli spaces of stable surfaces} \label{s4}

Let $\cM_{K^2,\chi}$ be the Gieseker's moduli space of canonical
surfaces of general type with given numerical invariants $K^2$ and
$\chi$. Let $\overline{\cM}_{K^2,\chi}$ be its compactification,
the moduli space of stable surfaces of Koll\'ar, Shepherd-Barron,
and Alexeev. Given a map $\D^{\times}\to\cM_{K^2,\chi}$ from a punctured
smooth curve germ, i.e.~ a family $\X^{\times}\to\D^{\times}$ of canonical surfaces,
one is interested in computing its ``stable limit''
$\D\to\overline \cM_{K^2,\chi}$, perhaps after a finite base
change. We would like to make more explicit a well-known algorithm
for computing a stable limit under an extra assumption that the
family $\X^{\times}\to\D^{\times}$ can be compactified by a flat family
$\X\to\D$, where the special fiber $X$ is irreducible, normal, and
has quotient (=log terminal) singularities. We will show that the stable limit can be found using only flips and divisorial contractions from the extremal neighborhoods of type $k1A$ and $k2A$ of the previous sections.

\begin{remark}
We are going to perform various operations on families
(e.g.~simultaneous partial resolutions) which make it necessary to
assume that the total space of the family is an algebraic space
(or an analytic space) rather than a scheme, even if the total
space of the original family is a scheme. However, our main
applications are to moduli of surfaces of geometric genus~$0$, in
which case any flat family is projective, and therefore its total
space is a scheme. So the reader can ignore this technicality.
\end{remark}

The next lemma shows that we can reduce our analysis to the case when $\X\to\D$ is a
$\Q$-Gorenstein smoothing and the special fiber has Wahl singularities only.
This is a global version of the M-resolution \cite{BC94}
of Behnke and Christophersen.

\begin{lemma}
Let $\cX\to\D$ be a flat family of irreducible reduced projective
surfaces such that the special fiber has quotient (=log-terminal)
singularities. There exists a finite morphism $\D' \to \D$, an
analytic space~$\X'$, and a proper birational map
$F:\,\X'\to\X\times_\D \D'$ such that
\begin{enumerate}
\item The special fiber $X'$ is irreducible and has Wahl singularities.
\item $\X'$ has terminal singularities.
\item $F$ induces a minimal resolution of general fibers $X'_s\to X_s$.
\item $K_{\cX'}$ is $F$-nef.
\end{enumerate}
\end{lemma}

\begin{proof}
By \cite{KSB88}, Theorem 3.5, we can choose $\cX'$ such that all properties above are satisfied except that $X'$ is allowed to have arbitrary T-singularities (i.e.~$\Q$-Gorenstein smoothable quotient singularities). Notice that locally on $\cX$ this is nothing but a P-resolution.
By \cite{BC94}, any singular point of $\cX'$ has an analytic neighborhood $U$
and a proper crepant map $V\to U$ (a crepant M-resolution) which has the properties required in the Lemma, in particular the special fiber has Wahl singularities only. Gluing these resolutions gives the required analytic space.
\end{proof}

\begin{theorem}\label{t3}
Assume we have a $\Q$-Gorenstein smoothing $(X \subset \X) \to (0 \in
\D)$ of a projective surface $X$ with only Wahl singularities.
Suppose some positive multiple of $K_X$ is effective. Then, after a finite
sequence of flips and divisorial contractions from extremal neighborhoods of type $k1A$ and $k2A$ (with $b_2(X_s)=1$, as in Proposition \ref{b2=1=>normal}), we obtain a $\Q$-Gorenstein smoothing $F \colon \X'\to \D$ with $K_{\X'}$ $F$-nef.
\end{theorem}

\begin{proof}
Notice that $K_{X}$ is nef if and only if $K_\X$ is $F$-nef.
Indeed, suppose $K_\X \cdot C_s<0$ for some curve in a general
fiber. Without loss of generality we can assume that $C_s$ is a
$(-1)$-curve, and therefore deforms over $\D^{\times}$. Since $\X$ is
$\Q$-Gorenstein, $K_X$ will intersect negatively one of the
irreducible components of the central fiber of the closure of this
family of $(-1)$-curves.

To avoid issues with the fact that $\X$ in general is not a
variety, we will run a MMP on $\X$ guided by a MMP on the special
fiber. Suppose $K_X$ is not nef. Then we have an extremal
contraction $X \to Y$ of log terminal surfaces, where the
exceptional locus $C$ is an irreducible rational curve.
Singularities of $Y$ are log terminal, and therefore rational. By
blowing down deformations \cite[11.4]{KM92}, we get an extremal neighborhood
$\X\to\Y$ of either divisorial or flipping type, and in the latter
case the flip exists in the category of analytic spaces by
\cite{M88}. It remains to show that we only encounter flips (or
divisorial contractions) of types $k1A$ and $k2A$ (with $b_2(X_s)=1$). Since $Y$ has quotient
singularities, we just have to rule out the possibility of having
non cyclic quotient singularities. However, this is known by
\cite{K88}, pp.~157--159, the proof that subcase (3c) of page 154 does not occur.
\end{proof}

\subsection{Example}

Below we construct a stable surface $X$ which belongs to $\overline{\cM}_{K^2,\chi}$ with $K^2=4$ and $p_g=0$ (so $q=0$ and $\chi=1$). The surface $X$ has two Wahl singularities $P_1=\frac{1}{252^2}(1,252 \cdot 145 -1)$ and $P_2=\frac{1}{7^2}(1,7 \cdot 5 -1)$, and no local-to-global obstructions to deform (i.e. deformations of its singularities globalize to deformations of $X$). We will show that a $\Q$-Gorenstein smoothing of $X$ is a simply connected smooth projective surface of general type with the above invariants. This is the first example in the literature with those invariants and two Wahl singularities, which is the maximum possible under no local-to-global obstructions. Examples with one Wahl singularity appeared in \cite{PPS09-2}.

We consider a  $\Q$-Gorenstein deformation of $X$ which preserves the singularity $P_i$ and smooths $P_j$. Since this deformation is trivial around $P_i$, we resolve simultaneously $P_i$ to obtain a $\Q$-Gorenstein smoothing of a surface with a single Wahl singularity at $P_j$. On this deformation we explicitly run the MMP of Theorem \ref{t3}. We do it for each $i=1,2$.

Let us consider the cubic pencil $$t_0 \, (x_0^2+x_1x_2)(2x_0-x_1+x_2) + t_1 \, x_0 x_1 x_2=0$$ in $\P_{x_0,x_1,x_2}^2$ with $[t_0:t_1] \in \P_{t_1,t_2}^1$. Let $L_1=(x_0=0)$,
$L_2=(x_1=0)$, $L_3=(x_2=0)$, $B=(x_0^2+x_1 x_2=0)$, and $A=(2x_0-x_1+x_2=0)$. The curves $B$ and $A$ are tangent at $[1:1:-1]$. This pencil defines an elliptic fibration $g\colon Z \to \P_{t_0,t_1}^1$ with singular fibers $I_7, 2I_1, III$ (Kodaira's notation; Cf. \cite{BHPV04}, p. 201), after we blow-up the nine base points. Choose a nodal $I_1$ fiber and denote it by $F$, and consider the line $M$ passing through the point $r:= A \cap L_1$ and the node of $F$. The line $M$ becomes a double section of $g$.


\begin{figure}[htbp]
\includegraphics[width=10cm]{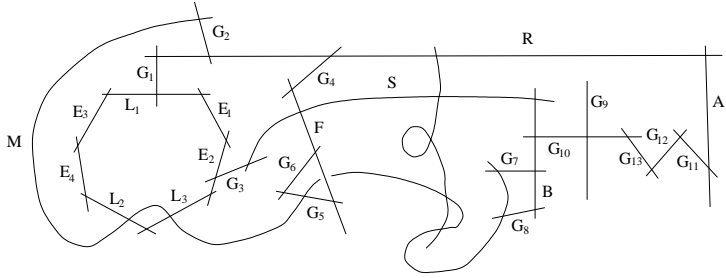}
\caption{Relevant curves in $\widetilde{Z}$} \label{f3}
\end{figure}

We blow up $13$ times $Z$ to obtain $\widetilde{Z}$. Write $\tau \colon \widetilde{Z} \to Z$ for the composition of these blow-ups. The curves $G_i$ in Figure \ref{f3} are the exceptional curves of $\tau$. The order of the blow-ups is indicated by the
subindices of the $G_i$. In Figure \ref{f3} we label all relevant
curves for the construction. The self-intersections of all of them
are in Figure \ref{f4}. Let $\sigma \colon \widetilde{Z} \to X$ be the contraction of the configurations
$$W_1 =  G_9 + G_{10} + B + S + F + G_5 + M + L_2 + E_4 +
E_3 + L_1 + E_1 + E_2$$ and $W_2 = G_{12} + G_{11} + A + R$. Let $P_i$ be the image of $W_i$, so $P_1=\frac{1}{252^2}(1,252 \cdot 145-1)$ and $P_2=\frac{1}{7^2}(1,7 \cdot 5-1)$.

\begin{figure}[htbp]
\includegraphics[width=10cm]{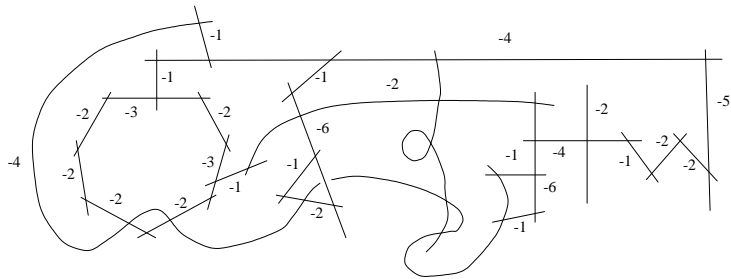}
\caption{Self-intersections of relevant curves in $\widetilde{Z}$}
\label{f4}
\end{figure}

\begin{proposition}
The surface $X$ has no local-to-global obstructions to deform. A $\Q$-Gorenstein smoothing of $X$ is a simply connected projective surface of general type with $K^2=4$ and $p_g=0$. The surface $X$ is a stable surface. \label{p4}
\end{proposition}

\begin{proof}
According to \cite{LP07}, Section 2, if $H^2(Z_{\widetilde{Z}}(- \log
(W_1 + W_2)))=0$, then $X$ has no local-to-global obstructions to deform. Let $\tau_0
\colon Z_0 \rightarrow Z$ be the composition of blow-ups
corresponding to $G_5$, $G_9$, and $G_{10}$. Let $C$ be the
general fiber of $g \colon Z \rightarrow \P_{t_0,t_1}^1$, and let $D=L_2+E_4+E_3+L_1+E_1+E_2$. Then
$$K_{Z_0} \sim - \tau_0^* C + G_5 + G_9 + 2 G_{10}$$ and so
$\text{dim}_{\C} H^2(T_{Z_0}(-F-A-B-G_9-D))=\text{dim}_{\C}H^0(\Omega_{Z_0}^1(K_{Z_0}+F+A+B+G_9+D))
\leq \text{dim}_{\C} H^0(\Omega_{Z_0}^1 (\tau_0^* C + D))=0$. We can follow the strategy in
\cite{LP07} to show $H^2(T_{\widetilde{Z}}(- \log
(W_1 + W_2)))=0$.

We now compute $\sigma^* K_{X}$. We have $K_{\widetilde{Z}} \sim -
\tau^* C + \sum_{i=1}^{13} G_i + G_6 + G_{10} + 2 G_{11} + 5
G_{12} + 8 G_{13}$ and $$\tau^* C \sim F + 2 G_5 + 3 G_6 + G_4
\sim B + A + 2G_9 + 4 G_{10} + 5 G_{11} + 9 G_{12} + 13 G_{13}$$
and so
$$ \sigma^* K_{X} \equiv \frac{125}{252} F + \frac{17}{36} B + \frac{5}{14} A + \frac{5}{7} R +
\frac{248}{252} S + G_1 + G_2 + G_3 + \frac{1}{2} G_4 +
\frac{250}{252} G_5 + $$ $$ \ \ \ \ \ \ \ \ \ \ \frac{1}{2} G_6 +
G_7 + G_8 + \frac{107}{252} G_9 + \frac{214}{252} G_{10} +
\frac{15}{14} G_{11} + \frac{25}{14} G_{12} + \frac{5}{2} G_{13}.
$$

Relevant intersections are $\sigma^* K_{X}.G_8=\sigma^* K_{X}.G_7=
\frac{29}{63}$, $\sigma^* K_{X}.G_{13}=\frac{17}{126}$, $\sigma^*
K_{X}.G_6=\frac{83}{84}$, $\sigma^* K_{X}.G_4=\frac{179}{252}$,
$\sigma^* K_{X}.G_3=\frac{47}{84}$, $\sigma^*
K_{X}.G_2=\frac{59}{84}$, and $\sigma^* K_{X}.G_1=\frac{85}{126}$. This says that $K_X$ is nef.
In addition, note that the support of $\sigma^* K_{X}$ above contains
$F$, $G_4$, $G_5$, and $G_6$. This is the support of a fiber, and so $\sigma^* K_{X}$ zero curves should be contained in fibers. There are no such curves. Then by Nakai-Moishezon criterion we have $K_{X}$ is ample, hence $X$ is stable. We clearly have $p_g(X)=0$ and $K_X^2=4=-13+13+4$, and so for the general fiber of the smoothing.

To compute the fundamental group of the general fiber of a $\Q$-Gorenstein smoothing of $X$, we
follow the same strategy as in \cite{LP07}, Theorem 3. Let us write $l \sim l'$ for two equivalent loops $l, l'$ in $\pi_1 (\widetilde{Z} \setminus (W_1 \cup W_2))$. We refer to \cite{Mum61} (see pages 12 and 20) for the relations we will use among loops. In particular notice that for any loop $l$ around $W_1$ we have $l^{252^2}=1$, and for any loop $l'$ around $W_2$ we have ${l'}^{7^2}=1$. The relevant loops for us are:  $l_1$ around $G_9$, $l_2$ around $G_{10}$, $l_3$ around $G_{12}$, $l_4$ around $B$, $l_5$ around $M$, and $l_6$ around $R$. Because of $G_{13}$, we have $l_3 \sim l_2^{\pm}$. In this way, $l_2^{49}=l_3^{49}=1$ in $\pi_1 (\widetilde{Z} \setminus (W_1 \cup W_2))$. By \cite{Mum61}, we have $l_2 \sim l_1^2$, and so $1=l_2^{49}=l_1^{98}$. By \cite{Mum61}, we also have $l_4 \sim l_1^7$, and so $l_4^{14}=l_1^{98}=1$. Because of $G_7$, we have the relation $l_4 \sim l_5^{\pm}$, and so $l_5^{14}=1$. But, again by \cite{Mum61}, we have $l_5 \sim l_1^{723}$. Hence, since $723=3 \cdot 241$, we obtain $l_1^{14}=1$. Therefore, the relation $l_4 \sim l_1^7$ implies $l_4^2=1$. But now we use $G_2$ to have $l_4 \sim l_6^{\pm}$, and so $l_6^2=1$. This together with $l_6^{49}=1$ gives $l_6=1$. Then $l_4=1$ and $l_5=1$, and so $l_1=1$. The fact that $l_1=l_6=1$ in $\pi_1 (\widetilde{Z} \setminus (W_1 \cup W_2))$ is enough to use directly the method in \cite{LP07}, to conclude that a $\Q$-Gorenstein smoothing of $X$ is simply connected.
\end{proof}

Consider a $\Q$-Gorenstein deformation of $X$ which preserves the singularity $P_1$ and smooths $P_2$. Let $X_1$ be the minimal resolution of $P_1 \in X$. Since this deformation is trivial around $P_1$, we resolve minimally simultaneously $P_1$ to obtain a $\Q$-Gorenstein smoothing $X_1 \subset \X_1 \to 0 \in \D$ of $X_1$. We are going to run Theorem \ref{t3} on $X_1 \subset \X_1 \to 0 \in \D$. In general, we use the following notation to save writing when running the MMP of Theorem \ref{t3}.

\begin{notation} \label{dualgraphs}
Locally, let $F \colon (C \subset X \subset
\X) \to (Q \in Y \subset \Y)$ be an extremal neighborhood of type $k1A$ or $k2A$ (with $b_2(X_s)=1$). Let
$\widetilde{X}$ be the minimal resolution of $X$. We remark that the proper transform $\widetilde{C} \subset \widetilde{X}$ of $C$ is a $(-1)$-curve. We have the following dual graph for the exceptional divisors and $\widetilde{C}$. We draw $\bullet$ for the exceptional curves and a $\ominus$ for $\widetilde{C}$, edges are the transversal intersections between curves. If $F$ is divisorial, then we draw $\bullet$ for exceptional curves in the
dual graph of the minimal resolution of $Y$. If $F$ is flipping,
we draw the dual graph of the exceptional curves in the minimal
resolution of $X^+$ (the surface corresponding to the extremal P-resolution) and the proper transform $\widetilde{C}^+$ of
$C^+$, $\bullet$ for the exceptional and a $\oplus$ for
$\widetilde{C}^+$. Numbers are minus the self-intersections of the
curves, no number means $(-1)$-curve.

\begin{figure}[htbp]
\includegraphics[width=7cm]{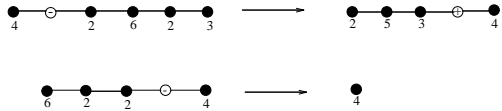}
\caption{Examples of dotted graphs for a flipping and a divisorial birational transformation of type \textit{k}2A}
\label{flipdivex}
\end{figure}

Globally, let $X$ be a projective surface with Wahl singularities, which has a $\Q$-Gorenstein smoothing $(X \subset \X) \to (0 \in \D)$. Let $\widetilde X$ be its minimal resolution. We consider the dual
graph of all exceptional divisors and some curves which may be
used when we perform flips and divisorial contractions of $\X$. At the
beginning we draw $\bullet$ for the exceptional curves and $\circ$
for the rest. One of them is a $\ominus$ which defines the first
birational operation (locally as above). The operation is
indicated with an arrow. The resulting dual graph is the minimal
resolution of the new central fiber (either $Y$ or $X^+$). When
flipping a $\oplus$ will appear, indicating the proper transform
of the flipping curve. Also, a new $\ominus$ will indicate the next
birational operation. Of course the new dotted graph depends on the computations of the previous sections, and they will be omitted.
\end{notation}

We start with the curve $C$ as in Figure \ref{f5}. This curve
defines a flipping extremal neighborhood of type $k1A$ $(C \subset X_1 \subset \X_1) \to (0 \in \D)$. (We use same letter for a curve and its proper transform.) We perform the flip and obtain $C^+ \subset X^+$. This is our new central fiber. In Figure \ref{f5} we show the action of the flip on the minimal resolutions of $X_1$ and $X^+$.

\begin{figure}[htbp]
\includegraphics[width=9cm]{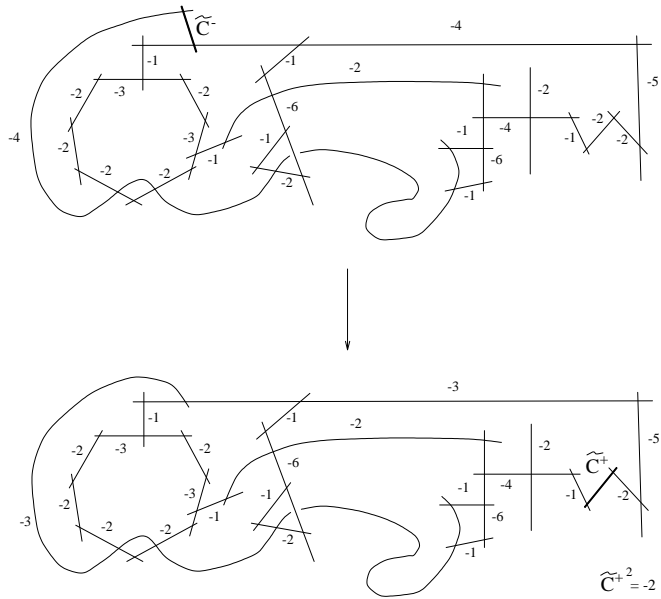}
\caption{The flip between surfaces: from $X_1$ to $X^+$} \label{f5}
\end{figure}

Figure \ref{f6} shows the dual graphs (as explained above) for a
sequence of flips and divisorial contractions starting with $X_1$. The first arrow represents the situation in Figure \ref{f5}. In order we have: two flips, two divisorial contractions and two flips. The three concurrent edges of the last graph represent a (simple) triple point.

\begin{proposition}
The general fiber of  $(X_1 \subset \X_1) \to (0 \in \D)$ is a rational surface. In this way
the singularity $P_1 \in X$ labels a boundary curve in $\overline{\cM}_{K^2,\chi}$ which generically consists of rational stable surfaces with only $\frac{1}{252^2}(1,252 \cdot 145 -1)$ as singularity. \label{p5}
\end{proposition}

\begin{proof} The last central fiber in the sequence of birational
transformations shown in Figure \ref{f6} is nonsingular.
\end{proof}

\begin{figure}[htbp]
\includegraphics[width=9cm]{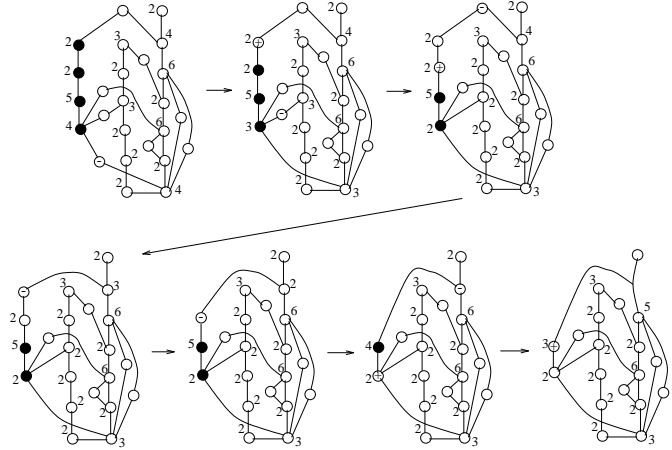}
\caption{Several flips and divisorial contractions} \label{f6}
\end{figure}

We now do it for $P_2 \in X$. As before, consider a $\Q$-Gorenstein deformation of $X$ which preserves the singularity $P_2$ and smooths $P_1$. Let $X_2$ be the minimal resolution of $P_2 \in X$. Since this deformation is trivial around $P_2$, we resolve minimally simultaneously $P_2$ to obtain a $\Q$-Gorenstein smoothing $(X_2 \subset \X_2) \to (0 \in \D)$ of $X_2$.

We now perform $20$ flips, the first $19$ are of type $k1A$ and the
last of type $k2A$. The corresponding sequence of dual graphs is shown in
Figures \ref{f7} and \ref{f8}. (We clarify that in Figures \ref{f7} and \ref{f8} there is one darker edge connecting a $(-2)$-curve and a $(-5)$-curve which are tangent. This is after the 9th flip.)

\begin{figure}[htbp]
\includegraphics[width=9.5cm]{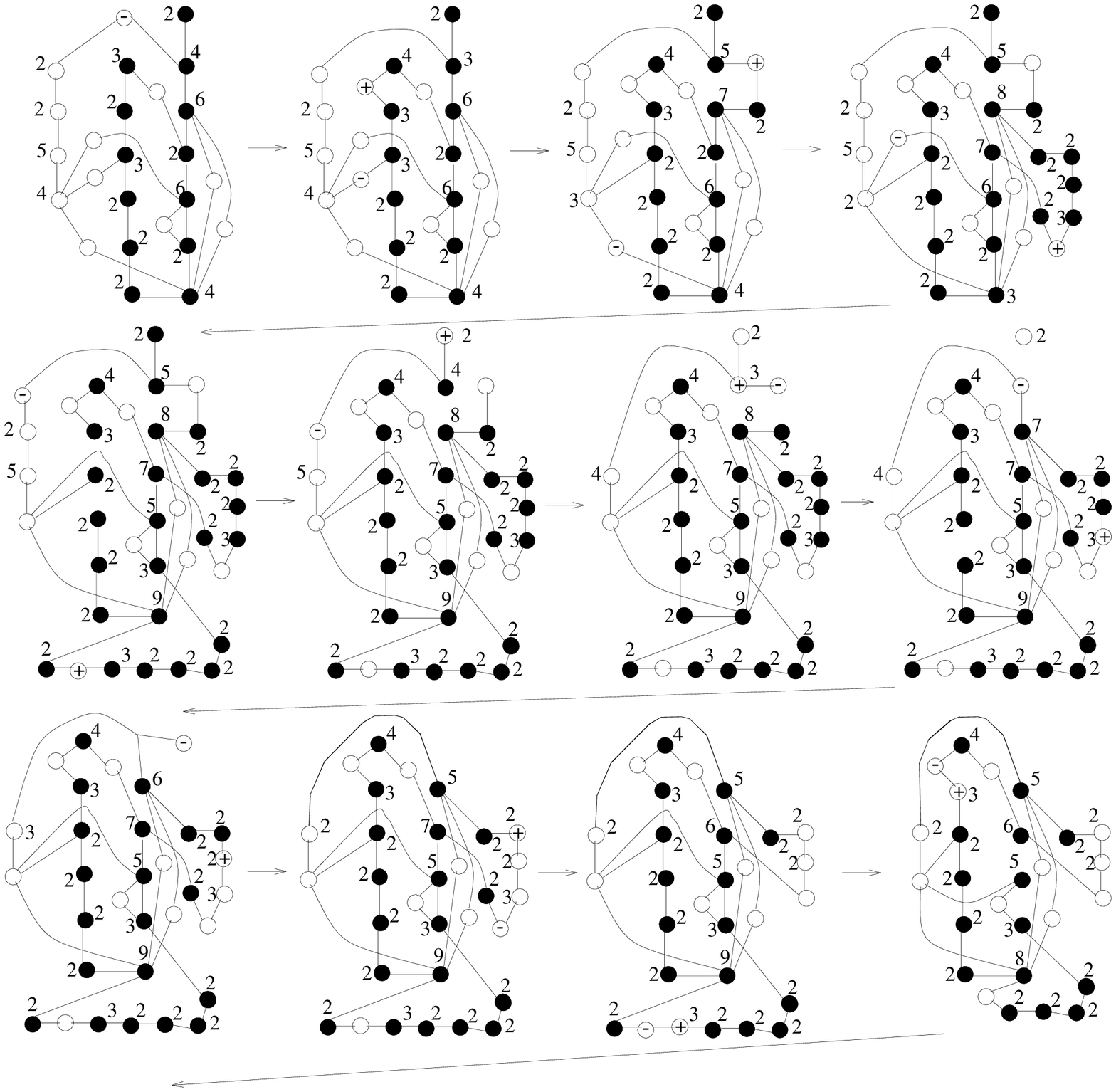}
\caption{}
\label{f7}
\end{figure}

\begin{figure}[htbp]
\includegraphics[width=9.5cm]{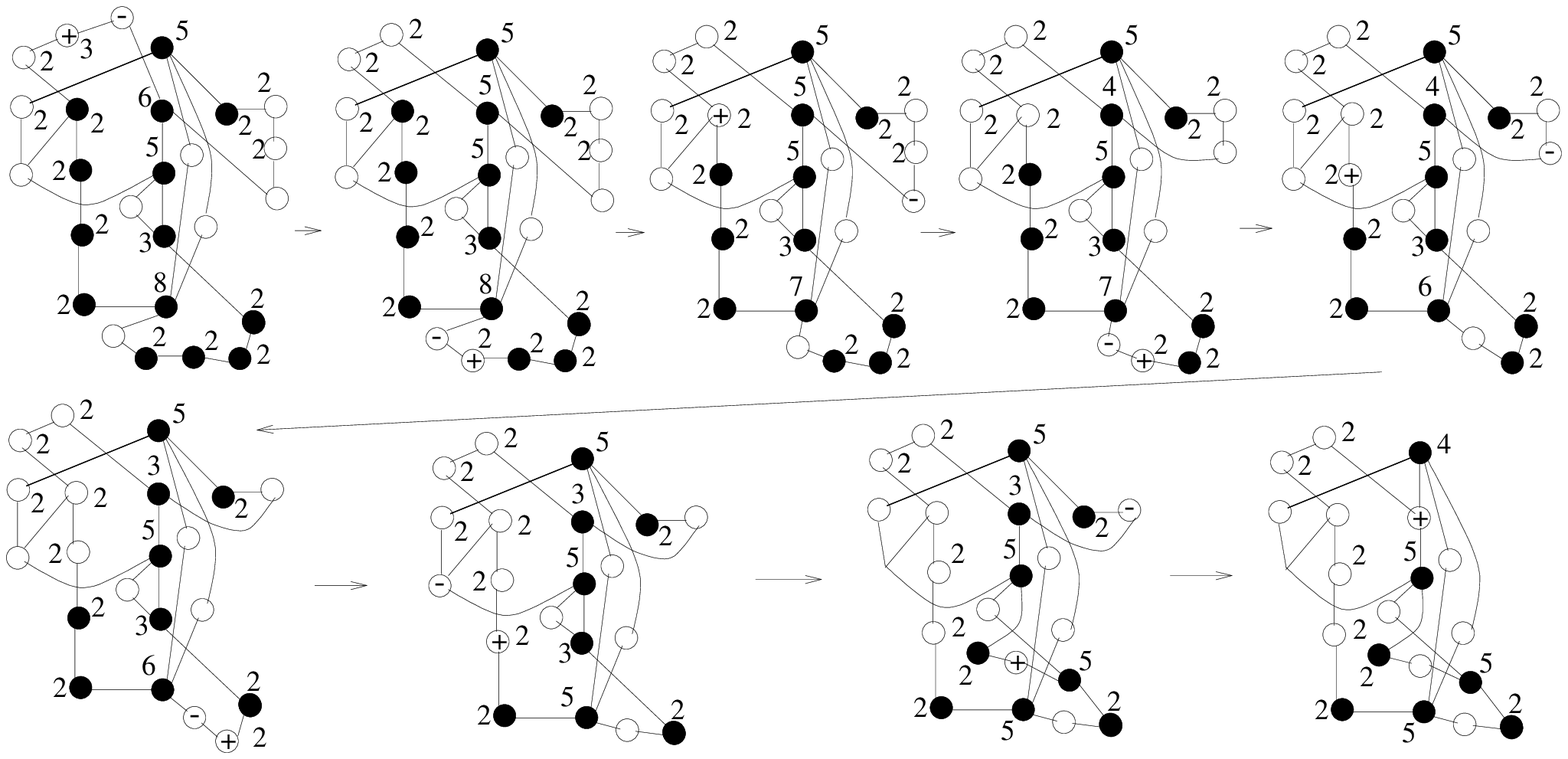}
\caption{}
\label{f8}
\end{figure}

We will prove that the canonical class of the last special fiber (see lower right corner of Figure \ref{f8}) is nef. To save notation, we first contract the configuration $(-5)-(-2)-(-1)-(-5)-(-2)-(-1)-(-5)-(-2)$ into
$(-4)-(-2)-(-3)-(-2)$. Let $X_{20}$ be the resulting surface (last special fiber), and
let $\widetilde{X}_{20}$ be its minimal resolution. In Figure
\ref{f9} we show relevant curves in the blow-up of $\widetilde{X}_{20}$ at one point (the corresponding $(-1)$-curve is $F_9$). Notice that the singular $X_{20}$ has two T-singularities given by the configurations $F+E_1+M+L_2$ and $B$.

The $F_i$'s in Figure \ref{f9} are the exceptional curves to
obtain the minimal initial elliptic fibration (with singular fibers $I_7, 2I_1, III$) from
$\P^2$. The $E_i$'s are the exceptional curves for the four
further blow-ups needed to get $\widetilde{X}_{20}$. Notice we do not
have the $9$th blow-up associated to $F_9$.

Let $\sigma \colon
\widetilde{X}_{20} \rightarrow X_{20}$ be the minimal resolution
of both T-singularities, and let $\pi \colon \widetilde{X}_{20}
\rightarrow \P^2$ be the composition of the $8+4$ blow-ups. In
that way, if $H$ is the class of a line in $\P^2$, then the
canonical class of $\widetilde{X}_{20}$ is
$$K_{\widetilde{X}_{20}} \sim -3H + \sum_{i=1}^8 F_i +F_3 + 2F_4 + F_6 + 2F_7 + E_1+2E_2 + E_3 + E_4.$$

Now we write down $\sigma^* K_{X_{20}}$ $\Q$-effectively. We have
$$ \frac{1}{2} (F+ F_1 + F_2 + 2F_3 + 3F_4 + F_5 + 2F_6 + 3F_7 + F_8 + 2E_1 + 3E_2) \equiv \frac{3}{2}H $$
$$ \frac{1}{2} (B + F_2+ 2F_3 + 3F_4 + F_5 + 2F_6 + 3F_7 + E_3 + E_4) \equiv H$$
$$ \frac{1}{2} (M + F_1 + E_1 + E_2 + E_3 + E_4) \equiv \frac{1}{2} H$$ and so $\sigma^* K_{X_{20}} \equiv \frac{1}{6} F + \frac{1}{6} M
+ \frac{1}{6} E_1 + \frac{1}{2} F_8 + \frac{1}{3} L_2$. Then
$K_{X_{20}}$ is nef and $K_{X_{20}}^2=0$.

\begin{figure}[htbp]
\includegraphics[width=10cm]{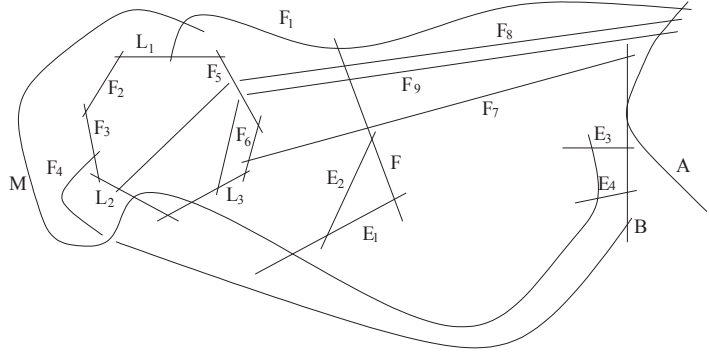}
\caption{The initial elliptic fibration with singular fibers $I_7,2I_1,III$ blown-up four times} \label{f9}
\end{figure}

We know that the
$\Q$-Gorenstein smoothing of $X_{20}$ is simply connected with
$p_g=0$, and it has a $(-4)$ and a $(-5)$ smooth rational curves $C_4$ and $C_5$ respectively, coming from the singularity at $P_2$. If $X_s$ is the general fiber of the $\Q$-Gorenstein smoothing of $X_{20}$, then $K_{X_s}\equiv (1 -\frac{1}{a} - \frac{1}{b}) G$, where $G$ is the general fiber of the elliptic fibration of $X_s$, and $a,b$ are the multiplicities of the two multiple fibers. In this way, we find $a=2$ and $b=3$ using the equations $K_{X_s} \cdot C_4=2$ and $K_{X_s} \cdot C_5 =3$. Therefore $X_s$ is a Dolgachev surface of type $2,3$ \cite{BHPV04}, p.383.

\begin{proposition}
The general fiber of  $(X_2 \subset \X_2) \to (0 \in \D)$ is a Dolgachev surface of type $2,3$ which has a chain of four $\P^1$'s with consecutive self-intersections $(-4)$, $(-5)$, $(-2)$, $(-2)$. In this way the singularity $P_2 \in X$ labels a boundary curve in $\overline{\cM}_{K^2,\chi}$ which generically parametrizes Dolgachev surfaces of type $2,3$ containing the exceptional divisor of $\frac{1}{7^2}(1,7 \cdot 5 -1)$. \label{p5}
\end{proposition}


\vspace{0.3cm}

{\tiny Department of Mathematics and Statistics,

University of Massachusetts,

Amherst, MA, USA.}

\vspace{0.3cm}

{\tiny Facultad de Matem\'aticas,

Pontificia Universidad
Cat\'olica de Chile,

Santiago, Chile.}

\end{document}